\documentclass{amsart}
\usepackage{times}
\usepackage{amssymb}
\usepackage{amsmath}
\usepackage{lscape}

\usepackage[pdftex,colorlinks,backref,pagebackref,hypertexnames=false,
            linkcolor=blue,urlcolor=blue,citecolor=blue]{hyperref}

\newcommand{\Z}{{\mathbb{Z}}}
\newcommand{\Q}{{\mathbb{Q}}}

\newcommand{\F}{{\mathbb{F}}}
\newcommand{\G}{{\mathbf{G}}}

\newcommand{\B}{{\mathbf{B}}}

\newcommand{\PP}{{\mathbf{P}}}

\newcommand{\Irr}{\text{Irr}}

\newcommand{\End}{\text{End}}

\newcommand{\ind}{\text{ind}}
\newcommand{\sgn}{\text{sgn}}
\newcommand{\SL}{\operatorname{SL}}
\newcommand{\GL}{\operatorname{GL}}
\newcommand{\Sz}{\operatorname{Sz}}

\renewcommand{\epsilon}{\varepsilon}

\newtheorem{theorem}{Theorem}[section]

\newtheorem{proposition}[theorem]{Proposition}
\newtheorem{corollary}[theorem]{Corollary}

\begin{document}

\title[On the Decomposition Numbers of ${^2F_4}(q^2)$]{On the
    Decomposition Numbers of the Ree Groups $\mathbf{{^2F_4}(q^2)}$\\
    in Non-Defining Characteristic} 

\author{Frank Himstedt}

\address{Technische Universit\"at M\"unchen, Zentrum Mathematik --
         SB-S-MA, Boltzmannstr. 3, 85748 Garching, Germany}

\email{himstedt@ma.tum.de}



\begin{abstract}
We compute the $\ell$-modular decomposition matrices of the simple Ree
groups ${^2F_4}(q^2)$, where $q^2=2^{2n+1}$ and $n$ is a positive
integer, for all primes $\ell > 3$ up to some entries in the unipotent
characters. Using these matrices we determine the smallest degree of a
non-trivial irreducible $\ell$-modular representation of
${^2F_4}(q^2)$ for all primes $\ell > 3$. We also obtain results on
the $3$-modular decomposition matrices of ${^2F_4}(q^2)$. 
\end{abstract}

\maketitle

\section{Introduction}

This paper deals with the modular representation theory of the simple
Ree groups ${^2F_4(q^2)}$ where $q^2 = 2^{2n+1}$ and $n$ is a positive
integer. Very valuable information on the modular representations
of these groups in non-defining characteristic was obtained 
by~G.~Hi{\ss}. In~\cite{HissHabil} and \cite{HissTrees2F4}, he
determined the Brauer trees of all blocks of ${^2F}_4(q^2)$
with a cyclic defect group. Furthermore, he computed the
$\ell$-modular decomposition numbers of the non-simple group
${^2F_4(2)}$ and of its derived subgroup, the simple Tits group
${^2F_4(2)}'$, for all odd primes~$\ell$, see~\cite{HissDecompTits}.

In this article, we consider the $\ell$-modular representation theory 
of the simple Ree\linebreak groups ${^2F_4(q^2)}$, $q^2 = 2^{2n+1}$, $n>0$, for
all odd primes $\ell$ such that the Sylow $\ell$-subgroups of
${^2F_4(q^2)}$ are not cyclic. For all such primes $\ell>3$, we
compute the $\ell$-modular decomposition matrices of
${^2F_4(q^2)}$ up to several entries in the unipotent characters and
one entry in the non-unipotent characters.
For fixed $q$, this determines all irreducible $\ell$-modular Brauer
characters of ${^2F_4(q^2)}$ up to at most four. As a
corollary, we show that for all $\ell > 3$ the $\ell$-modular
decomposition matrix of ${^2F_4(q^2)}$ has a lower unitriangular shape
and that the reductions modulo $\ell$ of all ordinary cuspidal
unipotent characters of ${^2F_4(q^2)}$ are irreducible Brauer
characters. This was conjectured by G.~Hi{\ss} and M.~Geck in a more
general context, see \cite[Conjecture~3.4]{GeckHiss}. 

Most of the methods we use to determine the decomposition
numbers are elementary in the sense that they only require calculations
with characters. The main ingredients are the character table of
${^2F_4}(q^2)$ which was computed by G.~Malle~\cite{MalleUni2F4} and
is contained in the CHEVIE library~\cite{CHEVIE}, and the character
tables of the proper parabolic subgroups of ${^2F_4(q^2)}$ which were
determined in~\cite{HimstedtHuang2F4MaxParab},
\cite{HimstedtHuang2F4Borel} and which are also available as CHEVIE
files. 
We produce projective characters of ${^2F_4(q^2)}$ by
inducing projective characters of the proper parabolic subgroups and
by tensoring defect $0$ characters with ordinary characters. Then, by
computing scalar products of these projectives with ordinary
irreducible characters, we obtain an approximation
to the decomposition matrix which already implies the lower
unitriangular shape of these matrices. To determine the decomposition
numbers below the diagonal we use relations which are obtained by
expressing Brauer characters as linear combinations of the elements of
certain basic sets of Brauer characters. 

Additionally, we use some less elementary tools like Hecke
algebra methods, arguments from modular Harish-Chandra theory and a
theorem of M.~Geck, G.~Hi{\ss} and G.~Malle on the cuspidality of the
modular Steinberg character. For the non-unipotent blocks we use 
C.~Bonnaf{\'e}'s and R.~Rouquier's modular version of the Jordan
decomposition of characters. All calculations concerning scalar products 
and induction and restriction of characters are carried out using CHEVIE. 

The situation is significantly different for $\ell=3$, since $3$ is a bad
prime for a root system of type $F_4$ and the Sylow $3$-subgroups of
${^2F_4(q^2)}$ are non-abelian. However, our methods are still good
enough to compute the $3$-modular decomposition matrices of
${^2F_4(q^2)}$ except for several entries in six rows corresponding to
ordinary unipotent irreducible characters. In particular, we
are able to show that the $3$-modular decomposition matrices of
${^2F_4(q^2)}$ have a lower unitriangular shape, too.

The $\ell$-modular decomposition numbers for $\ell \ge 3$ which we
are not able to determine are multiplicities of certain cuspidal Brauer
characters in the reduction modulo $\ell$ of ordinary unipotent
characters of large degree, including the Steinberg character. 
For these unknown multiplicities we only get upper bounds depending 
on $q$.  
The problems in the determination of these multiplicities seem
to be a usual phenomenon in the calculation of the decomposition
matrices of finite groups of Lie type in non-defining 
characteristic; see for example~\cite{GeckDissPub} and \cite{HissG2}. 
Maybe module theoretic arguments as in \cite{Waki_G2} might eventually
lead to more information on these numbers. 

As an application, we obtain new information on the degrees of modular 
irreducible representations of ${^2F_4}(q^2)$. Let $G = {^2F_4}(q^2)$, 
$q^2 = 2^{2n+1}$, $n>0$ and $k$ be an algebraically closed field 
of characteristic $\ell \ge 0$. We denote the smallest degree of a
nontrivial irreducible $kG$-representation by $d_\ell(G)$. 
Using the Brauer trees in~\cite{HissTrees2F4},
F.~L\"ubeck~\cite{LuebeckSmlDeg} was able to determine $d_\ell(G)$ for
all primes $\ell$ such that the Sylow $\ell$-subgroups of $G$ are
cyclic. In fact, for such $\ell$ he proved $d_\ell(G) = d_0(G)$ where
the latter is known due to G.~Malle's ordinary character table of $G$.
Using our $\ell$-modular decomposition matrices, we can extend
F.~L\"ubeck's result to all primes $\ell > 3$: We show that 
$d_\ell(G)=d_0(G)$ holds for all primes $\ell > 3$.  
The idea of the proof is similar to~\cite{Himstedt3D4Decomp}: The
decomposition matrices determine almost all irreducible Brauer
characters of $G$ and in particular their degrees. In most cases
(depending on $q$ and~$\ell$) only three or four degrees of
irreducible Brauer characters remain unknown. From the upper bounds
for the decomposition numbers, we can derive that these unknown
degrees are larger than $d_0(G)$. 
This last step turns out to be surprisingly difficult because
there are some ordinary unipotent irreducible characters
of $G$ whose degrees do not differ ``very much'' from the degree of the
Steinberg character; for more details see Section~\ref{sec:degrees}. 

This paper is organized as follows: In Section~\ref{sec:nota}, we
introduce notation for characters, Lusztig series and blocks. In
Sections~\ref{sec:decompmat} and \ref{sec:decompmatl3}, we state our
main results on the decomposition matrices of ${^2F_4}(q^2)$, which
are then proved in Section~\ref{sec:proofs}. In
Section~\ref{sec:degrees}, we describe the consequences for the
degrees of modular representations of ${^2F_4}(q^2)$ in odd
characteristics. Scalar products, relations and decomposition matrices
are given in several Appendices.


\section{Notation and Setup} 
\label{sec:nota}

In this section, we introduce the general setup and notation which
will be used throughout this paper.

\subsection{Group theoretical setup}
\label{grpsetup}

We use the notation and setup from
\cite[Section~2]{HimstedtHuang2F4Borel}~and 
\cite[Sections~3 and~4]{HimstedtHuang2F4MaxParab}. In particular, 
$n>0$ is an integer and $q=\sqrt{2^{2n+1}}$. Let $\Phi$ be the root 
system of type $F_4$ described in
\cite[Section~2]{HimstedtHuang2F4Borel}, so $\Phi$ has simple roots
$r_1,r_2,r_3,r_4$ and Dynkin diagram: 
\setlength{\unitlength}{0.9mm}
\begin{center}
\begin{picture}(85,13)
\thinlines
\put(5,5){\circle*{1.5}}
\put(30,5){\circle*{1.5}}
\put(55,5){\circle*{1.5}}
\put(80,5){\circle*{1.5}}
\put(7,5){\line(1,0){21}}
\put(32,5.5){\line(1,0){21}}
\put(32,4.5){\line(1,0){21}}
\put(57,5){\line(1,0){21}}
\put(41,7){\line(2,-1){4}}
\put(41,3){\line(2,1){4}}
\put(4,8){$r_1$}
\put(29,8){$r_2$}
\put(54,8){$r_3$}
\put(79,8){$r_4$}
\end{picture}
\end{center}
Fix a simply connected linear algebraic group $\G$ defined over an
algebraically closed field of characteristic $2$ with root system
$\Phi$ and a Frobenius morphism $F$ such that $G:=\G^F$ is the Ree
group ${^2F}_4(q^2)$ defined over the field $\F_{q^2}$. Let $\B$ be an
$F$-stable Borel subgroup and $\PP_a$ and $\PP_b$ be $F$-stable
maximal parabolic subgroups of $\G$ as 
in \cite{HimstedtHuang2F4MaxParab} and
\cite{HimstedtHuang2F4Borel}. The finite group $G={^2F}_4(q^2)$ has
the order 
\[
|G|=q^{24} \, \phi_1^2 \, \phi_2^2 \, \phi_4^2 \, \phi_8'^2 \,
\phi_8''^2 \, \phi_{12} \, \phi_{24}' \, \phi_{24}'',
\]
where $\phi_1=q-1$, $\phi_2=q+1$, $\phi_4=q^2+1$,
$\phi_8'=q^2+\sqrt{2}q+1$, $\phi_8''=q^2-\sqrt{2}q+1$,
$\phi_{12}=q^4-q^2+1$, $\phi_{24}'=q^4+\sqrt{2}q^3+q^2+\sqrt{2}q+1$,
$\phi_{24}''=q^4-\sqrt{2}q^3+q^2-\sqrt{2}q+1$.
Furthermore, we set $\phi_8=\phi_8' \phi_8''$ and 
$\phi_{24}=\phi_{24}'\phi_{24}''$.
The Borel subgroup $B:=\B^F$ and the maximal parabolic subgroups
$P_a:=\PP_a^F$ and $P_b:=\PP_b^F$ of $G$ containing $B$ have the orders
\[
|B| = q^{24} \, \phi_1^2 \, \phi_2^2, \quad |P_a| = q^{24} \, \phi_1^2 \,
\phi_2^2 \, \phi_4, \quad |P_b| = q^{24} \phi_1^2 \, \phi_2^2 \,
\phi_8' \, \phi_8''. 
\]
As in~\cite{HimstedtHuang2F4Borel} and
\cite{HimstedtHuang2F4MaxParab}, we write $T$, $L_a$ and $L_b$ for the
Levi subgroups and $U$, $U_a$ and $U_b$ for the unipotent radicals 
of $B$, $P_a$ and $P_b$, respectively. So $T$ is a
maximally split torus of $G$, and the Levi subgroups $L_a$, $L_b$ have
the structure $L_a \cong \Z_{q^2-1}\times\SL_2(q^2)\cong\GL_2(q^2)$ and
$L_b \cong \Z_{q^2-1}\times\Sz(q^2)$, respectively. Representatives for the
conjugacy classes of $G$, $B$, $P_a$ and~$P_b$ are given in
Appendix~A of~\cite{HimstedtHuang2F4MaxParab}
and~\cite{HimstedtHuang2F4Borel} , together with the 
corresponding class fusions.

\subsection{Ordinary and modular representations}
\label{ordmod}

In this whole article, $\ell$ is an odd prime number and
$(K,R,k)$ a splitting $\ell$-modular system for all subgroups of
$G={^2F_4}(q^2)$. 
The word ``character'' will always mean an ordinary character
associated with a representation over $K$. For subgroups $H$ of $G$,
we write $\Irr(H)$ for the set of ordinary irreducible characters of $H$. 
All Brauer characters, blocks, decomposition numbers will be
taken with respect to the fixed prime number $\ell$. 
The restriction of any class function $\vartheta$ of $G$ to the set 
of $\ell$-regular elements will be denoted by $\breve\vartheta$.
In particular, if~$\chi$ is an ordinary character, then~$\breve\chi$
is a Brauer character. Sometimes, it will be convenient to extend
such functions by zero to the $\ell$-singular elements.

\subsection{Characters of the Ree groups}
\label{char2F4}

The ordinary irreducible characters of $G={^2F}_4(q^2)$ were computed
by G.~Malle and are contained in the CHEVIE library~\cite{CHEVIE}. A
construction of the unipotent irreducible characters of $G$ is described
in~\cite{MalleUni2F4}. There are $21$ unipotent irreducible
characters of $G$ and we denote them by $\chi_{1}$, $\chi_2$, \dots,
$\chi_{21}$. This notation coincides with the numbering of the
irreducible characters of $G$ in CHEVIE. 
In Table~\ref{tab:unicharsG} in the Appendix, we collect 
some information on the unipotent characters including a dictionary  
between our notation and the notation in \cite{Carter2},
\cite{HissTrees2F4} and \cite{MalleUni2F4}.

The character $\chi_1$ is the trivial character, $\chi_{21}$ is the
Steinberg character. The characters in the principal series are
$\chi_1$, $\chi_4$, $\chi_5$, $\chi_6$, $\chi_7$, $\chi_{18}$ and $\chi_{21}$.
The characters $\chi_2$ and $\chi_{19}$ are constituents of the
Harish-Chandra induction of one of the cuspidal unipotent characters
of the Levi subgroup $\Z_{q^2-1}\times\Sz(q^2)$ of $P_b$. The
characters $\chi_3$ and $\chi_{20}$ are constituents of the
Harish-Chandra induction of the other cuspidal unipotent character
of this Levi subgroup. 
The remaining unipotent characters of $G$ are cuspidal.
The characters $\chi_2$ and $\chi_3$ are complex conjugate to each
other, and the same holds for the pairs 
$(\chi_{11}, \chi_{12})$, $(\chi_{13}, \chi_{14})$, 
$(\chi_{15}, \chi_{16})$, $(\chi_{19}, \chi_{20})$. 
Since not all of these unipotent irreducible characters of $G$ are
uniquely determined by their degree we also provide some character
values in the last two columns of Table \ref{tab:unicharsG}. The
$\epsilon_4$ is a complex fourth root of unity,  
see \cite[Table 5]{HimstedtHuang2F4Borel}.

The set of unipotent irreducible characters of $G$ can be partitioned
into certain subsets, called \textit{families}; see
\cite[Section~12.3]{Carter2}. The families of unipotent characters of
$G={^2F_4}(q^2)$ were determined by G.~Lusztig in
\cite[Section~14.2]{LusztigReductGrp}. There are $7$ families: four of 
them have $1$ element, two have $2$ elements and one has $13$
elements. The distribution of the unipotent irreducible characters of
$G$ into families is indicated by the horizontal lines in
Table~\ref{tab:unicharsG}. 

Our notation for the non-unipotent irreducible characters of
$G$, is motivated by the Jordan decomposition of characters. Let
$(\G^*, F^*)$ be dual to $(\G,F)$ and let $G^*:=\G^{*F^*}$. For every
semisimple element $s \in G^*$, there is a set
$\mathcal{E}(G,s)\subseteq\Irr(G)$, called the \textit{Lusztig series}
associated with $s$, and  
\[
\Irr(G) = \bigcup_{s}^{.} \mathcal{E}(G, s),
\]
where $s$ runs over a set of representatives for the semisimple conjugacy
classes of~$G^*$, is a partition of $\Irr(G)$. For every
semisimple element $s \in G^*$, there is a bijection bet\-ween
$\mathcal{E}(G,s)$ and the set of unipotent irreducible characters of
the centralizer $C_{G^*}(s)$. There are $18$~types   
$g_1$, $g_2$, \dots, $g_{18}$ of Lusztig series of $G$, and the Dynkin
type of the corresponding centralizers is described in 
\cite[Table~B.9]{HimstedtHuangDade2F4}. We write
$\chi=\chi_{i,\lambda}$ for the irreducible character in the Lusztig
series of type $g_i$ of $G$ corresponding to the unipotent character
$\lambda\in\Irr(C_{G^*}(s))$. The degree of $\chi_{i,\lambda}$
is
$
\chi_{i,\lambda}(1)=\lambda(1) \cdot [G^*:C_{G^*}(s)]_{2'},
$
where $\mathcal{E}(G,s)$ is of type $g_i$. Details on the
non-unipotent irreducible characters of $G$ are given in
Table~\ref{tab:nonuniG} in the Appendix.

\subsection{Characters of the parabolic subgroups}
\label{charparab}

The notation we use for the irreducible characters of the proper
parabolic subgroups $B$, $P_a$ and $P_b$, is the same as in 
\cite[Table~A.5, A.6]{HimstedtHuang2F4Borel} and 
\cite[Tables~A.4, A.5, A.8, A.9]{HimstedtHuang2F4MaxParab}.
That is, we denote the irreducible characters of $B$ by 
${_B\chi}_1$, \dots, ${_B\chi}_{58}$, the irreducible characters 
of $P_a$ by ${_{P_a}\chi}_1$, \dots, ${_{P_a}\chi}_{40}$ and  
the irreducible characters of $P_b$ by ${_{P_b}\chi}_1$, \dots,
${_{P_b}\chi}_{56}$ (maybe depending on some parameters $k$ or $l$).

\subsection{Induction and restriction}
\label{scalindres}

Let $H$ be a finite group and $H_1$ a subgroup of $H$. If $\chi$ is a 
character of $H_1$, we write $\chi^H$ for the induced character; and
if $\chi$ is a character of~$H$, we write~$\chi_{H_1}$ for the
restriction of $\chi$ to $H_1$. 
Using the class fusions in \cite[Table~A.4]{HimstedtHuang2F4Borel}, 
\cite[Tables~A.1, A.3, A.7]{HimstedtHuang2F4MaxParab} and CHEVIE, we
can easily compute the induction and restriction of characters between
the subgroups~$G$, $B$, $P_a$ and $P_b$. 

\subsection{Blocks and basic sets}
\label{blocksbas}

The distribution of ordinary irreducible characters of $G$
into blocks is compatible with the Lusztig series in the following
sense: For every semisimple element $s \in G^*$ of order prime to
$\ell$, the set  
\[
\mathcal{E}_\ell(G, s) := \bigcup_{t \in C_{G^*}(s)_\ell}
\mathcal{E}(G, st) \subseteq \Irr(G), 
\]
where $C_{G^*}(s)_\ell$ is the set of elements of $\ell$-power order
of the centralizer $C_{G^*}(s)$, is a union of
blocks. For the Ree groups, this is shown in
\cite[p.~113]{HissHabil}, see the remarks after the proof of
Lemma~D.3.4. Note that a similar compatibility between blocks and
geometric conjugacy classes holds in a much more general context, 
see for example~\cite[Theorem~2.2]{BroueMichelBlocks}
or~\cite[Theorem~9.12]{CabEng}. 
The blocks in $\mathcal{E}_\ell(G, 1)$ are called the unipotent blocks
of $G$. These are the blocks of $G$ containing at least one
unipotent character. 

Let $\mathcal{B}$ be a union of blocks of $G$. A basic set in
$\mathcal{B}$ is a set of linearly independent Brauer characters in
$\mathcal{B}$ such that every Brauer character in $\mathcal{B}$ is a
linear combination with integer coefficients of the elements in this
set. A basic set is called ordinary, if it consists of the restrictions
of some ordinary characters to the $\ell$-regular elements of $G$.
In this case, we identify these ordinary characters with the elements
of the basic set.
So, in order to describe the decomposition matrix of
$\mathcal{B}$, it is enough to describe the decomposition numbers of
the characters in an ordinary basic set and the relations expressing
the restrictions of the remaining ordinary characters in
$\mathcal{B}$ to the $\ell$-regular elements as linear combinations of
the characters in the basic set. For more details on blocks and basic
sets see~\cite[Section 3]{GeckHiss}.


\section{\texorpdfstring{Decomposition Matrices for
    $\ell>3$}{Decomposition Matrices for l>3}}
\label{sec:decompmat}

As in Section~\ref{sec:nota} let $G={^2F_4}(q^2)$, $q^2=2^{2n+1}$,
$n>0$, and let $\ell$ be an odd prime number. In \cite{HissHabil},
\cite{HissTrees2F4}, G.~Hi{\ss} determined the Brauer trees of all
blocks of $G$ with cyclic defect group. In particular, if the Sylow
$\ell$-subgroups of $G$ are cyclic, then the $\ell$-modular
decomposition numbers of~$G$ can be read off from these Brauer
trees. In this section, we describe the $\ell$-modular decomposition
numbers of $G$ for all remaining primes $\ell>3$.  
Some of our methods and the presentation are inspired
by~\cite{GeckHiss} and \cite{KoehlerDiss}. 

We fix a prime number $\ell>3$ and an $\ell$-modular splitting system
$(K,R,k)$ for all subgroups of $G$ as in Subsection~\ref{ordmod}. All
references to Brauer characters, blocks, decomposition numbers etc. in
this section will refer to this fixed prime $\ell$.
Of course, we get non-trivial decomposition matrices only if $\ell$
divides the group order. From now on, we assume that $\ell$ divides
\[
|G|=q^{24} \, \widetilde{\phi}_1^2 \, \phi_4^2 \, \phi_8'^2 \,
\phi_8''^2 \, \phi_{12} \, \phi_{24}' \, \phi_{24}'',
\]
where $\widetilde{\phi}_1:=\phi_1 \phi_2 = q^2-1$. The condition
$\ell>3$ implies that $\ell$ divides exactly one of the factors
$\widetilde{\phi}_1$, $\phi_4$, $\phi_8'$, $\phi_8''$, $\phi_{12}$,
$\phi_{24}'$, $\phi_{24}''$. So, we are in the generic
situation studied in~\cite{BMGenSylow},~\cite{BMMGenBlocks}. Note
that $\widetilde{\phi}_1$, $\phi_4$, $\phi_8'$, $\phi_8''$,
$\phi_{12}$, $\phi_{24}'$, $\phi_{24}''$ correspond to 
$(t2)$-cyclotomic polynomials in the sense of~\cite[3F]{BMGenSylow}. 
If $\ell$ divides $\phi_{12}$, $\phi_{24}'$ or $\phi_{24}''$, then the
Sylow $\ell$-subgroups of $G$ are cyclic and the decomposition numbers
can be read off from the Brauer trees in~\cite{HissHabil},
\cite{HissTrees2F4}. So, we only have to consider the cases 
\begin{equation} \label{eq:4cases}
\ell \, | \, q^2-1, \quad \ell \, | \, q^2+1,  \quad
\ell \, | \, q^2+\sqrt{2}q+1 , \quad \ell \, | \, q^2-\sqrt{2}q+1. 
\end{equation}

For every semisimple $\ell'$-element $s \in G^*$, the set 
$\mathcal{E}_\ell(G, s)$ is a union of blocks of $G$, see
Subsection~\ref{blocksbas}. The assumption $\ell>3$ implies that
$\ell$ is a good prime for $\G$ in the sense
of~\cite[p.~125]{DigneMichel}. Therefore, we can deduce from a 
general result of M.~Geck and G.~Hi{\ss}~\cite{GeckHissBasicSets}
that $\mathcal{E}(G,s)$ is an ordinary basic set for
$\mathcal{E}_\ell(G,s)$. 

For all primes $\ell>3$ satisfying one of the
conditions~(\ref{eq:4cases}), we are going to describe the
decomposition numbers of all ordinary irreducible characters in the
unipotent blocks $\mathcal{E}_\ell(G, 1)$. 
For the non-unipotent blocks, we are only going to describe the
decomposition numbers of the irreducible characters in the ordinary
basic sets $\mathcal{E}(G, s)$, $s \neq 1$. In particular, this
determines all non-unipotent irreducible Brauer characters of $G$ (up
to a single exception in the case $\ell \, | \, q^2+1$, see
Subsections~\ref{case2} and \ref{sub:rmkdecmat}). All of the statements in
the following subsections will be proved in Section~\ref{sec:proofs}. 
Note that in the decomposition matrices and tables of scalar products
in the Appendices~B, C and D, zeros are replaced by dots. 

\subsection{\texorpdfstring{The case $\ell \, | \, q^2-1$}{The case l
   divides q2-1}}
\label{case1}

We begin with some comments on the decomposition matrices in
Appendices~C and D: The decomposition numbers of the unipotent blocks
$\mathcal{E}_\ell(G,1)$ are given in Table~\ref{tab:decuni1}. The
first column of this table contains notation for the ordinary
irreducible characters in $\mathcal{E}_\ell(G,1)$. The $21$
irreducible Brauer characters  in the unipotent blocks are denoted by
$\phi_1, \dots, \phi_{21}$. 

Some information on the distribution of $\phi_1, \dots, \phi_{21}$
into modular Harish-Chandra series is presented in
the second row of Table~\ref{tab:decuni1}; for a definition of modular
Harish-Chandra series see~\cite[Section~2]{GeckHiss}.
There are $13$ such series,
corresponding to the Levi subgroups $T$, $L_b$ and $G$. The Levi
subgroup $T$ has a unique cuspidal unipotent Brauer character, and the
Levi subgroup $L_b$ has exactly two cuspidal unipotent Brauer
characters. The corresponding modular Harish-Chandra series of $G$
contain, respectively, $7$, $2$, $2$ irreducible Brauer characters. 
Additionally, $G$ has $10$ cuspidal unipotent Brauer characters. 

The columns labeled by $ps$ correspond to the irreducible Brauer characters in
the principal series. The Levi subgroup $L_b$ has two cuspidal
unipotent Brauer characters $\varphi_a$ and~$\varphi_b$; these are
the restrictions of the ordinary irreducible characters of degree
$\frac{q}{\sqrt{2}}(q^2-1)$ to the set of $\ell$-regular elements. 
The columns labeled by ${^2B_2}[a]$ and ${^2B_2}[b]$ correspond to the
irreducible Brauer characters in the modular Harish-Chandra series of
$G$ associated with $\varphi_a$, $\varphi_b$, respectively.
The remaining columns (labeled by ``$c$'') correspond to the cuspidal
unipotent Brauer characters. 

The decomposition numbers of the non-unipotent irreducible characters
are given in Tables~\ref{tab:dec2}, \ref{tab:dec3},
\ref{tab:decg5cyc}-\ref{tab:decnilp} in the Appendix. In these tables,
only the decomposition numbers of the ordinary basic sets
$\mathcal{E}(G, s)$, $s \neq 1$, are listed. The left most column of
the tables contains notation for the ordinary irreducible characters
in $\mathcal{E}(G,s)$. The second row of the tables contains notation
for the irreducible Brauer characters in the corresponding
blocks. Each of the blocks described by Table~\ref{tab:decnilp} has
only one irreducible Brauer character. 

\begin{theorem} \label{thm:decnum1}
Let $\ell$ be a prime dividing $q^2-1$. The $\ell$-modular
decomposition numbers of $G={^2F_4}(2^{2n+1})$, $n>0$, are given by
Tables~\ref{tab:decuni1} and \ref{tab:dec2}, \ref{tab:dec3},
\ref{tab:decg5cyc}-\ref{tab:decnilp} in the Appendix.
\end{theorem}

\subsection{\texorpdfstring{The case $3 \neq \ell \, | \, q^2+1$}{The
   case l not 3, l divides q2+1}} 
\label{case2}

The decomposition numbers of the unipotent blocks
$\mathcal{E}_\ell(G,1)$ are given in Table~\ref{tab:decuni2}. 
The $21$ irreducible Brauer characters $\phi_1, \dots,
\phi_{21}$ in the unipotent blocks are distributed into
$16$ Harish-Chandra series.
The Levi subgroups $T$ and $L_a$ each have a unique cuspidal unipotent
Brauer character, and the Levi subgroup $L_b$ has exactly two cuspidal
unipotent Brauer characters. The corresponding modular Harish-Chandra
series of $G$ contain, respectively, $4$, $1$, $2$, $2$ irreducible
Brauer characters. Additionally, $G$ has $12$ cuspidal unipotent
Brauer characters.  

Again, the principal series is abbreviated by $ps$.
The Levi subgroup~$L_a$ has a unique cuspidal unipotent Brauer
character (the modular Steinberg character, see~\cite{GHMCusp}) and
we denote the corresponding modular Harish-Chandra series of $G$ by
$A_1$. The Levi subgroup~$L_b$ has two cuspidal unipotent Brauer
characters: the restrictions of the ordinary irreducible
characters of degree $\frac{q}{\sqrt{2}}(q^2-1)$ to the $\ell$-regular
elements. The corresponding modular Harish-Chandra series are denoted
by ${^2B_2}[a]$ and ${^2B_2}[b]$, respectively.  
The remaining columns correspond to cuspidal unipotent Brauer characters. 

The decomposition numbers of the non-unipotent irreducible characters
are given in Tables~\ref{tab:dec2}-\ref{tab:decg5noncyc} and 
\ref{tab:dec6}-\ref{tab:decnilp}. The decomposition numbers of all
ordinary characters in $\mathcal{E}_\ell(G, s)$, where $s \in G^*$ is
semisimple of type $g_5$, are given in Table~\ref{tab:decg5noncyc}. 
For all other non-unipotent blocks, only the decomposition numbers of
the ordinary basic sets are given.

\begin{theorem} \label{thm:decnum2}
Let $\ell>3$ be a prime dividing $q^2+1$. The $\ell$-modular
decomposition numbers of $G={^2F_4}(2^{2n+1})$, $n>0$, are given by
Tables~\ref{tab:decuni2} and \ref{tab:dec2}-\ref{tab:decg5noncyc},
\ref{tab:dec6}-\ref{tab:decnilp} in the Appendix. There are the
following bounds: 
\begin{enumerate}
\item[(i)] \quad $2 \le a \le \frac{q^2-2}{3}$.

\smallskip

\item[(ii)] \quad $1 \le b \le \frac{q^2+\sqrt{2}q}{4}$. If 
$\ell \neq 11$ or $n \equiv 27$ mod $55$, then $b \ge 2$.

\smallskip

\item[(iii)] \quad $1 \le c \le \frac{q^2-\sqrt{2}q}{4}$. If 
$\ell \neq 11$ or $n \equiv 27$ mod $55$, then $c \ge 2$.

\smallskip

\item[(iv)] \quad $2 \le d \le \frac{q^4+2}{3}$.

\smallskip

\item[(v)] \quad $2 \le e \le \frac{q^2+2}{2}$. If 
$\ell \neq 11$ or $n \equiv 27$ mod $55$, then $e \ge 3$.

\smallskip

\item[(vi)] \quad $2 \le a' \le \frac{q^2-\sqrt{2}q}{4}$.
\end{enumerate}
\end{theorem}

\subsection{\texorpdfstring{The case $\ell \, | \,
    q^2+\sqrt{2}q+1$}{The case l divides q2+sqrt{2}q+1}}
\label{case3}

The decomposition numbers of the unipotent blocks
$\mathcal{E}_\ell(G,1)$ are given in Table~\ref{tab:decuni3}. 
The $21$ irreducible Brauer characters $\phi_1, \dots,
\phi_{21}$ in the unipotent blocks are distributed into
$18$ Harish-Chandra series.
The Levi subgroup $T$ has a unique cuspidal unipotent
Brauer character, and the Levi subgroup $L_b$ has exactly three
cuspidal unipotent Brauer characters. The corresponding modular
Harish-Chandra series of $G$ contain, respectively, $4$, $1$,
$1$, $1$ irreducible Brauer characters. Additionally, $G$ has $14$
cuspidal unipotent Brauer characters.  

The columns labeled by $ps$ correspond to the irreducible Brauer characters in
the principal series. The Levi subgroup $L_b$ has two cuspidal unipotent Brauer
characters $\varphi_a$, $\varphi_b$, the restrictions of the
ordinary irreducible characters of degree $\frac{q}{\sqrt{2}}(q^2-1)$
to the set of $\ell$-regular elements. Additionally, the modular Steinberg
character $\varphi_{St}$ of $L_b$ is cuspidal, see~\cite[Theorem~4.2]{GHMCusp}.
We denote the corresponding modular Harish-Chandra series of $G$ by
${^2B_2}[a]$, ${^2B_2}[b]$ and ${^2B_2}[\text{St}]$, respectively. The
remaining columns correspond to the cuspidal unipotent Brauer characters.  
The decomposition numbers of the non-unipotent irreducible characters
are given in Tables~\ref{tab:dec2}, \ref{tab:dec3},
\ref{tab:decg5cyc}-\ref{tab:decnilp} in the Appendix.
In these tables, only the decomposition numbers of the ordinary basic
sets $\mathcal{E}(G, s)$, $s \neq 1$, are listed.

\begin{theorem} \label{thm:decnum3}
Let $\ell$ be a prime dividing $q^2+\sqrt{2}q+1$. The $\ell$-modular
decomposition numbers of $G={^2F_4}(2^{2n+1})$, $n>0$, are given by
Tables~\ref{tab:decuni3} and \ref{tab:dec2}, \ref{tab:dec3},
\ref{tab:decg5cyc}-\ref{tab:decnilp} in the Appendix. There are the
following bounds: 
\begin{enumerate}
\item[(i)] \quad $0 \le a \le \frac{q^2+3\sqrt{2}q+4}{12}$.

\smallskip

\item[(ii)] \quad $0 \le b \le \frac{q^2+\sqrt{2}q}{4}$.

\smallskip

\item[(iii)] \quad $0 \le c \le \frac{q^2-2}{3}$.

\smallskip

\item[(iv)] \quad $0 \le d \le \frac{\sqrt{2}q(q^2-2)}{24}$.

\smallskip

\item[(v)] \quad $0 \le e \le \frac{\sqrt{2}q(q^2+2)}{8}$.

\smallskip

\item[(vi)] \quad $0 \le g \le \frac{\sqrt{2}q(q^2-2)}{8}$.

\smallskip

\item[(vii)] \quad $1 \le h \le \frac{\sqrt{2}q}{4}$.

\smallskip

\item[(viii)] \quad $0 \le i \le \frac{\sqrt{2}q(q^2+1)}{6}$.

\smallskip

\item[(ix)] \quad $1 \le j \le \frac{\sqrt{2}q}{4}$.

\smallskip

\item[(x)] \quad $0 \le r \le \frac{q^4-4}{12}$.

\smallskip

\item[(xi)] \quad $1 \le s \le \frac{q^2+\sqrt{2}q}{4}$. If 
$\ell \neq 5$ or $n \equiv 7$ or $n \equiv 12$ mod $20$, then $s \ge 2$.

\smallskip

\item[(xii)] \quad $0 \le t \le \frac{q^4}{4}$.

\smallskip

\item[(xiii)] \quad $1 \le u \le \frac{q^2+3\sqrt{2}q+4}{12}$.

\smallskip

\item[(xiv)] \quad $0 \le v \le \frac{q^4+2}{3}$.

\smallskip

\item[(xv)] \quad $1 \le w \le \frac{q^2+\sqrt{2}q+4}{4}$.

\smallskip

\item[(xvi)] \quad $1 \le x \le \frac{q}{\sqrt{2}}$. If 
$\ell \neq 5$ or $n \equiv 7$ or $n \equiv 12$ mod $20$, then $x \ge 2$.
\end{enumerate}
\end{theorem}

For small $q$, some of the lower and upper bounds in
Theorem~\ref{thm:decnum3} coincide. Thus, we obtain the following
obvious consequence:

\begin{corollary} \label{cor:decnum3q8}
Suppose $n=1$, that is $G = {^2F_4}(8)$, and 
$\ell \, | \, q^2+\sqrt{2}q+1$. Then $\ell = 13$ and for the
decomposition numbers in Theorem~\ref{thm:decnum3} we have
$h=j=1$ and $x=2$.
\end{corollary}

\subsection{\texorpdfstring{The case $\ell \, | \,
    q^2-\sqrt{2}q+1$}{The case l divides q2-sqrt2q+1}}
\label{case4}

This case is similar to the case $\ell \, | \, q^2+\sqrt{2}q+1$.
The decomposition numbers of the unipotent blocks $\mathcal{E}_\ell(G,1)$
are given in Table~\ref{tab:decuni4}.  
The irreducible Brauer characters $\phi_1, \dots, \phi_{21}$ in the
unipotent blocks are distributed into $18$ Harish-Chandra series. 
The Levi subgroup $T$ has a unique cuspidal unipotent
Brauer character, and the Levi subgroup~$L_b$ has three
cuspidal unipotent Brauer characters. The corresponding modular
Harish-Chandra series of $G$ contain, respectively, $4$,
$1$, $1$, $1$ irreducible Brauer characters. Additionally, $G$ has
$14$ cuspidal unipotent Brauer characters.  

As usual, the columns labeled by $ps$ correspond to the irreducible
Brauer characters in the principal series. As in the case
$\ell \, | \, q^2+\sqrt{2}q+1$, we denote the modular Harish-Chandra
series of $G$ coming from $L_b$ by ${^2B_2}[a]$, ${^2B_2}[b]$ and
${^2B_2}[\text{St}]$. They belong to the cuspidal Brauer characters
$\varphi_a$, $\varphi_b$ of degree $\frac{q}{\sqrt{2}}(q^2-1)$ and the
modular Steinberg character $\varphi_{St}$ of~$L_b$, respectively. The
columns labeled by ``c'' correspond to the cuspidal unipotent Brauer
characters.  
The decomposition numbers of the non-unipotent irreducible characters
are given in Tables~\ref{tab:dec2}, \ref{tab:dec3},
\ref{tab:decg5cyc}-\ref{tab:decnilp} in the Appendix.

\begin{theorem} \label{thm:decnum4}
Let $\ell$ be a prime dividing $q^2-\sqrt{2}q+1$. The $\ell$-modular
decomposition numbers of $G={^2F_4}(2^{2n+1})$, $n>0$, are given by
Tables~\ref{tab:decuni4} and \ref{tab:dec2}, \ref{tab:dec3},
\ref{tab:decg5cyc}-\ref{tab:decnilp} in the
Appendix. There are the following bounds:
\begin{enumerate}
\item[(i)] \quad $0 \le a \le \frac{q^2-3\sqrt{2}q+4}{12}$.

\smallskip

\item[(ii)] \quad $0 \le b \le \frac{q^2-\sqrt{2}q}{4}$.

\smallskip

\item[(iii)] \quad $0 \le c \le \frac{q^2-2}{3}$.

\smallskip

\item[(iv)] \quad $0 \le d \le \frac{\sqrt{2}q(q^2-2)}{24}$.

\smallskip

\item[(v)] \quad $0 \le e \le \frac{\sqrt{2}q-4}{4}$.

\smallskip

\item[(vi)] \quad $0 \le g \le \frac{\sqrt{2}q(q^2+2)}{8}$.

\smallskip

\item[(vii)] \quad $0 \le h \le \frac{\sqrt{2}q(q^2-2)}{8}$.

\smallskip

\item[(viii)] \quad $0 \le i \le \frac{\sqrt{2}q(q^2+1)}{6}$.

\smallskip

\item[(ix)] \quad $0 \le j \le \frac{\sqrt{2}q-4}{4}$.

\smallskip

\item[(x)] \quad $0 \le r \le \frac{q^4-4}{12}$.

\smallskip

\item[(xi)] \quad $1 \le s \le \frac{q^2-\sqrt{2}q}{4}$. If 
$\ell \neq 5$ or $n \equiv 2$ or $n \equiv 17$ mod $20$, then $s \ge 2$.

\smallskip

\item[(xii)] \quad $0 \le t \le \frac{q^2-3\sqrt{2}q+4}{12}$.

\smallskip

\item[(xiii)] \quad $0 \le u \le \frac{q^4}{4}$.

\smallskip

\item[(xiv)] \quad $0 \le v \le \frac{q^4+2}{3}$.

\smallskip

\item[(xv)] \quad $0 \le w \le \frac{q^2-\sqrt{2}q+4}{4}$.

\smallskip

\item[(xvi)] \quad $0 \le x \le \frac{q}{\sqrt{2}}-2$. 
\end{enumerate}
\end{theorem}

\begin{corollary} \label{cor:decnum4q8}
Suppose $n=1$, that is $G = {^2F_4}(8)$, and 
$\ell \, | \, q^2-\sqrt{2}q+1$. Then $\ell = 5$ and for the
decomposition numbers in Theorem~\ref{thm:decnum4} we have
$a=e=j=t=x=0$, $s=1$.
\end{corollary}

\subsection{Remarks on the decomposition matrices} \label{sub:rmkdecmat}

\begin{enumerate}
\item[(a)] If $\ell$ divides $q^2-1$, then $\ell$ is a linear prime for
$G$ in the sense of \cite[Section~6]{HissHabil} and most
of the statements in Theorem~\ref{thm:decnum1} are immediate
consequences of results of G.~Hi{\ss}; see~\cite[Theorem~6.3.7]{HissHabil}. 
One can see from Table~\ref{tab:decuni1} that each unipotent block
of~$G$ coincides with exactly one modular Harish-Chandra series. 

\item[(b)] Tables~\ref{tab:dec2}, \ref{tab:dec3} and
\ref{tab:dec6}-\ref{tab:decnilp} are valid for all odd primes $\ell$,
in particular for $\ell=3$ and those primes~$\ell$ where the Sylow
$\ell$-subgroups of $G$ are cyclic. For $\ell=3$ this will be
discussed in the next section, for the remaining odd primes this
follows from Theorems~\ref{thm:decnum1}, \ref{thm:decnum2},
\ref{thm:decnum3}, \ref{thm:decnum4} and \cite{HissHabil}. Example:
the \textit{otherwise} in Table~\ref{tab:dec10} means \textit{for all
  odd primes $\ell$ not dividing $q^4+1$}. 

\item[(c)] Together with the ordinary character table of $G$ in
CHEVIE, the decomposition matrices in Appendices~C and D determine
all irreducible $\ell$-modular Brauer characters of $G$ for all
primes $\ell \neq 2,3$ except for three irreducible Brauer
characters (two unipotent and one non-unipotent) in case 
$\ell \, | \, q^2+1$ and four irreducible Brauer characters in case
$\ell \, | \, q^4+1$.

\item[(d)] For all non-unipotent blocks, the tables in Appendix~D
describe only the decomposition numbers of the corresponding ordinary
basic sets with one exception: If $s_5 \in G^*$ is semisimple of type
$g_5$ and $3 \neq \ell \, | \, q^2+1$, then
Table~\ref{tab:decg5noncyc} contains information on the decomposition
numbers of all ordinary characters in $\mathcal{E}_\ell(G, s_5)$. The
reason for this is that we were not able to compute the
decomposition number $a'$ in $\chi_{5,St}$ and we hope that this
additional information might be useful in the determination of $a'$
using arguments similar to those in~\cite{OkuWakiSU3}.

\item[(e)] The decomposition numbers of ${^2F_4}(2)$ and the simple
Tits group ${^2F_4}(2)'$ were calculated for all odd primes $\ell$ by
G.~Hi{\ss}~\cite{HissDecompTits}. The case of defining characteristic
was handled by F.~Veldkamp~\cite{Veldkamp2F4}. These decomposition
matrices are contained in the GAP library \cite{GAP4}.
\end{enumerate}

\subsection{On a conjecture of M.~Geck and G.~Hi{\ss}} \label{sub:conj34}

In the following corollary we verify a conjecture of M.~Geck and
G.~Hi{\ss} in the special case of the Ree groups
$G={^2F_4}(q^2)$, see \cite[Conjecture~3.4]{GeckHiss}.

\begin{corollary} \label{cor:ghconj}
Let $\ell$ be a good prime for a root system of type $F_4$ and 
let $\mathcal{M}(\mathcal{F}_1) = \{\chi_1\}$, 
$\mathcal{M}(\mathcal{F}_2) = \{\chi_2, \chi_3\}$, 
$\mathcal{M}(\mathcal{F}_3) = \{\chi_4\}$, 
$\mathcal{M}(\mathcal{F}_4) = \{\chi_5, \dots, \chi_{17}\}$, 
$\mathcal{M}(\mathcal{F}_5) = \{\chi_{18}\}$, 
$\mathcal{M}(\mathcal{F}_6) = \{\chi_{19}, \chi_{20}\}$, 
$\mathcal{M}(\mathcal{F}_7) = \{\chi_{21}\}$ be the families of
unipotent irreducible characters of $G={^2F_4}(2^{2n+1})$, $n>0$, see
Subsection~\ref{char2F4}. Furthermore, let $D$ be the part of the
decomposition matrix of $G$ corresponding to the rows labeled by the
ordinary unipotent irreducible characters.
\begin{enumerate}
\item[(a)] The irreducible Brauer characters in
$\mathcal{E}_\ell(G,1)$ can be labeled such that $D$ has the following
shape:
\[
D = \left(\begin{array}{ccccccc}
D_1 & 0 & 0 & 0 & 0 & 0 & 0 \\
* & D_2 & 0 & 0 & 0 & 0 & 0 \\
* & * & D_3 & 0 & 0 & 0 & 0 \\
* & * & * & D_4 & 0 & 0 & 0 \\
* & * & * & * & D_5 & 0 & 0 \\
* & * & * & * & * & D_6 & 0 \\
* & * & * & * & * & * & D_7 
\end{array}\right)
\]
where $D_i$ is the identity matrix of size 
$|\mathcal{M}(\mathcal{F}_i)| \times |\mathcal{M}(\mathcal{F}_i)|$ for
$1 \le i \le 7$.

\item[(b)] If an ordinary unipotent character $\chi \in \Irr(G)$ is
cuspidal, then $\breve{\chi}$ is an irreducible Brauer character.
\end{enumerate}
\end{corollary}
\begin{proof}
The condition that $\ell$ is good for $F_4$ is equivalent with
$\ell>3$. So the claim follows from the decomposition matrices in
Appendix~C and the Brauer trees in~\cite{HissHabil},
\cite{HissTrees2F4}. 
\end{proof}


\section{\texorpdfstring{On the Decomposition Matrices for
    $\ell=3$}{On the Decomposition Matrices for l=3}}
\label{sec:decompmatl3}

In this section, we provide some information on the $3$-modular
decomposition numbers of $G={^2F_4}(q^2)$, $q^2=2^{2n+1}$, $n>0$.
We fix a $3$-modular splitting system $(K,R,k)$ for all subgroups of
$G$ as in Subsection~\ref{ordmod}. All references to Brauer
characters, blocks, decomposition numbers in this section will
refer to characteristic~$3$. 
The prime number $3$ is a bad prime for $\G$ in the sense
of~\cite[p.~125]{DigneMichel} and it divides $\phi_4$ and $\phi_{12}$ 
and does not divide $\widetilde{\phi}_1 \, \phi_8' \, \phi_8'' \,
\phi_{24}' \, \phi_{24}''$. The Sylow $3$-subgroups of $G$ are non-abelian.

For every semisimple element $s \in G^*$ of order prime to $3$, the
set $\mathcal{E}_3(G, s)$ is a union of blocks. However, since $3$ is
a bad prime for $\G$, the results on ordinary basic sets
in~\cite{GeckHissBasicSets} do not apply in this situation. In fact,
the set $\mathcal{E}(G, 1)$ of ordinary unipotent irreducible
characters of $G$ is not a basic set for $\mathcal{E}_3(G, 1)$.
Using the explicit knowledge of the character table of $G$, is is not
difficult to see that
\begin{equation*}
\{\chi_1, \chi_2, \chi_3, \chi_4, \chi_{5,1}, \chi_5, \chi_6, \chi_7, \chi_8, 
\chi_{10}, \chi_{11}, \chi_{12}, \chi_{13}, \chi_{14}, \chi_{15}, \chi_{18}, 
\chi_{19}, \chi_{20}, \chi_{21}\}
\end{equation*}
is a basic set for $\mathcal{E}_3(G, 1)$. There are also basic sets
consisting entirely of unipotent characters, but we had to use this
basic set with the non-unipotent character $\chi_{5,1}$ in order to
show that the decomposition matrix has a unitriangular shape.
Again, using the explicit knowledge of the character table of $G$, it
is easy to determine the relations expressing the restrictions of the
remaining ordinary characters in $\mathcal{E}_3(G, 1)$ to the
$3$-regular elements as linear combinations of the characters in the
basic set, see Table~\ref{tab:relsl3}. The following statements
will be proved in Section~\ref{sec:proofs}.

We begin with some comments on the decomposition numbers of the
unipotent blocks $\mathcal{E}_3(G, 1)$ which are given in 
Table~\ref{tab:decunil3}. 
There are $19$ irreducible Brauer characters in the unipotent blocks; 
notation for these Brauer characters is given in the first row of
Table~\ref{tab:decunil3}. These irreducible Brauer characters are 
distributed into $14$ Harish-Chandra series corresponding to the Levi
subgroups $T$, $L_a$, $L_b$ and $G$.  
The Levi subgroups $T$ and $L_a$ each have a unique cuspidal unipotent
Brauer character, and the Levi subgroup $L_b$ has exactly two cuspidal
unipotent Brauer characters. The corresponding modular Harish-Chandra
series of $G$ contain, respectively, $4$, $1$, $2$, $2$ irreducible
Brauer characters. Additionally, $G$ has $10$ cuspidal unipotent
Brauer characters.  

The principal series is abbreviated by $ps$.
The Levi subgroup~$L_a$ has a unique cuspidal unipotent Brauer
character, the modular Steinberg character; see~\cite{GHMCusp}.
We denote the corresponding modular Harish-Chandra series of $G$ by
$A_1$. The Levi subgroup $L_b$~has~two cuspidal unipotent Brauer
characters $\varphi_a$, $\varphi_b$, the restrictions of the
ordinary unipotent irreducible cuspidal characters to the $3$-regular
elements. We denote the corresponding modular Harish-Chandra series by
${^2B_2}[a]$, ${^2B_2}[b]$, respectively. 
The remaining columns correspond to the cuspidal unipotent Brauer
characters.  

The situation for the non-unipotent characters of $G$ is less
complicated. Using the character table of $G$, one can see that for
all $3'$-elements $s \neq 1$, the set $\mathcal{E}(G, s)$ is an
ordinary basic set for $\mathcal{E}_3(G,s)$. The decomposition numbers
of the non-unipotent irreducible characters are given in
Tables~\ref{tab:dec2}, \ref{tab:dec3},
\ref{tab:dec6}-\ref{tab:decnilp} in the Appendix.  
In these tables, only the decomposition numbers of the ordinary basic
sets $\mathcal{E}(G, s)$, $s \neq 1$, are listed.

\begin{theorem} \label{thm:decnuml3}
The $3$-modular decomposition numbers of $G={^2F_4}(2^{2n+1})$, $n>0$,
are given by Tables~\ref{tab:decunil3} and
\ref{tab:dec2}, \ref{tab:dec3} and \ref{tab:dec6}-\ref{tab:decnilp} in
the Appendix. There are the following bounds:
\begin{enumerate}
\item[(i)] \quad $2 \le a \le q^2$.

\smallskip

\item[(ii)] \quad $0 \le b \le \frac{q^2+\sqrt{2}q}{4}$. If 
$n \equiv 1$ or $4$ mod $6$, then $b \ge 1$. If 
$n \equiv 4$ or $13$ mod $18$, \linebreak 
\hspace*{0.25cm} then $b \ge 2$. 

\smallskip

\item[(iii)] \quad $0 \le c \le \frac{q^2-\sqrt{2}q}{4}$. If 
$n \equiv 1$ or $4$ mod $6$, then $c \ge 1$. If 
$n \equiv 4$ or $13$ mod $18$, \linebreak 
\hspace*{0.25cm} then $c \ge 2$. 

\smallskip

\item[(iv)] \quad $2 \le d \le q^4$.

\smallskip

\item[(v)] \quad $1 \le e \le \frac{q^2+2}{2}$. If 
$n \equiv 1$ or $4$ mod $6$, then $e \ge 2$. If 
$n \equiv 4$ or $13$ mod $18$, \linebreak 
\hspace*{0.25cm} then $e \ge 3$. 

\smallskip

\item[(vi)] \quad $0 \le x_7 \le \frac{q^2}{2}$.

\smallskip

\item[(vii)] \quad $0 \le x_8 \le \frac{q^2+3\sqrt{2}q+4}{12}$.

\smallskip

\item[(viii)] \quad $0 \le x_{10} \le \frac{q^2-2}{6}$.

\smallskip

\item[(ix)] \quad $1 \le x_{15} \le \frac{q^2+1}{3}$.

\smallskip

\item[(x)] \quad $0 \le x_{18} \le q^2(q^2-1)$.

\smallskip

\item[(xi)] \quad $1 \le x_{21} \le q^6$. 
\end{enumerate}
\end{theorem}

Theorem~\ref{thm:decnuml3} will be proved in Section~\ref{sec:proofs}.  

\begin{corollary} \label{cor:decnuml3q8}
Suppose $n=1$, that is $G = {^2F_4}(8)$, and $\ell = 3$. Then for the 
decomposition number $c$ in Theorem~\ref{thm:decnuml3} we have $c = 1$.
\end{corollary}
\begin{proof}
This follows from the bounds in Theorem~\ref{thm:decnuml3}.
\end{proof}

\begin{corollary} \label{cor:triang}
Let $\ell$ be an odd prime and $G={^2F_4}(2^{2n+1})$ where
$n \ge 0$. After a suitable arrangement of rows and columns, the
$\ell$-modular decomposition matrix of $G$ has a lower unitriangular
shape. 
\end{corollary}
\begin{proof}
For $n=0$, this follows from the decomposition matrices of ${^2F_4}(2)$
in GAP. For $n>0$, it is a consequence of the decomposition matrices in
Theorems~\ref{thm:decnum1}-\ref{thm:decnum3}, \ref{thm:decnum4},
\ref{thm:decnuml3} and the Brauer trees in~\cite{HissTrees2F4}.
\end{proof}

\subsection{\texorpdfstring{Remarks on the $\mathbf{3}$-modular
    decomposition matrices}{Remarks on the 3-modular decomposition matrices}}  
\label{sub:rmkl3}

It seems to be necessary to use methods different from the ones in
this paper to improve the bounds for the decomposition numbers in
Theorem~\ref{thm:decnuml3} substantially. We induced projective
characters of $B$ and several maximal subgroups of $G$, computed tensor
products of projective characters and irreducible characters
of~$G$ and considered restrictions of modules to $3$-blocks
of the parabolic subgroup~$P_a$. However, in this way we were not
able to derive better upper bounds for the decomposition numbers in
Theorem~\ref{thm:decnuml3}.


\section{Proofs}
\label{sec:proofs}

In this section, we prove the theorems of Sections~\ref{sec:decompmat}
and \ref{sec:decompmatl3}. We begin by describing some information and
general methods which will be used in the proofs of these theorems.

\subsection{Projective characters} \label{projectives}

One of the main ingredients in the proof of the theorems in
Sections~\ref{sec:decompmat} and \ref{sec:decompmatl3}
is the construction of projective characters. When
we speak of a projective character~of~$G$ we mean an ordinary
character which is a sum of characters of the projective
indecomposable $kG$-modules (PIMs). Via Brauer reciprocity, we can
interpret the decomposition numbers of $G$ as the scalar products of
the projective characters corresponding to the PIMs with the ordinary
irreducible characters of~$G$. We are going use the following methods
to produce projective characters:
\begin{itemize}
\item Every character of defect $0$ of $G$ is projective.

\item If $\chi$, $\psi$ are characters of $G$ and $\psi$ is
  projective, then the product $\chi \otimes \psi$ is projective.

\item If $\psi$ is a projective character of a subgroup $H$ of $G$,
  then the induced character $\psi^G$ is projective. This is
  particularly useful, if $H$ is an $\ell'$-subgroup, since then every
  character of $H$ is projective. 

\item If $L$ is one of the Levi subgroups $T$, $L_a$, $L_b$ and
  $\psi$ a projective character of $L$, then the Harish-Chandra
  induced character $R_L^G(\psi)$ is projective;
  see~\cite[Lemma~4.4.3]{HissHabil}. In particular, if $\gamma$ is the
  Gelfand-Graev character of $L$, then $R_L^G(\gamma)$ is projective.
\end{itemize}
The character tables of $B$, $P_a$, $P_b$, $G$ are available as CHEVIE
files and we have computer programs (written by C.~K\"ohler and the
author), based on the class fusions in~\cite{HimstedtHuang2F4MaxParab}
and \cite{HimstedtHuang2F4Borel}, for induction and restriction of
characters between these parabolic subgroups. Using these programs and
CHEVIE, we can easily compute induced and Harish-Chandra induced
characters as well as tensor products and scalar products of
characters of the various parabolic subgroups of $G$. 

\subsection{Hecke algebras} \label{hecke}

Another ingredient in the proof of the theorems in
Sections~\ref{sec:decompmat} and~\ref{sec:decompmatl3} 
are the decomposition numbers of the Hecke algebra
$\mathcal{H}$ corresponding to the permutation module on the cosets of
the Borel subgroup $B$ in $G$. By \cite[Corollary~4.10]{Dipper}, these
decomposition numbers form a submatrix of the decomposition matrix of
$\mathcal{E}_\ell(G, 1)$.

The computation of the decomposition numbers of $\mathcal{H}$ is
similar to the proof of~\cite[Proposition~5.1]{GeckDissPub}:
Let $\mathbf{1}_B$ be the trivial $RB$-module, and 
$V := \mathbf{1}_B^G$ the corresponding $RG$-permutation 
module with ordinary character $\vartheta$. Computing scalar products
with CHEVIE, we get
\[
\vartheta =
\chi_{1}+\chi_{4}+2\chi_{5}+2\chi_{6}+2\chi_{7}+\chi_{18}+\chi_{21}. 
\]
The Weyl group $W^1$ of $G$ (as a group with a BN-pair) has Coxeter
generators $w_a$, $w_b$ and is isomorphic to a dihedral group of order
$16$; see~\cite[Section~2]{HimstedtHuang2F4Borel}. 
The Hecke algebra $\mathcal{H} := \End_{RG}(V)$ of $V$ has a natural 
$R$-basis $\{T_w | w \in W^1\}$ satisfying the relations
\[
T_{\gamma} T_w = 
\begin{cases} T_{w_\gamma w} &, \text{if   } \, l(w_\gamma w)>l(w),\\
p_\gamma T_{w_\gamma w} + (p_\gamma-1) T_w &, \text{if   } \, l(w_\gamma w)<l(w)
\end{cases}
\]
for all $\gamma \in \{a, b\}$ and $w \in W^1$. Here, $l$ denotes the
Coxeter length and we have used the abbreviation
$T_\gamma=T_{w_\gamma}$. The parameters $p_a=q^2$ and $p_b=q^4$ are given
in~\cite[p.~66]{HissHabil}. As described
in~\cite[\S~67C]{CurtisReiner:90}, one can construct irreducible
representations $\ind$, $\sgn$, $\sigma_1$, $\sigma_2$, $S_1$,
$S_{-1}$ and $S_0$ of $\mathcal{H}$ affording irreducible
representations of $\mathcal{H}_K := K \otimes_R \mathcal{H}$ 
(by extending scalars), which will be denoted in the same way.
The $1$-dimensional of these representations are
\[
\ind: \begin{cases} T_a \mapsto q^2\\
                          T_b  \mapsto q^4
\end{cases}\hspace{-0.2cm}, \quad
\sgn: \begin{cases} T_a \mapsto -1\\
                          T_b  \mapsto -1
\end{cases}\hspace{-0.2cm}, \quad
\sigma_1: \begin{cases}    T_a \mapsto -1\\
                          T_b  \mapsto q^4
\end{cases}\hspace{-0.2cm}, \quad
\sigma_2: \begin{cases}    T_a \mapsto q^2\\
                          T_b  \mapsto -1
\end{cases},
\]
and for $\epsilon = 0, \pm 1$ there are the following $2$-dimensional
representations:
\[
S_\epsilon: \quad T_a \mapsto \left(\begin{array}{cc}
q^2 & 0\\ q^2 + \epsilon \sqrt{2}q + 1 & -1 \end{array}\right) , \quad
T_b \mapsto \left(\begin{array}{cc}
-1 & q^2\\ 0 & q^4 \end{array}\right).
\]
There is a natural bijection (``Fitting correspondence'') between the
isomorphism classes~of irreducible representations of
$\mathcal{H}_K$ and the irreducible constituents of $\vartheta$. 
The representations $\ind$, $\sgn$, $\sigma_1$, $\sigma_2$, $S_1$,
$S_{-1}$ and $S_0$ correspond to the unipotent characters $\chi_1$,
$\chi_{21}$, $\chi_{4}$, $\chi_{18}$, $\chi_{5}$, $\chi_{6}$ and $\chi_{7}$, 
respectively. This can be seen for example by computing generic degrees
using~\cite[Theorem~10.11.5]{Carter2}.

\begin{proposition} \label{prop:hecke_decnum} 
For $\ell \, | \, \widetilde{\phi}_1 \phi_4 \phi_8' \phi_8''$,
the $\ell$-modular decomposition numbers of the Hecke algebra
$\mathcal{H}$ are given by Table~\ref{tab:decnumH}. 
\end{proposition}
\begin{proof}
The decomposition numbers of the $1$-dimensional representations are
clear. For the $2$-dimensional ones, they follow easily from
considering simultaneous eigenspaces. 
\end{proof}

\begin{table}[!ht]
\caption[]{Decomposition numbers of the Hecke algebra $\mathcal{H}$.} 
\label{tab:decnumH}  

\begin{center}
\begin{tabular}{l|ccccccc|cccc|cccc|cccc}
\hline
\rule{0cm}{0.4cm}
& \multicolumn{7}{l|}{$\ell \, | \, q^2-1$} & 
  \multicolumn{4}{l|}{$\ell \, | \, q^2+1$} &
  \multicolumn{4}{l|}{$\ell \, | \, q^2+\sqrt{2}q+1$} &
  \multicolumn{4}{l}{$\ell \, | \, q^2-\sqrt{2}q+1$}
\rule[-0.2cm]{0cm}{0.36cm}\\
\hline
\rule{0cm}{0.36cm}
$\ind$ & $1$ & $.$ & $.$ & $.$ & $.$ & $.$ & $.$ & $1$ & $.$ & $.$ &
$.$ & $1$ \hspace{0.1cm} & $.$ \hspace{0.1cm} & $.$  & $.$  &
$1$ \hspace{0.1cm} & $.$ \hspace{0.1cm} & $.$ & $.$    
\rule[-0.2cm]{0cm}{0.36cm}\\
\rule{0cm}{0.36cm}
$\sigma_1$ & $.$ & $1$ & $.$ & $.$ & $.$ & $.$ & $.$ & $1$ & $.$ & $.$
& $.$ & $.$ \hspace{0.1cm} & $1$ \hspace{0.1cm} & $.$ & $.$  &
$.$ \hspace{0.1cm} & $1$ \hspace{0.1cm} & $.$  & $.$    
\rule[-0.2cm]{0cm}{0.36cm}\\
\rule{0cm}{0.36cm}
$S_1$ & $.$ & $.$ & $1$ & $.$ & $.$ & $.$ & $.$ & $.$ & $1$ & $.$
& $.$ &  $1$ \hspace{0.1cm} & $1$ \hspace{0.1cm} & $.$  & $.$  &
$.$ \hspace{0.1cm} & $.$ \hspace{0.1cm} & $1$  & $.$    
\rule[-0.2cm]{0cm}{0.36cm}\\
\rule{0cm}{0.36cm}
$S_{-1}$ & $.$ & $.$ & $.$ & $1$ & $.$ & $.$ & $.$ & $.$ & $.$ & $1$
& $.$ &  $.$ \hspace{0.1cm} & $.$ \hspace{0.1cm} & $1$  & $.$  &
$1$ \hspace{0.1cm} & $1$ \hspace{0.1cm} & $.$  & $.$    
\rule[-0.2cm]{0cm}{0.36cm}\\
\rule{0cm}{0.36cm}
$S_0$ & $.$ & $.$ & $.$ & $.$ & $1$ & $.$ & $.$ & $1$ & $.$ & $.$
& $1$ &  $.$ \hspace{0.1cm} & $.$ \hspace{0.1cm} & $.$  & $1$  &
$.$ \hspace{0.1cm} & $.$ \hspace{0.1cm} & $.$  & $1$  
\rule[-0.2cm]{0cm}{0.36cm}\\
\rule{0cm}{0.36cm}
$\sigma_2$ & $.$ & $.$ & $.$ & $.$ & $.$ & $1$ & $.$ & $.$ & $.$ & $.$
& $1$ &  $1$ \hspace{0.1cm} & $.$ \hspace{0.1cm} & $.$  & $.$  &
$1$ \hspace{0.1cm} & $.$ \hspace{0.1cm} & $.$  & $.$    
\rule[-0.2cm]{0cm}{0.36cm}\\
\rule{0cm}{0.36cm}
$\sgn$ & $.$ & $.$ & $.$ & $.$ & $.$ & $.$ & $1$ & $.$ & $.$ & $.$ &
$1$ &  $.$ \hspace{0.1cm} & $1$ \hspace{0.1cm} & $.$  & $.$  &
$.$ \hspace{0.1cm} & $1$ \hspace{0.1cm} & $.$  & $.$    
\rule[-0.2cm]{0cm}{0.36cm}\\
\hline
\end{tabular}
\end{center}
\end{table}

\subsection{Relations} \label{relations}

Let $\ell>3$ be a prime number. By Subsection~\ref{blocksbas}, 
the set of unipotent irreducible characters of $G$ is an ordinary
basic set for the unipotent blocks $\mathcal{E}_\ell(G, 1)$. In
particular, for every non-unipotent irreducible character 
$\chi \in \mathcal{E}_\ell(G, 1)$, there are $a_{\chi,j} \in \Z$ such
that 
\begin{equation} \label{eq:rel}
\breve\chi = \sum_{j=1}^{21} a_{\chi,j} \cdot \chi_j.
\end{equation}
Here, we have identified $\chi_j$ and $\breve\chi_j$. We
call~(\ref{eq:rel}) a \textit{relation} with respect to the basic set
of unipotent characters. Such relations can be used to derive lower
bounds for the decomposition numbers of the unipotent characters or
to prove the indecomposability of certain projective modules.
For the relevant primes $\ell$, the relations with respect to the
basic set of unipotent characters are given in
Tables~\ref{tab:rels1}-\ref{tab:rels4} in the Appendix. 

The rows of these tables are labeled by the
non-unipotent irreducible characters in $\mathcal{E}_\ell(G,1)$ and
the numbering of the columns corresponds to the unipotent irreducible
characters of $G$. The entry in the row corresponding to 
$\chi \in \mathcal{E}_\ell(G,1)$ and the $j$-th column~is equal to 
$a_{\chi,j}$. In these tables zeros are replaced by dots.
The data in Tables~\ref{tab:rels1}-\ref{tab:rels4} can easily be
computed from the character table of $G$ in CHEVIE. 

\subsection{Jordan decomposition of Brauer characters} \label{jordan}

Our main tool to determine the decomposition numbers of the
non-unipotent blocks of $G$ is C.~Bonnaf{\'e}'s and R.~Rouquier's
modular version of the Jordan decomposition of
characters~\cite[Theorem~11.8]{BonnafeRouquier}. Let $\ell$ be an odd
prime. For $1 \le i \le 18$, let $s_i \in G^*$ be a semisimple
$\ell'$-element of type $g_i$, see~\cite[Table~B.10]{HimstedtHuangDade2F4}. 
We consider the centralizer of $s_i$ in the algebraic
group~$\G^*$. Since the center $Z(\G)$ is connected, the centralizer 
$C_{\G^*}(s_i)$ is also connected. Now suppose $i \neq 5$. Then, 
$C_{\G^*}(s_i)$ can be realized as the centralizer of a semisimple
element of an order not divisible by $2$ and $3$. So
by~\cite[Proposition~13.16]{CabEng}, the centralizer $C_{\G^*}(s_i)$
is a Levi subgroup of~$\G^*$ and we can
apply~\cite[Theorem~11.8]{BonnafeRouquier}. Thus, for all $i \neq 5$, 
Lusztig induction induces a 1-1-correspondence between the sets 
$\mathcal{E}_\ell(G, s_i)$ and $\mathcal{E}_\ell(C_{G^*}(s_i), 1)$ of
ordinary irreducible characters, and with respect to this
correspondence the decomposition matrices of 
$\mathcal{E}_\ell(G, s_i)$ and $\mathcal{E}_\ell(C_{G^*}(s_i), 1)$
coincide (after a suitable ordering of the columns).

The Dynkin type of $C_{G^*}(s_i)$ is given
in~\cite[Table~B.9]{HimstedtHuangDade2F4}, see
also~\cite[Lemma~D.3.2]{HissHabil}. For $i \neq 1,5$, it is
${^2B_2}$, $A_1$ or $A_0$ and so the decomposition numbers of
$\mathcal{E}_\ell(C_{G^*}(s_i),1)$ are known, see~\cite{Burkhardt}
and~\cite{James}. Thus, Bonnaf{\'e}'s and Rouquier's Jordan
decomposition gives us the decomposition numbers for all non-unipotent
blocks of $G$ except for $\mathcal{E}_\ell(G, s_5)$. 
Note that the semisimple element $s_5$ has order $3$ and
is isolated in the sense of \cite[1.B]{BonnafeQuasi}. This follows
from \cite[Corollary~1.4]{BonnafeQuasi} and the data
in~\cite[Table~B.9]{HimstedtHuangDade2F4}. So, $C_{\G^*}(s_5)$ is not
contained in a Levi subgroup of a proper parabolic subgroup of $\G^*$
and Bonnaf{\'e}'s and Rouquier's theorem does not apply.


\subsection{Proof of Theorem~\ref{thm:decnum1}}

Suppose $\ell \, | \, q^2-1$. By~\cite[Proposition~6.3.4]{HissHabil},
$\ell$ is a linear prime for $G$, so that we can
apply~\cite[Theorem~6.3.7]{HissHabil}. 
We begin with the decomposition numbers of the unipotent blocks of $G$
in Table~\ref{tab:decuni1}. The numbers in the right most column are
obtained by counting the elements of $\ell$-power order in $G^*$,
which can easily be done using the representatives
in~\cite[Table~B.10]{HimstedtHuangDade2F4}. The decomposition numbers 
in Table~\ref{tab:decuni1} follow from~\cite[Theorem~6.3.7]{HissHabil}
and the relations in Table~\ref{tab:rels1}. 

The modular Harish-Chandra series can be determined as follows: By
Subsection~\ref{hecke}, the Brauer characters $\phi_1$, $\phi_4$,
$\phi_5$, $\phi_6$, $\phi_7$, $\phi_{18}$, $\phi_{21}$ are the
irreducible Brauer characters in the principal
series. By~\cite[Lemma~4.3]{HissHC}, the Brauer characters
$\phi_i$, $8 \le i \le 17$, are cuspidal.
The Levi subgroup $L_b$ has $4$ unipotent irreducible Brauer
characters: the restrictions of the trivial character and of
the Steinberg character and the restrictions of the two
ordinary cuspidal unipotent characters of degree
$\frac{q}{\sqrt{2}}(q^2-1)$ to the $\ell$-regular elements
(this follows from the decomposition numbers in~\cite{Burkhardt}).
The first two of these Brauer characters are in the principal series,
the latter two are cuspidal Brauer characters. 
Again by~\cite[Lemma~4.3]{HissHC}, we know that $\phi_2$, $\phi_3$,
$\phi_{19}$, $\phi_{20}$ are not cuspidal. Now, the claim about
the Harish-Chandra series of these four characters follows from the
remarks in Subsection~\ref{char2F4}.

Finally, we treat the decomposition numbers for the non-unipotent blocks
in Tables~\ref{tab:dec2}, \ref{tab:dec3} and
\ref{tab:decg5cyc}-\ref{tab:decnilp}. Let $s \neq 1$ be a semisimple
$\ell'$-element of $G^*$. If $s$ is not of type $g_5$, then the
decomposition numbers of $\mathcal{E}(G, s)$ are clear by
Bonnaf{\'e}'s and Rouquier's Jordan decomposition, see
Subsection~\ref{jordan}. This proves the decomposition numbers in
Tables~\ref{tab:dec2}, \ref{tab:dec3} and
\ref{tab:dec6}-\ref{tab:decnilp}. The decomposition numbers in
Table~\ref{tab:decg5cyc} follow from the fact that $\ell$ is a linear
prime and \cite[Theorem~6.3.7]{HissHabil}. This completes the proof of
Theorem~\ref{thm:decnum1}. \hfill $\Box$


\subsection{Proof of Theorem~\ref{thm:decnum2}}

Suppose $\ell>3$ and $\ell \, | \, q^2+1$. We start with the
decomposition numbers of the unipotent blocks of $G$ 
in Table~\ref{tab:decuni2}. 
As in the proof of Theorem~\ref{thm:decnum1}, the numbers in the right
most column are obtained by counting the elements of $\ell$-power
order in $G^*$, which can easily be done using the representatives
in~\cite[Table~B.10]{HimstedtHuangDade2F4}. Note that under the
assumptions of the theorem, we have $\frac{1}{48}(\ell^f-1)(\ell^f-11)>0$ 
if and only if $\ell \neq 11$ or $n \equiv 27$ mod $55$ where
$q^2=2^{2n+1}$. Consequently, the relation in the last row of
Table~\ref{tab:rels2} exists if and only if $\ell \neq 11$ or 
$n \equiv 27$ mod $55$.

Using the relations in Table~\ref{tab:rels2}, the rows of the
decomposition matrix of $\mathcal{E}_\ell(G, 1)$ corresponding to the
non-unipotent irreducible characters can be written as linear
combinations of the rows corresponding to the unipotent irreducible
characters. So, it is sufficient to determine the decomposition
numbers of the unipotent irreducible characters $\chi_1, \dots, \chi_{21}$. 
We construct projective characters $\Psi_1$, \dots, $\Psi_{21}$ of $G$ 
according to Table~\ref{tab:proj2} in the Appendix and compute the
scalar products $(\chi_i, \Psi_j)_G$ for $1 \le i,j \le 21$ using
CHEVIE. These scalar products are given in Table~\ref{tab:scalar2} in
the Appendix. We already see from Table~\ref{tab:scalar2} that the
decomposition matrix of $\mathcal{E}_\ell(G, 1)$ has a lower
unitriangular shape giving us a natural bijection between the set of
ordinary unipotent irreducible characters and the
set of irreducible Brauer characters in $\mathcal{E}_\ell(G, 1)$. For
$1 \le i \le 21$, let $\phi_i$ be the irreducible Brauer character
corresponding to $\chi_i$ and $\Phi_i$ the character of the
corresponding PIM. 

From Table~\ref{tab:scalar2}, we already get the assertions about all
$\Phi_i$, $i \neq 1,4,7,8,9,17,18$ in the decomposition matrix
Table~\ref{tab:decuni2}. Let $a$ be the multiplicity of $\chi_{18}$ in
$\Phi_{17}$ and $b,c,d,e$ the multiplicity of $\chi_{21}$ in $\Phi_8$, $\Phi_9$,
$\Phi_{17}$, $\Phi_{18}$, respectively. The entries in
Table~\ref{tab:scalar2} lead to the upper bounds for
$a,b,c,d,e$ in Theorem~\ref{thm:decnum2}, and the lower bounds follow
from the decomposition numbers of the non-unipotent characters in the
last three rows of Table~\ref{tab:decuni2}, since decomposition
numbers are non-negative. So we have shown all assertions about
$\Phi_{8}$, $\Phi_{9}$, $\Phi_{17}$, $\Phi_{18}$ in
Theorem~\ref{thm:decnum2}. Furthermore, \cite[Corollary~4.10]{Dipper} and
Proposition~\ref{prop:hecke_decnum} imply the assertions about $\Phi_1$ and
$\Phi_7$. 

So, we are only left with the decomposition numbers in the fourth
column of Table~\ref{tab:decuni2}. All of them are clear from
Table~\ref{tab:scalar2} except for the decomposition numbers 
$(\chi_7, \Phi_4)_G$, $(\chi_{21}, \Phi_4)_G \in \{0,1\}$. Since
$\Phi_7$ is not a summand of $\Psi_4$, we get $(\chi_7, \Phi_4)_G=1$.

To determine the decomposition number $(\chi_{21}, \Phi_4)_G$, we
collect some information on the modular Harish-Chandra series of $G$.
By Proposition~\ref{prop:hecke_decnum}, we already know that $\phi_1$,
$\phi_5$, $\phi_6$, $\phi_7$ are the Brauer characters in the
principal series. 
The Levi subgroup $L_b$ has two cuspidal unipotent irreducible Brauer
characters: the restrictions of the two ordinary cuspidal unipotent
characters of degree $\frac{q}{\sqrt{2}}(q^2-1)$ to the set of
$\ell$-regular elements (this follows from the decomposition numbers
in~\cite{Burkhardt} and \cite[Lemma~4.3]{HissHC}). We write
$\varphi_a$ and~$\varphi_b$ for these two Brauer characters 
and $\Phi_a$, $\Phi_b$ for the characters of the corresponding PIMs
of~$L_b$. Let ${^2B_2}[a]$, ${^2B_2}[b]$ be the modular
Harish-Chandra series of $G$ corresponding 
to~$\varphi_a$ and $\varphi_b$, respectively.
Let $u(\Phi_a)$ be the character of the unipotent quotient of $\Phi_a$
and $u(R_{L_b}^G(\Phi_a))$ the character of the unipotent quotient of
the Harish-Chandra induced character $R_{L_b}^G(\Phi_a)$,
see~\cite[Section~6]{HissHC}. By~\cite[Lemma~6.1]{HissHC},
Harish-Chandra induction commutes with taking unipotent quotients. 
Using the class fusions in~\cite{HimstedtHuang2F4MaxParab} and CHEVIE,
we compute  
\[
u(R_{L_b}^G(\Phi_a)) = R_{L_b}^G(u(\Phi_a)) = {_{P_b}\chi_2}(0)^G =
\chi_2 + \chi_{19}. 
\]
Thus, from~\cite[Section~5]{HissHC} we see that $\phi_2$ and $\phi_{19}$
are the only irreducible Brauer characters in the series
${^2B_2}[a]$. Analogously, by computing $u(R_{L_b}^G(\Phi_b))$
we see that~$\phi_3$ and $\phi_{20}$ are the only irreducible Brauer
characters in the series ${^2B_2}[b]$.

Next, we consider the modular Harish-Chandra series $A_1$. Let
$\Phi_{St}$ be the character of the PIM corresponding to the modular
Steinberg character of $L_a$. Using CHEVIE, we can compute the
unipotent quotient 
\begin{equation} \label{eq:HCLa}
u(R_{L_a}^G(\Phi_{St}))=R_{L_a}^G(u(\Phi_{St}))={_{P_a}\chi_2}(0)^G= 
\chi_4+\chi_5+\chi_6+\chi_7+\chi_{21}.
\end{equation}
By the construction in~\cite[Section~4]{HimstedtHuang2F4Borel},
the character ${_{B}\chi_{8}}^G$ is the Gelfand-Graev character of $G$.
Hence, Table~\ref{tab:proj2} implies that $\phi_{21}$ is the modular
Steinberg character of $G$ and it follows
from~\cite[Theorem~4.2]{GHMCusp} that $\phi_{21}$ is cuspidal. From
the decomposition~(\ref{eq:HCLa}) we get that $\phi_4$ is the only
irreducible Brauer character of~$G$ in the Harish-Chandra series $A_1$
and $(\chi_{21}, \Phi_4)_G=1$. Consequently, the remaining Brauer
characters $\phi_i$ for $8 \le i \le 18$ have to be cuspidal. This
proves all assertions about the unipotent blocks in
Theorem~\ref{thm:decnum2}. 

Finally, we deal with the decomposition numbers of the non-unipotent
irreducible characters of $G$ in
Tables~\ref{tab:dec2}-\ref{tab:decg5noncyc} and 
\ref{tab:dec6}-\ref{tab:decnilp}. Let $s \neq 1$ be a semisimple
$\ell'$-element of $G^*$. If $s$ is not of type $g_5$, then the
decomposition numbers of $\mathcal{E}(G, s)$ are clear by
Bonnaf{\'e}'s and Rouquier's Jordan decomposition, see
Subsection~\ref{jordan}. This proves the decomposition numbers in
Tables~\ref{tab:dec2}, \ref{tab:dec3} and
\ref{tab:dec6}-\ref{tab:decnilp}. The decomposition numbers in
Table~\ref{tab:decg5noncyc} can be determined as follows:
Suppose $s_5 \in G^*$ is a semisimple $\ell'$-element of type $g_5$. 
We get an approximation to the decomposition matrix of 
$\mathcal{E}(G, s_5)$ from the following scalar products of basic set
characters with projective characters:

\smallskip

\begin{center}
\begin{tabular}{l|ccc}
\hline
\rule{0cm}{0.4cm}
& $R_{L_a}^G(\gamma_{L_a})$ & ${_{P_b}\chi_{22}}^G$ & ${_B\chi_{8}}^G$ \\
\hline
\hspace{0.01cm} $\chi_{5,1}$ & $1$ & $.$ & $.$ 
\rule[-0.2cm]{0cm}{0.36cm}\\
\rule{0cm}{0.36cm}
$\chi_{5,q^2(q^2-1)}$ & $.$ & $1$ & $.$
\rule[-0.2cm]{0cm}{0.36cm}\\
\rule{0cm}{0.36cm}
$\chi_{5,St}$ & $1$ & $\frac{q^2-\sqrt{2}q}{4}$ & $1$
\rule[-0.2cm]{0cm}{0.36cm}\\
\hline
\end{tabular}
\end{center}

\smallskip

\noindent Here, $\gamma_{L_a}$ is the Gelfand-Graev character of
$L_a$, the character ${_B\chi_{8}}^G$ is the Gelfand-Graev character
of $G$ and ${_{P_b}\chi_{22}}$ is projective since $\ell \nmid |P_b|$.
This already gives the upper bound for $a'$ in
Theorem~\ref{thm:decnum2}. Let $\Phi_{5,1}, \Phi_{5,2}, \Phi_{5,3}$ be
the characters of the PIMs corresponding to the irreducible Brauer
characters $\phi_{5,1}$, $\phi_{5,2}$, $\phi_{5,3}$. We have the
following relations on the $\ell$-regular elements: 
\begin{eqnarray*}
\breve\chi_{6,1} & = & \breve\chi_{5,1} + \breve\chi_{5,q^2(q^2-1)},\\
\breve\chi_{6,St} & = & -\breve\chi_{5,q^2(q^2-1)} + \breve\chi_{5,St},\\
\breve\chi_{15,1} & = & -\breve\chi_{5,1} -2\cdot\breve\chi_{5,q^2(q^2-1)} + \breve\chi_{5,St}.
\end{eqnarray*}
Via these relations, we can express the decomposition numbers of
$\chi_{6,1}$, $\chi_{6,St}$, $\chi_{15,1}$ in terms of the
decomposition numbers of $\chi_{5,1}$, $\chi_{5,q^2(q^2-1)}$ and
$\chi_{5,St}$. Since decomposition numbers are non-negative, we get 
$(\chi_{5,St}, \Phi_{5,1})_G=1$ and $a' \ge 2$.
This completes the proof of Theorem~\ref{thm:decnum2}. \hfill $\Box$


\subsection{Proof of Theorem~\ref{thm:decnum3}}

Suppose $\ell \, | \, q^2+\sqrt{2}q+1$. The decomposition numbers
of the unipotent blocks of $G$ in Table~\ref{tab:decuni3} can be
determined as follows: As in the proof of Theorem~\ref{thm:decnum1},
the numbers in the right most column are obtained by counting the
elements of $\ell$-power order in $G^*$, which can easily be done
using the representatives in~\cite[Table~B.10]{HimstedtHuangDade2F4}. 
Note that under the assumptions of the theorem, we have
$\frac{1}{96}(\ell^f-1)(\ell^f-5)>0$ if and only if $\ell \neq 5$ or
$n \equiv 7$ or $n \equiv 12$ mod $20$ where
$q^2=2^{2n+1}$. Consequently, the relation in the last row of 
Table~\ref{tab:rels3} exists if and only if $\ell \neq 5$ or 
$n \equiv 7$ or $n \equiv 12$ mod $20$.

Using the relations in Table~\ref{tab:rels3}, the rows of the
decomposition matrix of $\mathcal{E}_\ell(G, 1)$ corresponding to the
non-unipotent irreducible characters can be written as linear
combinations of the rows corresponding to the unipotent irreducible
characters. So, it is sufficient to determine the decomposition
numbers of the unipotent irreducible characters $\chi_1, \dots, \chi_{21}$. 
We construct projective characters $\Psi_1$, \dots, $\Psi_{21}$ of $G$ 
according to Table~\ref{tab:proj3} and compute the scalar products 
$(\chi_i, \Psi_j)_G$ for $1 \le i,j \le 21$ using CHEVIE. In
Table~\ref{tab:proj3}, $\Phi_a$ and $\Phi_b$ are the characters of the
projective covers of the unipotent irreducible Brauer characters
$\varphi_a$ and $\varphi_b$ of $L_b$, respectively, and $\Phi_{St}$
is the character of the projective cover of the modular Steinberg
character $\varphi_{St}$ of $L_b$; see the comments in Subsection~\ref{case3}. For
the calculations in CHEVIE, we only deal with the unipotent
quotients. The scalar products $(\chi_i, \Psi_j)_G$ are given in
Table~\ref{tab:scalar3} in the Appendix. The symbol $*$ in this table
means some non-negative integer depending on $q$. We already see from
Table~\ref{tab:scalar3} that the decomposition matrix of
$\mathcal{E}_\ell(G, 1)$ has a lower unitriangular shape giving us a
natural bijection between the set of ordinary unipotent irreducible
characters and the set of irreducible Brauer characters in
$\mathcal{E}_\ell(G, 1)$. For $1 \le i \le 21$, let 
$\phi_i$ be the irreducible Brauer character corresponding to $\chi_i$
and $\Phi_i$ the character of the corresponding PIM.  

From Table~\ref{tab:scalar3}, we already get the assertions about
$\Phi_i$ for $i=6,7,8,15,16,21$ in the decomposition matrix
Table~\ref{tab:decuni3}. Furthermore, \cite[Corollary~4.10]{Dipper}
and Proposition~\ref{prop:hecke_decnum} imply the assertions about
$\Phi_1$ and $\Phi_4$ and we see that $\phi_1$, $\phi_4$, $\phi_6$ and
$\phi_7$ are the irreducible Brauer character in the principal series.

We introduce abbreviations for some of the decomposition numbers:
Let $a$ and $c$ be the multiplicity of $\chi_{18}$ in $\Phi_{9}$ and
$\Phi_{17}$, respectively. Furthermore, let $r$, $s$, $v$, $w$ be the
multiplicity of $\chi_{21}$ in $\Phi_9$, $\Phi_{10}$, $\Phi_{17}$,
$\Phi_{18}$, respectively. We construct three additional projective
characters of $G$: 
\[
\Psi_{13}' := \chi_3 \otimes \chi_8, \quad 
\Psi_{14}' := \chi_2 \otimes \chi_8, \quad 
\Psi_{18}' := \chi_2 \otimes \chi_6.
\]
These characters are projective since $\chi_6$ and $\chi_8$ have
defect~$0$. Using CHEVIE, we compute the scalar products 
$(\Psi_i', \chi_j)_G$ for $i = 13,14,18$ and $1 \le j \le 21$,
see Table~\ref{tab:addscal}. For all pairs $(\Psi_i',\chi_j)$ not
listed in this table, we have $(\Psi_i', \chi_j)_G=0$. 
The scalar products for $\Psi_{13}'$, $\Psi_{14}'$ imply:
\[
(\chi_{18}, \Phi_{13})_G=(\chi_{20}, \Phi_{13})_G=(\chi_{18},
\Phi_{14})_G=(\chi_{19}, \Phi_{14})_G=0. 
\]

\begin{table}[!ht]
\caption[]{The scalar products $(\Psi_i', \chi_j)_G$.} 
\label{tab:addscal}

\begin{center}
\begin{tabular}{c|ccccccc} \hline
\rule{0cm}{0.45cm}
& $\chi_7$ & $\chi_{13}$ & $\chi_{14}$ & $\chi_{18}$ & $\chi_{19}$ & $\chi_{20}$ & $\chi_{21}$
\rule[-0.2cm]{0cm}{0.3cm}\\
\hline
\rule{-0.12cm}{0.4cm}
$\Psi_{13}'$ & $0$ & $\frac{q}{\sqrt{2}}$ & $0$ & $0$ &
$\frac{q^2+\sqrt{2}q}{4}$ & $0$ & $\frac{\sqrt{2}q(q^2+3\sqrt{2}q+4)}{24}$
\rule[-0.2cm]{0cm}{0.3cm}\\
\rule{-0.12cm}{0.4cm}
$\Psi_{14}'$ & $0$ & $0$ & $\frac{q}{\sqrt{2}}$ & $0$ & $0$ &
$\frac{q^2+\sqrt{2}q}{4}$ & $\frac{\sqrt{2}q(q^2+3\sqrt{2}q+4)}{24}$
\rule[-0.2cm]{0cm}{0.3cm}\\
\rule{-0.12cm}{0.4cm}
$\Psi_{18}'$ & $\frac{q}{\sqrt{2}}$ & $\frac{q}{\sqrt{2}}$ & $0$
& $\frac{q}{\sqrt{2}}$ & $\frac{q^2+2}{2}$ & $\frac{q^2+\sqrt{2}q}{4}$
& $\frac{\sqrt{2}q(q^2+\sqrt{2}q+4)}{8}$
\rule[-0.2cm]{0cm}{0.3cm}\\
\hline
\end{tabular}
\end{center}
\end{table}

The pairs $(\chi_{11}, \chi_{12})$, $(\chi_{13}, \chi_{14})$,
$(\chi_{19}, \chi_{20})$ are pairs of complex conjugate characters,
and so are the pairs $(\Phi_{11}, \Phi_{12})$, 
$(\Phi_{13}, \Phi_{14})$, $(\Phi_{19}, \Phi_{20})$. On the other hand,
$\chi_9$, $\chi_{17}$ and $\chi_{18}$ are real-valued, and so are
$\Phi_9$, $\Phi_{17}$ and $\Phi_{18}$. Hence, we have the
following identities of decomposition numbers:
\begin{center}
\begin{tabular}{cccccccccc}
$(\chi_{18}, \Phi_{11})_G$ & $=$ & $(\chi_{18}, \Phi_{12})_G$ & $=:$ & $b$,\qquad
$(\chi_{19}, \Phi_{9})_G$  & $=$ & $(\chi_{20}, \Phi_{9})_G$ & $=:$ & $d$, \\
$(\chi_{19}, \Phi_{11})_G$ & $=$ & $(\chi_{20}, \Phi_{12})_G$ & $=:$ & $e$,\qquad
$(\chi_{20}, \Phi_{11})_G$ & $=$ & $(\chi_{19}, \Phi_{12})_G$ & $=:$ & $g$,\\
$(\chi_{19}, \Phi_{13})_G$ & $=$ & $(\chi_{20}, \Phi_{14})_G$ & $=:$ & $h$,\qquad
$(\chi_{19}, \Phi_{17})_G$ & $=$ & $(\chi_{20}, \Phi_{17})_G$ & $=:$ & $i$,\\
$(\chi_{19}, \Phi_{18})_G$ & $=$ & $(\chi_{20}, \Phi_{18})_G$ & $=:$ &$j $,\qquad
$(\chi_{21}, \Phi_{11})_G$ & $=$ & $(\chi_{21}, \Phi_{12})_G$ & $=:$ & $t$,\\
$(\chi_{21}, \Phi_{13})_G$ & $=$ & $(\chi_{21}, \Phi_{14})_G$ & $=:$ & $u$,\qquad
$(\chi_{21}, \Phi_{19})_G$ & $=$ & $(\chi_{21}, \Phi_{20})_G$ & $=:$ & $x$.
\end{tabular}
\end{center}

\medskip

From the scalar products in Table~\ref{tab:scalar3}, we readily get
the upper bounds in Theorem~\ref{thm:decnum3} except for the
decomposition numbers $h$, $j$, $u$, $w$. The scalar products with the
projectives in Table~\ref{tab:proj3} lead to upper bounds for these
remaining four decomposition numbers. However, these bounds are not
good enough for our purposes. Instead, we use the projectives in
Table~\ref{tab:addscal}. From this table, we get 
\[
\Psi_{13}' = \frac{q}{\sqrt{2}} \Phi_{13} + A \cdot \Phi_{19} + B
\cdot \Phi_{21} + \Phi,
\]
where $A, B$ are non-negative integers and $\Phi$ is a projective
character belonging to non-unipotent blocks. Thus, we get:
$\frac{q^2+\sqrt{2}q}{4} = \frac{q}{\sqrt{2}} \cdot h + A \ge
\frac{q}{\sqrt{2}} \cdot h$ and so $h \le \frac{\sqrt{2}q}{4} +
\frac{1}{2}$. Since $h$ is an integer, it follows 
$h \le \frac{\sqrt{2}q}{4}$. For $\Psi_{18}'$, we get 
\[
\Psi_{18}' = \frac{q}{\sqrt{2}} \Phi_7 + \frac{q}{\sqrt{2}} \Phi_{13}
+ \frac{q}{\sqrt{2}} \Phi_{18} + A \cdot \Phi_{19} + B \cdot \Phi_{20} 
+ C \cdot \Phi_{21} + \Phi,
\]
where $A, B, C$ are non-negative integers and $\Phi$ is a projective
character belonging to non-unipotent blocks. Then, we get:
$\frac{q^2+\sqrt{2}q}{4} \ge \frac{q}{\sqrt{2}} \cdot j$ which implies
$j \le \frac{\sqrt{2}q}{4}$. In a similar way, one can obtain the
upper bounds for $u$ and $w$ from $\Psi_{13}'$ and $\Psi_{18}'$. 
The lower bounds for the decomposition numbers $a, b, \dots, x$ follow
from the decomposition numbers of the non-unipotent characters in the
last five rows of Table~\ref{tab:decuni3}, since decomposition
numbers are non-negative. This proves all assertions about $\Phi_i$
for $i=9,10,11,12,13,14,17,18,19,20$ in the decomposition matrix 
Table~\ref{tab:decuni3}.

So, we still have to complete the decomposition numbers in columns
$2$, $3$ and $5$ of the decomposition matrix. The relation
corresponding to $\chi_{10,1}$ and the fact that decomposition numbers
are non-negative implies $(\chi_5, \Phi_2)_G=(\chi_5, \Phi_3)_G=1$.
Then, $j \ge 1$ implies that $\Phi_{18}$ cannot be subtracted from
$\Psi_2$, $\Psi_3$ and $\Psi_5$, and we get 
$(\chi_{18}, \Phi_2)_G=(\chi_{18}, \Phi_3)_G=(\chi_{18}, \Phi_5)_G=1$.
From the relations corresponding to
$\chi_{10,\frac{q}{\sqrt{2}}(q^2-1)_a}$ and
$\chi_{10,\frac{q}{\sqrt{2}}(q^2-1)_b}$ and $\chi_{10,St}$, we get
$(\chi_{19}, \Phi_2)_G=(\chi_{20}, \Phi_3)_G=(\chi_{21},
\Phi_2)_G=(\chi_{21}, \Phi_3)_G=1$. 

By \cite[Theorem~4.2]{GHMCusp}, the modular Steinberg character of $G$
is cuspidal. So, the construction of $\Psi_5$ in Table~\ref{tab:proj3}
implies that the projective cover of $\phi_{21}$ is not a summand of
the projective module corresponding to $\Psi_5$. Hence, 
$(\chi_{21}, \Phi_5)_G=1$. We have already seen that $\phi_1$,
$\phi_4$, $\phi_6$ and $\phi_7$ are the irreducible Brauer characters
in the principal series. By the construction of the projectives in
Table~\ref{tab:proj3}, it is clear that $\phi_2$, $\phi_3$, $\phi_5$
are the only irreducible Brauer characters in the Harish-Chandra
series ${^2B_2}[a]$, ${^2B_2}[b]$, ${^2B_2}[\text{St}]$, respectively. This
finishes the proof of Theorem~\ref{thm:decnum3} for the unipotent blocks.

Next, we deal with the non-unipotent blocks in Tables~\ref{tab:dec2},
\ref{tab:dec3} and \ref{tab:decg5cyc}-\ref{tab:decnilp}. Let $s \neq
1$ be a semisimple $\ell'$-element of $G^*$. If $s$ is not of type
$g_5$, then the decomposition numbers of $\mathcal{E}(G, s)$ are clear
by Bonnaf{\'e}'s and Rouquier's Jordan decomposition, see
Subsection~\ref{jordan}. This proves the decomposition numbers in
Tables~\ref{tab:dec2}, \ref{tab:dec3} and
\ref{tab:dec6}-\ref{tab:decnilp}. The decomposition numbers in
Table~\ref{tab:decg5cyc} follow from the fact that $\chi_{5,1}$, 
$\chi_{5,q^2(q^2-1)}$, $\chi_{5,St}$ have defect $0$. This completes
the proof of Theorem~\ref{thm:decnum3}. \hfill $\Box$


\subsection{Proof of Theorem~\ref{thm:decnum4}}

The proof is similar to the proof of Theorem~\ref{thm:decnum3}. So we
only give a brief sketch. Suppose $\ell \, | \, q^2-\sqrt{2}q+1$.
The numbers in the right most column of the decomposition matrix
Table~\ref{tab:decuni4} are obtained by counting the elements of
$\ell$-power order in~$G^*$. Under the assumptions of the theorem, we have
$\frac{1}{96}(\ell^f-1)(\ell^f-5)>0$ if and only if $\ell \neq 5$ or
$n \equiv 2$ or $n \equiv 17$ mod $20$ where $q^2=2^{2n+1}$. Thus, the
relation in the last row of Table~\ref{tab:rels4} exists if and only
if $\ell \neq 5$ or $n \equiv 2$ or $n \equiv 17$ mod $20$.

We construct projective characters $\Psi_1$, \dots, $\Psi_{21}$ of $G$ 
according to Table~\ref{tab:proj4} and compute the scalar products 
$(\chi_i, \Psi_j)_G$ for $1 \le i,j \le 21$ using CHEVIE. Here,
$\Phi_a$ and $\Phi_b$ are the characters of the
projective covers of the unipotent irreducible Brauer characters
$\varphi_a$ and $\varphi_b$ of $L_b$, respectively, and $\Phi_{St}$
is the character of the projective cover of the modular Steinberg
character $\varphi_{St}$ of $L_b$; see the comments in
Subsection~\ref{case4}. In this way, we obtain the approximation
Table~\ref{tab:scalar4} to the decomposition matrix of the unipotent
characters. The unitriangular shape gives a natural bijection between
the set of ordinary unipotent irreducible characters and the set of
irreducible Brauer characters in $\mathcal{E}_\ell(G, 1)$. For 
$1 \le i \le 21$, let $\phi_i$ be the irreducible Brauer character
corresponding to $\chi_i$ and $\Phi_i$ the character of the corresponding PIM.  

From Table~\ref{tab:scalar4}, we get the assertions about
$\Phi_i$ for $i=5,7,9,15,16,21$ in the decomposition matrix
Table~\ref{tab:decuni4}. Furthermore, \cite[Corollary~4.10]{Dipper}
and Proposition~\ref{prop:hecke_decnum} imply the assertions about
$\Phi_1$ and $\Phi_4$ and we see that $\phi_1$, $\phi_4$, $\phi_5$ and
$\phi_7$ are the irreducible Brauer character in the principal series.
The characters 
\[
\Psi_{11}' := \chi_3 \otimes \chi_9, \quad 
\Psi_{11}'':= {_{P_b}\chi_{25}}^G, \quad 
\Psi_{12}' := \chi_2 \otimes \chi_9, \quad 
\Psi_{18}' := \chi_2 \otimes \chi_5
\]
are projective since $\chi_5$, $\chi_9$ and ${_{P_b}\chi_{25}}$ have
defect~$0$. Using CHEVIE, we compute the scalar products of these
projectives with the ordinary unipotent characters of $G$, 
see Table~\ref{tab:addscal2}. All scalar products $(\Psi_i',\chi_j)$
and $(\Psi_{11}'',\chi_j)$ not listed in this table are zero. 

\begin{table}[!ht]
\caption[]{The scalar products $(\Psi_i', \chi_j)_G$ 
and $(\Psi_{11}'', \chi_j)_G$.}  
\label{tab:addscal2}

\begin{center}
\begin{tabular}{c|ccccccc} \hline
\rule{0cm}{0.45cm}
& $\chi_7$ & $\chi_{11}$ & $\chi_{12}$ & $\chi_{18}$ & $\chi_{19}$ & $\chi_{20}$ & $\chi_{21}$
\rule[-0.2cm]{0cm}{0.3cm}\\
\hline
\rule{-0.12cm}{0.4cm}
$\Psi_{11}'$ & $0$ & $\frac{q}{\sqrt{2}}$ & $0$ & $0$ &
$\frac{q^2-\sqrt{2}q}{4}$ & $0$ & $\frac{\sqrt{2}q(q^2-3\sqrt{2}q+4)}{24}$
\rule[-0.2cm]{0cm}{0.3cm}\\
\rule{-0.12cm}{0.4cm}
$\Psi_{11}''$ & $0$ & $1$ & $0$ & $0$ &
$\frac{q}{\sqrt{2}}$ & $0$ & $\frac{q^2-\sqrt{2}q}{4}$
\rule[-0.2cm]{0cm}{0.3cm}\\
\rule{-0.12cm}{0.4cm}
$\Psi_{12}'$ & $0$ & $0$ & $\frac{q}{\sqrt{2}}$ & $0$ & $0$ &
$\frac{q^2-\sqrt{2}q}{4}$ & $\frac{\sqrt{2}q(q^2-3\sqrt{2}q+4)}{24}$
\rule[-0.2cm]{0cm}{0.3cm}\\
\rule{-0.12cm}{0.4cm}
$\Psi_{18}'$ & $\frac{q}{\sqrt{2}}$ & $\frac{q}{\sqrt{2}}$ & $0$
& $\frac{q}{\sqrt{2}}$ & $\frac{q^2+2}{2}$ & $\frac{q^2-\sqrt{2}q}{4}$
& $\frac{\sqrt{2}q(q^2-\sqrt{2}q+4)}{8}$
\rule[-0.2cm]{0cm}{0.3cm}\\
\hline
\end{tabular}
\end{center}
\end{table}

The scalar products for $\Psi_{11}'$, $\Psi_{12}'$ imply: 
\[
(\chi_{18}, \Phi_{11})_G=(\chi_{20}, \Phi_{11})_G=(\chi_{18},
\Phi_{12})_G=(\chi_{19}, \Phi_{12})_G=0. 
\]
From the scalar products in Table~\ref{tab:scalar4} and
Table~\ref{tab:addscal2} we get the upper bounds for the decomposition 
numbers $a$, $b$, \dots, $v$, $w$ in the same way as in the proof of
Theorem~\ref{thm:decnum3}. 
The only difference is the upper bound for $x$ which can be derived as
follows: From Table~\ref{tab:addscal2}, we get 
\[
\Psi_{11}'' = \Phi_{11} + A \cdot \Phi_{19} + B \cdot \Phi_{21} +
\Phi, 
\]
where $A, B$ are non-negative integers and $\Phi$ is a projective
character belonging to non-unipotent blocks. So we get:
$e + A = \frac{q}{\sqrt{2}}$ and 
$t + A \cdot x + B = \frac{q^2-\sqrt{2}q}{4}$. The first equation and 
the upper bound $e \le \frac{\sqrt{2}q-4}{4}$ imply
$A \ge \frac{q}{2\sqrt{2}}+1$ and the the second equation gives:
\[
A \cdot x \le t + A \cdot x + B = \frac{q^2-\sqrt{2}q}{4}.
\]
Hence $x \le \frac{q^2-\sqrt{2}q}{4A} \le \frac{q}{\sqrt{2}} - 3 +
\frac{12}{\sqrt{2}q+4}$. Since~$x$ is an integer, it follows 
$x \le \frac{q}{\sqrt{2}}-2$. 
The proof of the remaining statements is analogous to the proof of 
Theorem~\ref{thm:decnum3}. \hfill $\Box$


\subsection{Proof of Theorem~\ref{thm:decnuml3}}

Suppose $\ell = 3$. As in the proof of Theorem~\ref{thm:decnum1}, the
numbers in the right most column of Table~\ref{tab:decunil3} are
obtained by counting the elements of $3$-power order in $G^*$. 
Note that under the assumptions of the theorem, we have
$\frac{1}{2}(3^f-3)>0$ if and only if $n \equiv 1$ or $n \equiv 4$
mod $6$ where $q^2=2^{2n+1}$. Consequently, the relations corresponding to
$\chi_{6,1}$ and $\chi_{6,St}$ in Table~\ref{tab:relsl3} exist if and
only if $n \equiv 1$ or $n \equiv 4$ mod~$6$.
Furthermore, $\frac{1}{48}(3^f-3)(3^f-9)>0$ 
if and only if $n \equiv 4$ or $n \equiv 13$ mod $18$. So the relation
in the last row of Table~\ref{tab:relsl3} exists if and 
only if $n \equiv 4$ or $n \equiv 13$ mod $18$.

We choose the ordinary basic set
\begin{equation*}
\{\chi_1, \chi_2, \chi_3, \chi_4, \chi_{5,1}, \chi_5, \chi_6, \chi_7, \chi_8, 
\chi_{10}, \chi_{11}, \chi_{12}, \chi_{13}, \chi_{14}, \chi_{15}, \chi_{18}, 
\chi_{19}, \chi_{20}, \chi_{21}\}
\end{equation*}
of $\mathcal{E}_3(G, 1)$ as in Section~\ref{sec:decompmatl3}.
Using the relations in Table~\ref{tab:relsl3}, the rows of the
decomposition matrix of $\mathcal{E}_\ell(G, 1)$ corresponding to
the irreducible characters not belonging to the above basic set can
be written as linear combinations of the rows corresponding to the
irreducible characters in the above basic set. So, we only have to  
determine the decomposition numbers of the irreducible characters in
the above basic set. 

We construct projective characters $\Psi_i$ and $\Psi_{5,1}$ of $G$
according to Table~\ref{tab:projl3} in the Appendix and compute the 
scalar products of these projective characters with the ordinary
characters in the above basic set using CHEVIE, see
Table~\ref{tab:scalarl3} in the Appendix. We already get 
from Table~\ref{tab:scalarl3} that the decomposition matrix of
$\mathcal{E}_\ell(G, 1)$ has a lower unitriangular shape giving us a
natural bijection between the above ordinary basic set and the set of
irreducible Brauer characters in $\mathcal{E}_\ell(G, 1)$.
Let $\phi_i$ be the irreducible Brauer character corresponding
to~$\chi_i$ and $\phi_{5,1}$ be the irreducible Brauer character 
corresponding to $\chi_{5,1}$. We write $\Phi_i$ and~$\Phi_{5,1}$ for
the corresponding PIMs, respectively.

From Table~\ref{tab:scalarl3}, we get the assertions about all
PIMs except for $\Phi_i$, $i = 1$, $4$, $7$, $8$, $10$, $15$, $18$ and
$\Phi_{5,1}$ in the decomposition matrix Table~\ref{tab:decunil3}.
Let $a$ be the multiplicity of $\chi_{18}$ in $\Phi_{15}$ and
$b,c,d,e$ the multiplicity of $\chi_{21}$ in $\Phi_8$, $\Phi_{10}$,
$\Phi_{15}$, $\Phi_{18}$, respectively. Furthermore, let $x_i$ be the
multiplicity of $\chi_i$ in $\Phi_{5,1}$ for $i=7$, $8$, $10$, $15$,
$18$, $21$. The assertions about $\Phi_{5,1}$ follow from $\Psi_{5,1}$ by
subtracting defect~$0$ characters. The entries in
Table~\ref{tab:scalarl3} lead to the upper bounds for
$a$, $b$, $c$, $d$, $e$, $x_i$ in Theorem~\ref{thm:decnuml3}, and the
lower bounds follow from the decomposition numbers of the
characters not in the basic set in the lower part of
Table~\ref{tab:decunil3}. So we have shown all assertions about 
$\Phi_{5,1}$, $\Phi_{10}$, $\Phi_{15}$, $\Phi_{18}$ in
Theorem~\ref{thm:decnuml3}. Furthermore, \cite[Corollary~4.10]{Dipper}  
and Proposition~\ref{prop:hecke_decnum} imply the assertions about
$\Phi_1$ and $\Phi_7$. 
So, we are only left with the decomposition numbers 
\[
(\chi_{5,1}, \Phi_4)_G, (\chi_{7}, \Phi_4)_G, (\chi_{21}, \Phi_4)_G,
(\chi_{10}, \Phi_8)_G \in \{0,1\}.
\]
The relation for $\chi_9$ in Table~\ref{tab:relsl3} and the fact that
decomposition numbers are non-negative imply 
$(\chi_{5,1}, \Phi_4)_G = (\chi_{10}, \Phi_8)_G = 1$. Since
$\Phi_7$ is not a summand of $\Psi_4$, we can deduce $(\chi_7, \Phi_4)_G=1$.
The remaining assertions about the unipotent blocks, in particular,
the modular Harish-Chandra series and $(\chi_{21}, \Phi_4)_G=1$ can be
shown analogously to the proof of Theorem~\ref{thm:decnum2}.

The decomposition numbers of the non-unipotent irreducible characters
of $G$ in Tables~\ref{tab:dec2}, \ref{tab:dec3},
\ref{tab:dec6}-\ref{tab:decnilp} follow from Bonnaf{\'e}'s and 
Rouquier's Jordan decomposition, see Subsection~\ref{jordan}. Note
that for $\ell=3$, one does not have to deal with blocks corresponding to
semisimple elements of type $g_5$. This completes the proof of
Theorem~\ref{thm:decnuml3}. \hfill $\Box$


\section{Degrees of irreducible representations}
\label{sec:degrees}

For a finite group $S$ and an algebraically closed field $k$ of
characteristic $\ell \ge 0$, we write $d_\ell(S)$ for the smallest
degree of a nonlinear irreducible $kS$-representation. 
In~\cite[Problem~1.1]{TiepDurham}, P.~Tiep proposed the following
problem: 

\medskip

\noindent \textit{Given a finite quasisimple group $S$ and $\ell$,
determine $d_\ell(S)$ and all nontrivial irreducible
$kS$-representations of degree $d_\ell(S)$.}

\medskip

\noindent For related problems and applications of smallest degrees
see~\cite{TiepDurham} and \cite{TiepZalesskii}. 

\subsection{Smallest degrees of the Ree groups}
In this section, we consider $d_\ell(G)$ for the simple Ree groups 
$G = {^2F_4}(q^2)$, $q^2=2^{2n+1}$, $n > 0$.  
Of course, due to G.~Malle's ordinary character table, the smallest
degree $d_\ell(G)$ is known for $\ell = 0$ and all primes 
$\ell$ not dividing the order of $G$. For the defining 
characteristic, one has $d_2(G) = 26$, see~\cite{LuebeckSmlDegDef}. 
Using the Brauer trees in~\cite{HissTrees2F4},
F.~L\"ubeck~\cite{LuebeckSmlDeg} showed that $d_\ell(G)$ coincides
with $d_0(G)$ for all primes $\ell$ such that the Sylow
$\ell$-subgroups of $G$ are cyclic. In Theorem~\ref{thm:smalldeg}, we
are going to extend this result to all primes $\ell > 3$. 

\begin{theorem} \label{thm:smalldeg}
Let $G = {^2F_4}(q^2)$, $q^2=2^{2n+1}$, $n > 0$. For every prime
$\ell > 3$ one has
\[
d_\ell(G) = d_0(G) = \frac{q}{\sqrt{2}}(q^2-1)(q^2+1)^2(q^4-q^2+1)
\]
and up to isomorphism, there are two $kG$-modules of this
degree. Their Brauer characters are $\breve{\chi}_2$ and
$\breve{\chi}_3$.
\end{theorem}

Before we start with the proof of the theorem, we describe the general
strategy: The decomposition matrices in Appendices~C and D show that 
$\breve{\chi}_2$ and $\breve{\chi}_3$ are irreducible Brauer
characters of degree $d_0(G)$. Furthermore, these decomposition
matrices determine almost all irreducible Brauer characters of $G$ and
in particular their degrees, see Remark~\ref{sub:rmkdecmat}~(c).  
We show that the remaining unknown degrees are larger than $d_0(G)$.  
To achieve this, we use techniques similar to those for Steinberg's
triality groups in~\cite{Himstedt3D4Decomp}: The lower and upper
bounds for the unknown decomposition numbers in
Theorems~\ref{thm:decnum1}, \ref{thm:decnum2}, \ref{thm:decnum3},
\ref{thm:decnum4} lead to lower bounds for the unknown degrees and
eventually show, that these degrees are larger than $d_0(G)$.

However, there are significant differences to~\cite{Himstedt3D4Decomp}
which are due to the fact that the degrees of the ordinary unipotent
characters $\chi_{18}$, $\chi_{19}$, $\chi_{20}$, $\chi_{21}$
are ``asymptotically close together''. Let us have a closer
look at the case $\ell \, | \, q^2+\sqrt{2}q+1$ and the unknown
degree of the modular Steinberg character $\phi_{21}$. From the
decomposition matrix Table~\ref{tab:decuni3}, we get:
\begin{equation} \label{eq:lowerbound}
\deg(\phi_{21}) = \chi_{21}(1) - x \cdot (\phi_{19}(1)+\phi_{20}(1)) - 
w \cdot \phi_{18}(1) - v \cdot \phi_{17}(1) - \dots - \phi_{2}(1).
\end{equation}
By the decomposition matrix Table~\ref{tab:decuni3} we know the degrees
$\phi_1(1), \dots, \phi_{17}(1)$. Assume that we have already
proved sufficiently good lower and upper bounds for the unknown degrees
$\phi_{18}(1)$, $\phi_{19}(1)$, $\phi_{20}(1)$. Plugging in the upper
bounds for the decomposition numbers $r$, $s$, $t$, $u$, $v$, $w$, $x$
given in Theorem~\ref{thm:decnum3} and the upper bounds 
for $\phi_{18}(1)$, $\phi_{19}(1)$, $\phi_{20}(1)$ into the right hand
side of (\ref{eq:lowerbound}) leads to a lower bound for 
$\deg(\phi_{21})$. Unfortunately, this bound turns out to be negative
and does not give any information at all. To overcome this 
difficulty, we use certain dependencies between the various
decomposition numbers which are derived from projective characters.

\begin{proof}
Because of the Brauer trees in~\cite{HissHabil}, 
\cite{HissTrees2F4} and F.~L\"ubeck's result \cite{LuebeckSmlDeg}, we
only have to consider the cases
\[
\ell \, | \, q^2-1, \quad \ell \, | \, q^2+1,  \quad
\ell \, | \, q^2+\sqrt{2}q+1 , \quad \ell \, | \, q^2-\sqrt{2}q+1. 
\]
We only demonstrate the case $\ell \, | \, q^2+\sqrt{2}q+1$. The case 
$\ell \, | \, q^2-\sqrt{2}q+1$ is similar, the other cases are much
easier. 
Suppose $\ell$ is a prime dividing $q^2+\sqrt{2}q+1$. In this case,
Theorem~\ref{thm:decnum3} determines all decomposition numbers of $G$,
except for several decomposition numbers in the ordinary characters 
$\chi_{18}$, $\chi_{19}$, $\chi_{20}$, $\chi_{21}$. Thanks to the
unitriangular shape of the decomposition matrix, this gives us the
degrees of all irreducible Brauer characters of~$G$ except for 
$\phi_{18}(1)$, $\phi_{19}(1)$, $\phi_{20}(1)$, $\phi_{21}(1)$. We see
$\phi_2(1) = \phi_3(1) = d_0(G)$ and that all other known degrees are
strictly bigger than $d_0(G)$.
So, for the four unknown degrees we have to show $\phi_{18}(1)$,
$\phi_{19}(1)$, $\phi_{20}(1)$, $\phi_{21}(1) > d_0(G)$. From the decomposition matrix
Table~\ref{tab:decuni3}, we get 
\begin{equation} \label{eq:deg18}
\phi_{18}(1) = \chi_{18}(1) 
- \sum_{i=1}^3 \phi_{i}(1)
- \phi_{5}(1)
- a \cdot \phi_{9}(1)
- b \cdot (\phi_{11}(1)+\phi_{12}(1))
- c \cdot \phi_{17}(1).
\end{equation}
Plugging in the known degrees and the upper bounds for $a$, $b$, $c$
in Theorem~\ref{thm:decnum3} into the right hand side of
(\ref{eq:deg18}), we see $\deg(\phi_{18}) > d_0(G)$. Furthermore, 
plugging in $a=b=c=0$ into (\ref{eq:deg18}) we get an upper bound for
$\deg(\phi_{18})$. The proof of $\phi_{19}(1) > d_0(G)$ is similar:
From the decomposition matrix Table~\ref{tab:decuni3}, we get 
\begin{equation*}
\phi_{19}(1) = \chi_{19}(1) 
- \phi_2(1)
- d \cdot \phi_{9}(1)
- e \cdot \phi_{11}(1)
- g \cdot \phi_{12}(1)
- h \cdot \phi_{13}(1)
- i \cdot \phi_{17}(1)
- j \cdot \phi_{18}(1).
\end{equation*}
Plugging in the known degrees and the upper bounds for $d$, $e$, $g$, 
$h$, $i$, $j$ in Theorem~\ref{thm:decnum3} and the upper bound for
$\phi_{18}(1)$, we obtain 
$\deg(\phi_{19}) = \deg(\phi_{20}) > d_0(G)$. Note that $\phi_{19}$
and $\phi_{20}$ are complex conjugate to each other.

Next, we show $\deg(\phi_{21}) > d_0(G)$ which is by far the most
difficult part of the proof. We recall that the degrees of the
ordinary unipotent characters $\chi_1, \dots, \chi_{21}$ are known and
are given by polynomials in $q$. Thanks to the unitriangular shape of
the decomposition matrix~\ref{tab:decuni3}, we can express
$\deg(\phi_{21})$ as a polynomial in $q$ and the unknown decomposition
numbers $a$, $b$, $c$, \dots, $w$, $x$ and get:
\begin{eqnarray} \label{eq:d21}
\deg(\phi_{21}) & = & q^{24} -\sqrt{2} x q^{23} +(2xj\mathbf{-w}) q^{22}
\nonumber \\
&& +\left(\frac{xh}{2}
-\frac{1}{4}
+\frac{wa}{12}
\mathbf{-\frac{t}{2}}
+\frac{wb}{2}
-\frac{2xjc}{3}
+\frac{xe}{2}
+\frac{xd}{6}
-xjb \right.\\
&&\left. +\frac{2xi}{3}
\mathbf{-\frac{r}{12}}
-\frac{5xj}{2}
+\frac{xg}{2}
+\frac{5w}{4}
+\frac{wc}{3}
\mathbf{-\frac{u}{2}}
-\frac{xja}{6}
-\frac{s}{6}
\mathbf{-\frac{v}{3}}\right)q^{20} \nonumber \\
&& +P(q, a, b, c \dots, w, x), \nonumber
\end{eqnarray}
where $P(q, a, b, c \dots, w, x)$ is a polynomial in $q$ and the
unknown decomposition numbers $a$, $b$, \dots, $w$, $x$ with
coefficients in $\Q[\sqrt{2}]$ such that the degree of 
$P(q, a, b, c \dots, w, x)$, when considered as a polynomial in $q$,
is at most $19$. Note that $P(q, a, b, c \dots, w, x)$ can be given
explicitly. 

Next, we prove several inequalities between the unknown decomposition
numbers, which allow us to bound the bold face expressions
in~(\ref{eq:d21}) from below. We are going to use some of the
projective characters $\Psi_i$ and $\Psi_j'$ constructed in the proof
of Theorem~\ref{thm:decnum3}. Scalar products of these projective
characters with some ordinary characters are given in
Tables~\ref{tab:addscal} and \ref{tab:scalar3}. We write $\Phi_i$ for
the character of the PIM corresponding to $\phi_i$.

From Table~\ref{tab:addscal}, we get 
$\Psi_{13}' = \frac{q}{\sqrt{2}} \Phi_{13} + A \cdot \Phi_{19} + B
\cdot \Phi_{21} + \Phi$, where $A, B$ are non-negative integers and
$\Phi$ is a projective character belonging to non-unipotent blocks. Thus, we get:
$\frac{q}{\sqrt{2}} \cdot h + A = \frac{q^2+\sqrt{2}q}{4}$ and 
$\frac{q}{\sqrt{2}} \cdot u + A \cdot x + B = \frac{\sqrt{2}q(q^2+3\sqrt{2}q+4)}{24}$.
The first equation implies $A = \frac{q^2+\sqrt{2}q}{4} -
\frac{q}{\sqrt{2}} \cdot h$ and from the second equation, we then get 
\[
\frac{q}{\sqrt{2}} \cdot u + A \cdot x \le 
\frac{q}{\sqrt{2}} \cdot u + A \cdot x + B =
\frac{\sqrt{2}q(q^2+3\sqrt{2}q+4)}{24}. 
\]
So
\begin{equation} \label{eq:lb_u}
-\frac{u}{2} \ge \frac{\sqrt{2}Ax}{2q} -
\frac{q^2+3\sqrt{2}q+4}{24} = 
\frac{\sqrt{2}qx}{8}+\frac{x}{4}-\frac{xh}{2}-\frac{q^2}{24}-\frac{\sqrt{2}q}{8}-\frac{1}{6}. 
\end{equation}
Analogously, using the projective characters $\Psi_{18}'$,
$\Psi_{11}$, $\Psi_{9}$, $\Psi_{17}$, we obtain:
\begin{eqnarray*}
-w & \ge &
u+\frac{3\sqrt{2}qx}{4}+\frac{x}{2}-xh-2xj+\frac{\sqrt{2}x}{q}-\frac{q^2}{4}-\frac{\sqrt{2}q}{4}-1,
\end{eqnarray*}

\begin{eqnarray} \label{eq:decineq}
-\frac{t}{2} & \ge &
\frac{wq^2}{8}+\frac{w\sqrt{2}q}{8}-\frac{wb}{2}+\frac{x\sqrt{2}q^3}{8}-\frac{xe}{2}-\frac{xjq^2}{4}-\frac{xj\sqrt{2}q}{4} \nonumber \\
&& +xjb-\frac{xg}{2}-\frac{q^4}{8},\nonumber \\
-\frac{r}{12} & \ge &
\frac{wq^2}{144}+\frac{w\sqrt{2}q}{48}+\frac{w}{36}-\frac{wa}{12}+\frac{x\sqrt{2}q^3}{144}-\frac{x\sqrt{2}q}{72}-\frac{xd}{6}\\
&&
-\frac{xjq^2}{72}-\frac{xj\sqrt{2}q}{24}-\frac{xj}{18}+\frac{xja}{6}-\frac{q^4}{144}+\frac{1}{36},\nonumber \\
-\frac{v}{3} & \ge &
\frac{wq^2}{9}-\frac{2w}{9}-\frac{wc}{3}+\frac{x\sqrt{2}q^3}{9}+\frac{x\sqrt{2}q}{9}-\frac{2xi}{3}-\frac{2xjq^2}{9}\nonumber \\
&&+\frac{4xj}{9}+\frac{2xjc}{3}-\frac{q^4}{9}-\frac{2}{9}.\nonumber 
\end{eqnarray}
By replacing the bold face $-w$, $-\frac{t}{2}$, $-\frac{r}{12}$,
$-\frac{u}{2}$, $-\frac{v}{3}$ in (\ref{eq:d21}) by the right hand
sides of~(\ref{eq:lb_u}) and (\ref{eq:decineq}), we obtain a lower
bound for $\deg(\phi_{21})$. Again, this lower bound can be written as a
polynomial $Q(q,a,b,c,\dots,w,x)$ in $q$ and $a$, $b$, \dots, $x$
with coefficients in $\Q(\sqrt{2})$. We consider two cases:

\smallskip

Case 1: $n>1$. Expand the polynomial $Q(q,a,b,c,\dots,w,x)$,
and in all terms with a positive coefficient, replace $a$, $b$,
\dots, $x$ by the lower bounds in Theorem~\ref{thm:decnum3}; in all  
terms with a negative coefficient, replace $a$, $b$, \dots, 
$x$ by the upper bounds in Theorem~\ref{thm:decnum3}.
In this way, we obtain another lower bound for $\deg(\phi_{21})$ which
is now a polynomial in $q$ only, and from this bound we get
$\deg(\phi_{21}) > d_0(G)$. 

\smallskip

Case 2: $n=1$, that is $G = {^2F_4}(8)$. First, substitute $q$
by $\sqrt{8}$ in $Q(q,a,b,c,\dots,w,x)$. Then, expand this polynomial
in $a$, $b$, \dots, $w$, $x$, and in all terms with a positive
coefficient, replace $a$, $b$, \dots, $x$ by the lower bounds 
in Theorem~\ref{thm:decnum3}, in all terms with a negative coefficient,
replace $a$, $b$, \dots, $x$ by the upper bounds in
Theorem~\ref{thm:decnum3}. Note that by Theorem~\ref{thm:decnum3} and
Corollary~\ref{cor:decnum3q8}, one can use $h=j=1$, $x=2$ 
and $s \ge 2$ in this case. In this way, we obtain: 
\[
\deg(\phi_{21}) \ge \frac{11769507827}{3} > 64638 = d_0(G).
\]
So, all unknown degrees of irreducible Brauer characters are strictly
bigger than $d_0(G)$ and the claim follows.
\end{proof}

\subsection{Remarks on Theorem~\ref{thm:smalldeg}} \label{rmk:smalldeg}

\begin{enumerate}
\item[(a)] For fixed $q$, the bounds in Theorem~\ref{thm:decnuml3}
imply that the degrees of all nontrivial $3$-modular irreducible Brauer
characters $\neq \breve{\chi}_2$, $\breve{\chi}_3$ of $G$ are larger than
$d_0(G)$, except for possibly $\deg(\phi_{18})$ or $\deg(\phi_{21})$.

\item[(b)] Theorem~\ref{thm:smalldeg} improves the bounds of
V.~Landazuri, G.~Seitz, A.~Zalesskii, P.~Tiep in~\cite{LS}, \cite{SZ},
\cite{TiepGrass} for $\ell > 3$. 

\end{enumerate}

\newpage


\section*{Appendix A: Notation for characters}
\label{sec:appendixA}

\setcounter{table}{0}
\renewcommand{\thetable}{\textrm{A.\arabic{table}}}

\begin{table}[!ht]
\caption[]{Notation, degrees and families of the unipotent irreducible
characters of $G={^2F}_4(q^2)$. The notation we use is the CHEVIE
notation in the left most column. For a definition of $\phi_j$,
$\phi_j'$, $\phi_j''$ see Subsection~\ref{grpsetup}; see
also~\cite[Table~4]{HimstedtHuang2F4MaxParab}.} \label{tab:unicharsG}   

\begin{center}
\begin{tabular}{llllllc} \hline
\multicolumn{4}{c}{Notation} & Degree & Conj. & Value\\
CHEVIE & in \cite{Carter2} & in \cite{HissTrees2F4} & in
\cite{MalleUni2F4} \hspace{-0.25cm} & & Class &
\rule[-0.1cm]{0cm}{0.5cm}\\
\hline
\rule{0cm}{0.5cm}
${\chi}_1$ & $1$ & ${\xi}_1$ & ${\chi}_1$ & $1$ &&
\rule[-0.2cm]{0cm}{0.5cm}\\
\hline
\rule{0cm}{0.4cm}
${\chi}_2$    & ${^2B}_2[a],1$ & $\xi_5$  & $\chi_5$  &
$\frac{q}{\sqrt{2}}\phi_1\phi_2\phi_4^2\phi_{12}$ & $c_{1,11}$ & $-\frac{q}{\sqrt{2}}+\epsilon_4q^2$
\rule[-0.2cm]{0cm}{0.5cm}\\
\rule{0cm}{0.5cm}
${\chi}_3$ & ${^2B}_2[b],1$ & $\xi_6$ & $\chi_6$ & $\frac{q}{\sqrt{2}}\phi_1\phi_2\phi_4^{2}\phi_{12}$ & $c_{1,11}$ & $-\frac{q}{\sqrt{2}}-\epsilon_4q^2$
\rule[-0.2cm]{0cm}{0.5cm}\\
\hline
\rule{0cm}{0.4cm}
${\chi}_4$    & $\epsilon'$ & $\xi_{2}$ & $\chi_{2}$ & $q^{2}\phi_{12}\phi_{24}$ &&
\rule[-0.2cm]{0cm}{0.5cm}\\
\hline
\rule{0cm}{0.5cm}
${\chi}_5$    & $\rho_2'$ & $\xi_{9}$ & $\chi_{9}$ & $\frac{q^{4}}{4}\phi_4^{2}\phi_8''^2\phi_{12}\phi_{24}'$&&
\rule[-0.2cm]{0cm}{0.5cm}\\
\rule{0cm}{0.5cm}
${\chi}_6$    & $\rho_2''$ & $\xi_{10}$ & $\chi_{10}$ & $\frac{q^{4}}{4}\phi_4^{2}\phi_8'^2\phi_{12}\phi_{24}''$&&
\rule[-0.2cm]{0cm}{0.5cm}\\
\rule{0cm}{0.5cm}
${\chi}_7$    & $\rho_2$ & $\xi_{11}$ & $\chi_{11}$ & $\frac{q^{4}}{2}\phi_8^{2}\phi_{24}$ &&
\rule[-0.2cm]{0cm}{0.5cm}\\
\rule{0cm}{0.5cm}
${\chi}_8$    & cusp & $\xi_{12}$ & $\chi_{12}$ & $\frac{q^4}{12}\phi_1^{2}\phi_2^{2}\phi_8'^2\phi_{12}\phi_{24}'$ &&
\rule[-0.2cm]{0cm}{0.5cm}\\
\rule{0cm}{0.5cm}
${\chi}_9$    & cusp & $\xi_{13}$ & $\chi_{13}$ & $\frac{q^4}{12}\phi_1^{2}\phi_2^{2}\phi_8''^2\phi_{12}\phi_{24}''$ &&
\rule[-0.2cm]{0cm}{0.5cm}\\
\rule{0cm}{0.5cm}
${\chi}_{10}$ & cusp & $\xi_{14}$ & $\chi_{14}$ & $\frac{q^{4}}{6}\phi_1^{2}\phi_2^{2}\phi_4^{2}\phi_{24}$ &&
\rule[-0.2cm]{0cm}{0.5cm}\\
\rule{0cm}{0.5cm}
${\chi}_{11}$ & cusp & $\xi_{15}$ & $\chi_{15}$ &
$\frac{q^{4}}{4}\phi_1^{2}\phi_2^{2}\phi_4^{2}\phi_{12}\phi_{24}''$ & $c_{1,11}$ &
$-\frac{q^4}{4}-\epsilon_4\frac{q^3}{\sqrt{2}}$
\rule[-0.2cm]{0cm}{0.5cm}\\
\rule{0cm}{0.5cm}
${\chi}_{12}$ & cusp & $\xi_{16}$ & $\chi_{16}$ &
$\frac{q^{4}}{4}\phi_1^{2}\phi_2^{2}\phi_4^{2}\phi_{12}\phi_{24}''$ & $c_{1,11}$ &
$-\frac{q^4}{4}+\epsilon_4\frac{q^3}{\sqrt{2}}$
\rule[-0.2cm]{0cm}{0.5cm}\\
\rule{0cm}{0.5cm}
${\chi}_{13}$ & cusp & $\xi_{17}$ & $\chi_{17}$ &
$\frac{q^{4}}{4}\phi_1^{2}\phi_2^{2}\phi_4^{2}\phi_{12}\phi_{24}'$ &
$c_{1,11}$ & $-\frac{q^4}{4}-\epsilon_4\frac{q^3}{\sqrt{2}}$
\rule[-0.2cm]{0cm}{0.5cm}\\
\rule{0cm}{0.5cm}
${\chi}_{14}$ & cusp & $\xi_{18}$ & $\chi_{18}$ &
$\frac{q^{4}}{4}\phi_1^{2}\phi_2^{2}\phi_4^{2}\phi_{12}\phi_{24}'$ &
$c_{1,11}$ & $-\frac{q^4}{4}+\epsilon_4\frac{q^3}{\sqrt{2}}$ 
\rule[-0.2cm]{0cm}{0.5cm}\\
\rule{0cm}{0.5cm}
${\chi}_{15}$ & cusp & $\xi_{19}$ & $\chi_{19}$ &
$\frac{q^{4}}{3}\phi_1^{2}\phi_2^{2}\phi_4^{2}\phi_8^{2}$ & $c_{5,3}$ &
$\frac{q^2}{6}-\frac{1}{3} +\epsilon_4\frac{\sqrt{3}q^2}{2}$ \hspace{-0.2cm}
\rule[-0.2cm]{0cm}{0.5cm}\\
\rule{0cm}{0.5cm}
${\chi}_{16}$ & cusp & $\xi_{20}$ & $\chi_{20}$ & $\frac{q^{4}}{3}\phi_1^{2}\phi_2^{2}\phi_4^{2}\phi_8^{2}$ & $c_{5,3}$ &
$\frac{q^2}{6}-\frac{1}{3} -\epsilon_4\frac{\sqrt{3}q^2}{2}$ \hspace{-0.2cm}
\rule[-0.2cm]{0cm}{0.5cm}\\
\rule{0cm}{0.5cm}
${\chi}_{17}$ & cusp & $\xi_{21}$ & $\chi_{21}$ & $\frac{q^{4}}{3}\phi_1^{2}\phi_2^{2}\phi_{12}\phi_{24}$ &&
\rule[-0.2cm]{0cm}{0.5cm}\\
\hline
\rule{0cm}{0.4cm}
${\chi}_{18}$ & $\epsilon''$ & $\xi_{3}$ & $\chi_{3}$ & $q^{10}\phi_{12}\phi_{24}$ &&
\rule[-0.2cm]{0cm}{0.5cm}\\
\hline
\rule{0cm}{0.4cm}
${\chi}_{19}$ & ${^2B}_2[a],\epsilon$ & $\xi_{7}$ & $\chi_{7}$ &
$\frac{q^{13}}{\sqrt{2}}\phi_1\phi_2\phi_4^{2}\phi_{12}$ & $c_{1,3}$ &
$\epsilon_4\frac{q^9}{\sqrt{2}}$
\rule[-0.2cm]{0cm}{0.5cm}\\
\rule{0cm}{0.5cm}
${\chi}_{20}$ & ${^2B}_2[b],\epsilon$ & $\xi_{8}$ & $\chi_{8}$ &
$\frac{q^{13}}{\sqrt{2}}\phi_1\phi_2\phi_4^{2}\phi_{12}$ & $c_{1,3}$ &
$-\epsilon_4\frac{q^9}{\sqrt{2}}$
\rule[-0.2cm]{0cm}{0.5cm}\\
\hline
\rule{0cm}{0.4cm}
${\chi}_{21}$ & $\epsilon$ & $\xi_{4}$ & $\chi_{4}$ & $q^{24}$ &&
\rule[-0.2cm]{0cm}{0.3cm}\\
\hline
\end{tabular}
\end{center}
\end{table}

\newpage


\begin{table}[!ht]
\caption[]{Notation, degrees and numbers of the non-unipotent
  irreducible characters of $G={^2F}_4(q^2)$. Dependencies on
  parameters $k,l$ are omitted. For a definition of $\phi_j$,
  $\phi_j'$, $\phi_j''$, see Subsection~\ref{grpsetup}.} \label{tab:nonuniG} 

\begin{center}
\begin{tabular}{llll} \hline
\rule{0cm}{0.36cm}
Character & Notation & Degree & Number of \hspace{-0.2cm}\\
& in CHEVIE && Characters
\rule[-0.1cm]{0cm}{0.36cm}\\
\hline
\rule{0cm}{0.36cm}
$\chi_{2,1}$ &
${\chi}_{22}$ & 
$\phi_4^{2}\phi_8\phi_{12}\phi_{24}$ & 
$\frac{1}{2}(q^2-2)$
\rule[-0.2cm]{0cm}{0.36cm}\\
\rule{0cm}{0.36cm}
$\chi_{2,\frac{q}{\sqrt{2}}(q^2-1)_a}$ &
${\chi}_{23}$ & 
$\frac{q}{\sqrt{2}}\phi_1\phi_2\phi_4^{2}\phi_8\phi_{12}\phi_{24}$ & 
\rule[-0.2cm]{0cm}{0.36cm}\\
\rule{0cm}{0.36cm}
$\chi_{2,\frac{q}{\sqrt{2}}(q^2-1)_b}$ &
${\chi}_{24}$ & 
$\frac{q}{\sqrt{2}}\phi_1\phi_2\phi_4^{2}\phi_8\phi_{12}\phi_{24}$ & 
\rule[-0.2cm]{0cm}{0.36cm}\\
\rule{0cm}{0.36cm}
$\chi_{2,St}$ &
${\chi}_{25}$ & 
$q^{4}\phi_4^{2}\phi_8\phi_{12}\phi_{24}$ & 
\rule[-0.2cm]{0cm}{0.36cm}\\
\hline
\rule{0cm}{0.36cm}
$\chi_{3,1}$ &
${\chi}_{26}$ & 
$\phi_4\phi_8^{2}\phi_{12}\phi_{24}$ & 
$\frac{1}{2}(q^2-2)$
\rule[-0.2cm]{0cm}{0.36cm}\\
\rule{0cm}{0.36cm}
$\chi_{3,St}$ &
${\chi}_{27}$ & 
$q^{2}\phi_4\phi_8^{2}\phi_{12}\phi_{24}$ & 
\rule[-0.2cm]{0cm}{0.36cm}\\
\hline
\rule{0cm}{0.36cm}
$\chi_{4,1}$ &
${\chi}_{28}$ & 
$\phi_4^{2}\phi_8^{2}\phi_{12}\phi_{24}$ & 
$\frac{1}{16}(q^4-10q^2+16)$
\rule[-0.2cm]{0cm}{0.36cm}\\
\hline
\rule{0cm}{0.36cm}
$\chi_{5,1}$ &
${\chi}_{29}$ & 
$\phi_1\phi_2\phi_8^{2}\phi_{24}$ & 
$1$
\rule[-0.2cm]{0cm}{0.36cm}\\
\rule{0cm}{0.36cm}
$\chi_{5,q^2(q^2-1)}$ &
${\chi}_{30}$ & 
$q^{2}\phi_1^{2}\phi_2^{2}\phi_8^{2}\phi_{24}$ & 
\rule[-0.2cm]{0cm}{0.36cm}\\
\rule{0cm}{0.36cm}
$\chi_{5,St}$ &
${\chi}_{31}$ & 
$q^{6}\phi_1\phi_2\phi_8^{2}\phi_{24}$ & 
\rule[-0.2cm]{0cm}{0.36cm}\\
\hline
\rule{0cm}{0.36cm}
$\chi_{6,1}$ &
${\chi}_{32}$ & 
$\phi_1\phi_2\phi_8^{2}\phi_{12}\phi_{24}$ & 
$\frac{1}{2}(q^2-2)$
\rule[-0.2cm]{0cm}{0.36cm}\\
\rule{0cm}{0.36cm}
$\chi_{6,St}$ &
${\chi}_{33}$ & 
$q^{2}\phi_1\phi_2\phi_8^{2}\phi_{12}\phi_{24}$ & 
\rule[-0.2cm]{0cm}{0.36cm}\\
\hline
\rule{0cm}{0.36cm}
$\chi_{7,1}$ &
${\chi}_{34}$ & 
$\phi_1\phi_2\phi_4\phi_8^{2}\phi_{12}\phi_{24}$ & 
$\frac{1}{4}(q^4-2q^2)$
\rule[-0.2cm]{0cm}{0.36cm}\\
\hline
\rule{0cm}{0.36cm}
$\chi_{8,1}$ &
${\chi}_{42}$ & 
$\phi_1\phi_2\phi_4^{2}\phi_8'\phi_{12}\phi_{24}$ & 
$\frac{1}{4}(q^2-\sqrt{2}q)$
\rule[-0.2cm]{0cm}{0.36cm}\\
\rule{0cm}{0.36cm}
$\chi_{8,\frac{q}{\sqrt{2}}(q^2-1)_a}$ &
${\chi}_{43}$ & 
$\frac{q}{\sqrt{2}}\phi_1^{2}\phi_2^{2}\phi_4^{2}\phi_8'\phi_{12}\phi_{24}$ & 
\rule[-0.2cm]{0cm}{0.36cm}\\
\rule{0cm}{0.36cm}
$\chi_{8,\frac{q}{\sqrt{2}}(q^2-1)_b}$ &
${\chi}_{44}$ & 
$\frac{q}{\sqrt{2}}\phi_1^{2}\phi_2^{2}\phi_4^{2}\phi_8'\phi_{12}\phi_{24}$ & 
\rule[-0.2cm]{0cm}{0.36cm}\\
\rule{0cm}{0.36cm}
$\chi_{8,St}$ &
${\chi}_{45}$ & 
$q^{4}\phi_1\phi_2\phi_4^{2}\phi_8'\phi_{12}\phi_{24}$ & 
\rule[-0.2cm]{0cm}{0.36cm}\\
\hline
\rule{0cm}{0.36cm}
$\chi_{9,1}$ &
${\chi}_{35}$ & 
$\phi_1\phi_2\phi_4^{2}\phi_8\phi_8'\phi_{12}\phi_{24}$ & 
$\frac{1}{8}(q^4-\sqrt{2}q^3-2q^2+2\sqrt{2}q)$
\rule[-0.2cm]{0cm}{0.36cm}\\
\hline
\rule{0cm}{0.36cm}
$\chi_{10,1}$ &
${\chi}_{46}$ & 
$\phi_1\phi_2\phi_4^{2}\phi_8''\phi_{12}\phi_{24}$ & 
$\frac{1}{4}(q^2+\sqrt{2}q)$
\rule[-0.2cm]{0cm}{0.36cm}\\
\rule{0cm}{0.36cm}
$\chi_{10,\frac{q}{\sqrt{2}}(q^2-1)_a}$ \hspace{-0.2cm} &
${\chi}_{47}$ & 
$\frac{q}{\sqrt{2}}\phi_1^{2}\phi_2^{2}\phi_4^{2}\phi_8''\phi_{12}\phi_{24}$ & 
\rule[-0.2cm]{0cm}{0.36cm}\\
\rule{0cm}{0.36cm}
$\chi_{10,\frac{q}{\sqrt{2}}(q^2-1)_b}$ \hspace{-0.2cm} &
${\chi}_{48}$ & 
$\frac{q}{\sqrt{2}}\phi_1^{2}\phi_2^{2}\phi_4^{2}\phi_8''\phi_{12}\phi_{24}$ & 
\rule[-0.2cm]{0cm}{0.36cm}\\
\rule{0cm}{0.36cm}
$\chi_{10,St}$ &
${\chi}_{49}$ & 
$q^{4}\phi_1\phi_2\phi_4^{2}\phi_8''\phi_{12}\phi_{24}$ & 
\rule[-0.2cm]{0cm}{0.36cm}\\
\hline
\rule{0cm}{0.36cm}
$\chi_{11,1}$ &
${\chi}_{36}$ & 
$\phi_1\phi_2\phi_4^{2}\phi_8\phi_8''\phi_{12}\phi_{24}$ & 
$\frac{1}{8}(q^4+\sqrt{2}q^3-2q^2-2\sqrt{2}q)$
\rule[-0.2cm]{0cm}{0.36cm}\\
\hline
\rule{0cm}{0.36cm}
$\chi_{12,1}$ &
${\chi}_{37}$ & 
$\phi_1^{2}\phi_2^{2}\phi_4^{2}\phi_8\phi_{12}\phi_{24}$ & 
$\frac{1}{16}(q^4-2q^2)$
\rule[-0.2cm]{0cm}{0.36cm}\\
\hline
\rule{0cm}{0.36cm}
$\chi_{13,1}$ &
${\chi}_{50}$ & 
$\phi_1^{2}\phi_2^{2}\phi_4^{2}\phi_8'^{2}\phi_{12}\phi_{24}$ & 
$\frac{1}{96}(q^4-2\sqrt{2}q^3-2q^2+4\sqrt{2}q)$
\rule[-0.2cm]{0cm}{0.36cm}\\
\hline
\rule{0cm}{0.36cm}
$\chi_{14,1}$ &
${\chi}_{51}$ & 
$\phi_1^{2}\phi_2^{2}\phi_4^{2}\phi_8''^{2}\phi_{12}\phi_{24}$ & 
$\frac{1}{96}(q^4+2\sqrt{2}q^3-2q^2-4\sqrt{2}q)$
\rule[-0.2cm]{0cm}{0.36cm}\\
\hline
\rule{0cm}{0.36cm}
$\chi_{15,1}$ &
${\chi}_{38}$ & 
$\phi_1^{2}\phi_2^{2}\phi_8^{2}\phi_{12}\phi_{24}$ & 
$\frac{1}{48}(q^4-10q^2+16)$
\rule[-0.2cm]{0cm}{0.36cm}\\
\hline
\rule{0cm}{0.36cm}
$\chi_{16,1}$ &
${\chi}_{39}$ & 
$\phi_1^{2}\phi_2^{2}\phi_4^{2}\phi_8^{2}\phi_{24}$ & 
$\frac{1}{6}(q^4-q^2-2)$
\rule[-0.2cm]{0cm}{0.36cm}\\
\hline
\rule{0cm}{0.36cm}
$\chi_{17,1}$ &
${\chi}_{40}$ & 
$\phi_1^{2}\phi_2^{2}\phi_4^{2}\phi_8^{2}\phi_{12}\phi_{24}'$ & 
$\frac{1}{12}(q^4+q^2-\sqrt{2}q^3-\sqrt{2}q)$
\rule[-0.2cm]{0cm}{0.36cm}\\
\hline
\rule{0cm}{0.36cm}
$\chi_{18,1}$ &
${\chi}_{41}$ & 
$\phi_1^{2}\phi_2^{2}\phi_4^{2}\phi_8^{2}\phi_{12}\phi_{24}''$ & 
$\frac{1}{12}(q^4+\sqrt{2}q^3+q^2+\sqrt{2}q)$
\rule[-0.2cm]{0cm}{0.36cm}\\
\hline
\end{tabular}
\end{center}
\end{table}

\newpage


\begin{landscape}
\section*{Appendix B: Scalar products, projective characters and relations}
\label{sec:appendixB}

\setcounter{table}{0}
\renewcommand{\thetable}{\textrm{B.\arabic{table}}}

\begin{table}[!ht]
\caption[]{Scalar products of the unipotent irreducible characters of
$G$ with some projective characters for 
$3 \neq \ell \, | \, q^2+1$. See Table~\ref{tab:proj2} for a
construction of these projectives.} \label{tab:scalar2} 

\begin{center}
\begin{tabular*}{20cm}{
l@{\extracolsep\fill}|ccccccccccccccccccccc
}
\hline
\rule{0cm}{0.4cm}
& \hspace{-0.2cm} $\Psi_{1}$ \hspace{-0.2cm} & $\Psi_{2}$ & $\Psi_{3}$ & $\Psi_{4}$ & $\Psi_{5}$ &
$\Psi_{6}$ & $\Psi_{7}$ & $\Psi_{8}$ & $\Psi_{9}$ & $\Psi_{10}$ &
$\Psi_{11}$ & $\Psi_{12}$ & $\Psi_{13}$ & $\Psi_{14}$ & $\Psi_{15}$ &
$\Psi_{16}$ & $\Psi_{17}$ & \hspace{-0.2cm} $\Psi_{18}$ \hspace{-0.2cm} & $\Psi_{19}$ & $\Psi_{20}$ &
\hspace{-0.2cm} $\Psi_{21}$ \hspace{-0.2cm} 
\rule[-0.1cm]{0cm}{0.32cm}\\
\hline
\rule{0cm}{0.32cm}
$\chi_{1}$ & $1$ & $.$ & $.$ & $.$ & $.$ & $.$ & $.$ & $.$ & $.$ & $.$ &
$.$ & $.$ & $.$ & $.$ & $.$ & $.$ & $.$ & $.$ & $.$ & $.$ & $.$
\rule[-0.1cm]{0cm}{0.32cm}\\
\rule{0cm}{0.32cm}
$\chi_{2}$ & $.$ & $1$ & $.$ & $.$ & $.$ & $.$ & $.$ & $.$ & $.$ & $.$ &
$.$ & $.$ & $.$ & $.$ & $.$ & $.$ & $.$ & $.$ & $.$ & $.$ & $.$
\rule[-0.1cm]{0cm}{0.32cm}\\
\rule{0cm}{0.32cm}
$\chi_{3}$ & $.$ & $.$ & $1$ & $.$ & $.$ & $.$ & $.$ & $.$ & $.$ & $.$ &
$.$ & $.$ & $.$ & $.$ & $.$ & $.$ & $.$ & $.$ & $.$ & $.$ & $.$
\rule[-0.1cm]{0cm}{0.32cm}\\
\rule{0cm}{0.32cm}
$\chi_{4}$ & $1$ & $.$ & $.$ & $1$ & $.$ & $.$ & $.$ & $.$ & $.$ & $.$ &
$.$ & $.$ & $.$ & $.$ & $.$ & $.$ & $.$ & $.$ & $.$ & $.$ & $.$
\rule[-0.1cm]{0cm}{0.32cm}\\
\rule{0cm}{0.32cm}
$\chi_{5}$ & $.$ & $.$ & $.$ & $.$ & $1$ & $.$ & $.$ & $.$ & $.$ & $.$ &
$.$ & $.$ & $.$ & $.$ & $.$ & $.$ & $.$ & $.$ & $.$ & $.$ & $.$
\rule[-0.1cm]{0cm}{0.32cm}\\
\rule{0cm}{0.32cm}
$\chi_{6}$ & $.$ & $.$ & $.$ & $.$ & $.$ & $1$ & $.$ & $.$ & $.$ & $.$ &
$.$ & $.$ & $.$ & $.$ & $.$ & $.$ & $.$ & $.$ & $.$ & $.$ & $.$
\rule[-0.1cm]{0cm}{0.32cm}\\
\rule{0cm}{0.32cm}
$\chi_{7}$ & $1$ & $.$ & $.$ & $1$ & $.$ & $.$ & $1$ & $.$ & $.$ & $.$ &
$.$ & $.$ & $.$ & $.$ & $.$ & $.$ & $.$ & $.$ & $.$ & $.$ & $.$
\rule[-0.1cm]{0cm}{0.32cm}\\
\hline
\end{tabular*} 
\end{center}
\end{table}
\end{landscape}

\newpage


\setcounter{table}{1}
\renewcommand{\thetable}{\textrm{B.\arabic{table}}}

\begin{table}[!ht]
\caption[]{Projective characters of $G$ for 
$3 \neq \ell \, | \, q^2+1$.} \label{tab:proj2}  

\begin{center}
\begin{tabular}{cll} \hline
\rule{0cm}{0.45cm}
Projective & Construction & Comments
\rule[-0.2cm]{0cm}{0.3cm}\\
\hline
\rule{-0.12cm}{0.4cm}
$\Psi_1$ & ${_{P_b}\chi_1(0)}^G-\chi_5-\chi_6$ & $\ell \nmid |P_b|$
\rule[-0.2cm]{0cm}{0.3cm}\\
\rule{-0.12cm}{0.4cm}
$\Psi_2$ & $\chi_2$ & defect~$0$
\rule[-0.2cm]{0cm}{0.3cm}\\
\rule{-0.12cm}{0.4cm}
$\Psi_3$ & $\chi_3$  & defect~$0$
\rule[-0.2cm]{0cm}{0.3cm}\\
\rule{-0.12cm}{0.4cm}
$\Psi_4$ & ${_{P_b}\chi_8(0)}^G-\chi_5-\chi_6$ & $\ell \nmid |P_b|$
\rule[-0.2cm]{0cm}{0.3cm}\\
\rule{-0.12cm}{0.4cm}
$\Psi_5$ & $\chi_5$  & defect~$0$
\rule[-0.2cm]{0cm}{0.3cm}\\
\rule{-0.12cm}{0.4cm}
$\Psi_6$ & $\chi_6$  & defect~$0$
\rule[-0.2cm]{0cm}{0.3cm}\\
\rule{-0.12cm}{0.4cm}
$\Psi_7$ & ${_{P_b}\chi_4(0)}^G-\chi_5-\chi_6$ & $\ell \nmid |P_b|$
\rule[-0.2cm]{0cm}{0.3cm}\\
\hline
\end{tabular}
\end{center}
\end{table}

\newpage


\begin{landscape}
\begin{table}[!ht]
\caption[]{Scalar products of the unipotent irreducible characters of
$G$ with some projective characters for 
$\ell \, | \, q^2+\sqrt{2}q+1$. See Table~\ref{tab:proj3} for a
construction of these projectives.} \label{tab:scalar3} 

\begin{center}
\begin{tabular*}{20cm}{
l@{\extracolsep\fill}|ccccccccccccccccccccc
}
\hline
\rule{0cm}{0.4cm}
& \hspace{-0.2cm} $\Psi_{1}$ \hspace{-0.2cm} & \hspace{-0.2cm}
$\Psi_{2}$ \hspace{-0.2cm} & \hspace{-0.2cm}
$\Psi_{3}$ \hspace{-0.2cm} & \hspace{-0.2cm}
$\Psi_{4}$ \hspace{-0.2cm} & \hspace{-0.2cm}
$\Psi_{5}$ \hspace{-0.2cm} & 
\hspace{-0.2cm} $\Psi_{6}$ \hspace{-0.2cm} & \hspace{-0.2cm}
$\Psi_{7}$ \hspace{-0.2cm} & \hspace{-0.2cm}
$\Psi_{8}$ \hspace{-0.2cm} & \hspace{-0.2cm}
$\Psi_{9}$ \hspace{-0.2cm} & \hspace{-0.2cm}
$\Psi_{10}$ \hspace{-0.2cm} & 
\hspace{-0.2cm} $\Psi_{11}$ \hspace{-0.2cm} & \hspace{-0.2cm}
$\Psi_{12}$ \hspace{-0.2cm} & \hspace{-0.2cm}
$\Psi_{13}$ \hspace{-0.2cm} & \hspace{-0.2cm} 
$\Psi_{14}$ \hspace{-0.2cm} & \hspace{-0.2cm}
$\Psi_{15}$ \hspace{-0.2cm} & \hspace{-0.2cm}
$\Psi_{16}$ \hspace{-0.2cm} & \hspace{-0.2cm}
$\Psi_{17}$ \hspace{-0.2cm}  
& \hspace{-0.2cm} $\Psi_{18}$ \hspace{-0.2cm} & \hspace{-0.2cm}
$\Psi_{19}$ \hspace{-0.2cm} & \hspace{-0.2cm} $\Psi_{20}$ \hspace{-0.2cm} & 
\hspace{-0.2cm} $\Psi_{21}$ \hspace{-0.2cm} 
\rule[-0.1cm]{0cm}{0.32cm}\\
\hline
\rule{0cm}{0.32cm}
$\chi_{1}$ & $1$ & $.$ & $.$ & $.$ & $.$ & $.$ & $.$ & $.$ & $.$ & $.$ &
$.$ & $.$ & $.$ & $.$ & $.$ & $.$ & $.$ & $.$ & $.$ & $.$ & $.$
\rule[-0.1cm]{0cm}{0.32cm}\\
\rule{0cm}{0.32cm}
$\chi_{2}$ & $.$ & $1$ & $.$ & $.$ & $.$ & $.$ & $.$ & $.$ & $.$ & $.$ &
$.$ & $.$ & $.$ & $.$ & $.$ & $.$ & $.$ & $.$ & $.$ & $.$ & $.$
\rule[-0.1cm]{0cm}{0.32cm}\\
\rule{0cm}{0.32cm}
$\chi_{3}$ & $.$ & $.$ & $1$ & $.$ & $.$ & $.$ & $.$ & $.$ & $.$ & $.$ &
$.$ & $.$ & $.$ & $.$ & $.$ & $.$ & $.$ & $.$ & $.$ & $.$ & $.$
\rule[-0.1cm]{0cm}{0.32cm}\\
\rule{0cm}{0.32cm}
$\chi_{4}$ & $.$ & $.$ & $.$ & $1$ & $.$ & $.$ & $.$ & $.$ & $.$ & $.$ &
$.$ & $.$ & $.$ & $.$ & $.$ & $.$ & $.$ & $.$ & $.$ & $.$ & $.$
\rule[-0.1cm]{0cm}{0.32cm}\\
\hline
\end{tabular*} 
\end{center}
\end{table}
\end{landscape}

\newpage


\begin{table}[!ht]
\caption[]{Projective characters of $G$ for 
$\ell \, | \, q^2+\sqrt{2}q+1$.} \label{tab:proj3}

\begin{center}
\begin{tabular}{cll} \hline
\rule{0cm}{0.45cm}
Projective & Construction & Comments
\rule[-0.2cm]{0cm}{0.3cm}\\
\hline
\rule{-0.12cm}{0.4cm}
$\Psi_1$ & ${_{P_a}\chi_1(0)}^G-\chi_6-\chi_7$ & $\ell \nmid |P_a|$
\rule[-0.2cm]{0cm}{0.3cm}\\
\rule{-0.12cm}{0.4cm}
$\Psi_2$ & $R_{L_b}^G(\Phi_a)-\chi_6-\chi_7$ & Harish-Chandra induced projective
\rule[-0.2cm]{0cm}{0.3cm}\\
\rule{-0.12cm}{0.4cm}
$\Psi_3$ & $R_{L_b}^G(\Phi_b)-\chi_6-\chi_7$ & Harish-Chandra induced projective
\rule[-0.2cm]{0cm}{0.3cm}\\
\rule{-0.12cm}{0.4cm}
$\Psi_4$ & ${_{P_a}\chi_2(0)}^G-\chi_6-\chi_7$ & $\ell \nmid |P_a|$
\rule[-0.2cm]{0cm}{0.3cm}\\
\rule{-0.12cm}{0.4cm}
$\Psi_5$ & $R_{L_b}^G(\Phi_{St})-\chi_6-\chi_7$ & Harish-Chandra induced projective
\rule[-0.2cm]{0cm}{0.3cm}\\
\rule{-0.12cm}{0.4cm}
$\Psi_6$ & $\chi_6$ & defect $0$
\rule[-0.2cm]{0cm}{0.3cm}\\
\rule{-0.12cm}{0.4cm}
$\Psi_7$ & $\chi_7$ & defect $0$
\rule[-0.2cm]{0cm}{0.3cm}\\
\hline
\end{tabular}
\end{center}
\end{table}

\newpage


\begin{landscape}
\begin{table}[!ht]
\caption[]{Scalar products of the unipotent irreducible characters of
$G$ with some projective characters for 
$\ell \, | \, q^2-\sqrt{2}q+1$. See Table~\ref{tab:proj4} for a
construction of these projectives.} \label{tab:scalar4} 

\begin{center}
\begin{tabular*}{20cm}{
l@{\extracolsep\fill}|ccccccccccccccccccccc
}
\hline
\rule{0cm}{0.4cm}
& \hspace{-0.2cm} $\Psi_{1}$ \hspace{-0.2cm} & \hspace{-0.2cm}
$\Psi_{2}$ \hspace{-0.2cm} & \hspace{-0.2cm}
$\Psi_{3}$ \hspace{-0.2cm} & \hspace{-0.2cm}
$\Psi_{4}$ \hspace{-0.2cm} & \hspace{-0.2cm}
$\Psi_{5}$ \hspace{-0.2cm} & 
\hspace{-0.2cm} $\Psi_{6}$ \hspace{-0.2cm} & \hspace{-0.2cm}
$\Psi_{7}$ \hspace{-0.2cm} & \hspace{-0.2cm}
$\Psi_{8}$ \hspace{-0.2cm} & \hspace{-0.2cm}
$\Psi_{9}$ \hspace{-0.2cm} & \hspace{-0.2cm}
$\Psi_{10}$ \hspace{-0.2cm} & 
\hspace{-0.2cm} $\Psi_{11}$ \hspace{-0.2cm} & \hspace{-0.2cm}
$\Psi_{12}$ \hspace{-0.2cm} & \hspace{-0.2cm}
$\Psi_{13}$ \hspace{-0.2cm} & \hspace{-0.2cm} 
$\Psi_{14}$ \hspace{-0.2cm} & \hspace{-0.2cm}
$\Psi_{15}$ \hspace{-0.2cm} & \hspace{-0.2cm}
$\Psi_{16}$ \hspace{-0.2cm} & \hspace{-0.2cm}
$\Psi_{17}$ \hspace{-0.2cm}  
& \hspace{-0.2cm} $\Psi_{18}$ \hspace{-0.2cm} & \hspace{-0.2cm}
$\Psi_{19}$ \hspace{-0.2cm} & \hspace{-0.2cm} $\Psi_{20}$ \hspace{-0.2cm} & 
\hspace{-0.2cm} $\Psi_{21}$ \hspace{-0.2cm} 
\rule[-0.1cm]{0cm}{0.32cm}\\
\hline
\rule{0cm}{0.32cm}
$\chi_{1}$ & $1$ & $.$ & $.$ & $.$ & $.$ & $.$ & $.$ & $.$ & $.$ & $.$ &
$.$ & $.$ & $.$ & $.$ & $.$ & $.$ & $.$ & $.$ & $.$ & $.$ & $.$
\rule[-0.1cm]{0cm}{0.32cm}\\
\rule{0cm}{0.32cm}
$\chi_{2}$ & $.$ & $1$ & $.$ & $.$ & $.$ & $.$ & $.$ & $.$ & $.$ & $.$ &
$.$ & $.$ & $.$ & $.$ & $.$ & $.$ & $.$ & $.$ & $.$ & $.$ & $.$
\rule[-0.1cm]{0cm}{0.32cm}\\
\rule{0cm}{0.32cm}
$\chi_{3}$ & $.$ & $.$ & $1$ & $.$ & $.$ & $.$ & $.$ & $.$ & $.$ & $.$ &
$.$ & $.$ & $.$ & $.$ & $.$ & $.$ & $.$ & $.$ & $.$ & $.$ & $.$
\rule[-0.1cm]{0cm}{0.32cm}\\
\rule{0cm}{0.32cm}
$\chi_{4}$ & $.$ & $.$ & $.$ & $1$ & $.$ & $.$ & $.$ & $.$ & $.$ & $.$ &
$.$ & $.$ & $.$ & $.$ & $.$ & $.$ & $.$ & $.$ & $.$ & $.$ & $.$
\rule[-0.1cm]{0cm}{0.32cm}\\
\hline
\end{tabular*} 
\end{center}
\end{table}
\end{landscape}

\newpage


\begin{table}[!ht]
\caption[]{Projective characters of $G$ for 
$\ell \, | \, q^2-\sqrt{2}q+1$.} \label{tab:proj4}

\begin{center}
\begin{tabular}{cll} \hline
\rule{0cm}{0.45cm}
Projective & Construction & Comments
\rule[-0.2cm]{0cm}{0.3cm}\\
\hline
\rule{-0.12cm}{0.4cm}
$\Psi_1$ & ${_{P_a}\chi_1(0)}^G-\chi_5-\chi_7$ & $\ell \nmid |P_a|$
\rule[-0.2cm]{0cm}{0.3cm}\\
\rule{-0.12cm}{0.4cm}
$\Psi_2$ & $R_{L_b}^G(\Phi_a)$ & Harish-Chandra induced projective
\rule[-0.2cm]{0cm}{0.3cm}\\
\rule{-0.12cm}{0.4cm}
$\Psi_3$ & $R_{L_b}^G(\Phi_b)$ & Harish-Chandra induced projective
\rule[-0.2cm]{0cm}{0.3cm}\\
\rule{-0.12cm}{0.4cm}
$\Psi_4$ & ${_{P_a}\chi_2(0)}^G-\chi_5-\chi_7$ & $\ell \nmid |P_a|$
\rule[-0.2cm]{0cm}{0.3cm}\\
\rule{-0.12cm}{0.4cm}
$\Psi_5$ & $\chi_5$ & defect $0$
\rule[-0.2cm]{0cm}{0.3cm}\\
\rule{-0.12cm}{0.4cm}
$\Psi_6$ & $R_{L_b}^G(\Phi_{St})-\chi_5-\chi_7$ & Harish-Chandra induced projective
\rule[-0.2cm]{0cm}{0.3cm}\\
\rule{-0.12cm}{0.4cm}
$\Psi_7$ & $\chi_7$ & defect $0$
\rule[-0.2cm]{0cm}{0.3cm}\\
\hline
\end{tabular}
\end{center}
\end{table}

\newpage


\begin{landscape}
\begin{table}[!ht]
\caption[]{Scalar products of the basic set characters of $G$ with some
projective characters for $\ell = 3$. See Table~\ref{tab:projl3} for a
construction of these projectives.} \label{tab:scalarl3} 

\begin{center}
\begin{tabular*}{20cm}{
l@{\extracolsep\fill}|ccccccccccccccccccc
}
\hline
\rule{0cm}{0.4cm}
& \hspace{-0.2cm} $\Psi_{1}$ \hspace{-0.2cm} & \hspace{-0.2cm}
$\Psi_{2}$ \hspace{-0.2cm} & \hspace{-0.2cm}
$\Psi_{3}$ \hspace{-0.2cm} & \hspace{-0.2cm}
$\Psi_{4}$ \hspace{-0.2cm} & \hspace{-0.2cm}
$\Psi_{5,1}$ \hspace{-0.2cm} & 
\hspace{-0.2cm} $\Psi_{5}$ \hspace{-0.2cm} & \hspace{-0.2cm}
$\Psi_{6}$ \hspace{-0.2cm} & \hspace{-0.2cm}
$\Psi_{7}$ \hspace{-0.2cm} & \hspace{-0.2cm}
$\Psi_{8}$ \hspace{-0.2cm} & \hspace{-0.2cm}
$\Psi_{10}$ \hspace{-0.2cm} & 
\hspace{-0.2cm} $\Psi_{11}$ \hspace{-0.2cm} & \hspace{-0.2cm}
$\Psi_{12}$ \hspace{-0.2cm} & \hspace{-0.2cm}
$\Psi_{13}$ \hspace{-0.2cm} & \hspace{-0.2cm} 
$\Psi_{14}$ \hspace{-0.2cm} & \hspace{-0.2cm}
$\Psi_{15}$ \hspace{-0.2cm} & \hspace{-0.2cm}
$\Psi_{18}$ \hspace{-0.2cm} & \hspace{-0.2cm}
$\Psi_{19}$ \hspace{-0.2cm}  
& \hspace{-0.2cm} $\Psi_{20}$ \hspace{-0.2cm} & \hspace{-0.2cm}
$\Psi_{21}$ \hspace{-0.2cm}
\rule[-0.1cm]{0cm}{0.32cm}\\
\hline
\rule{0cm}{0.32cm}
$\chi_{1}$ & $1$ & $.$ & $.$ & $.$ & $.$ & $.$ & $.$ & $.$ & $.$ & $.$ &
$.$ & $.$ & $.$ & $.$ & $.$ & $.$ & $.$ & $.$ & $.$
\rule[-0.1cm]{0cm}{0.32cm}\\
\rule{0cm}{0.32cm}
$\chi_{2}$ & $.$ & $1$ & $.$ & $.$ & $.$ & $.$ & $.$ & $.$ & $.$ & $.$ &
$.$ & $.$ & $.$ & $.$ & $.$ & $.$ & $.$ & $.$ & $.$
\rule[-0.1cm]{0cm}{0.32cm}\\
\rule{0cm}{0.32cm}
$\chi_{3}$ & $.$ & $.$ & $1$ & $.$ & $.$ & $.$ & $.$ & $.$ & $.$ & $.$ &
$.$ & $.$ & $.$ & $.$ & $.$ & $.$ & $.$ & $.$ & $.$
\rule[-0.1cm]{0cm}{0.32cm}\\
\rule{0cm}{0.32cm}
$\chi_{4}$ & $1$ & $.$ & $.$ & $1$ & $.$ & $.$ & $.$ & $.$ & $.$ & $.$ &
$.$ & $.$ & $.$ & $.$ & $.$ & $.$ & $.$ & $.$ & $.$
\rule[-0.1cm]{0cm}{0.32cm}\\
\rule{0cm}{0.32cm}
$\chi_{5,1}$ & $.$ & $.$ & $.$ & $1$ & $1$ & $.$ & $.$ & $.$ & $.$ & $.$ &
$.$ & $.$ & $.$ & $.$ & $.$ & $.$ & $.$ & $.$ & $.$
\rule[-0.1cm]{0cm}{0.32cm}\\
\hline
\end{tabular*} 
\end{center}
\end{table}
\end{landscape}

\newpage


\begin{table}[!ht]
\caption[]{Projective characters of $G$ for 
$\ell = 3$.} \label{tab:projl3}  

\begin{center}
\begin{tabular}{cll} \hline
\rule{0cm}{0.45cm}
Projective & Construction & Comments
\rule[-0.2cm]{0cm}{0.3cm}\\
\hline
\rule{-0.12cm}{0.4cm}
$\Psi_1$ & ${_{P_b}\chi_{1}}(0)^G - \chi_5 - \chi_6$ & $3 \nmid |P_b|$
\rule[-0.2cm]{0cm}{0.3cm}\\
\rule{-0.12cm}{0.4cm}
$\Psi_2$ & $\chi_2$ & defect $0$
\rule[-0.2cm]{0cm}{0.3cm}\\
\rule{-0.12cm}{0.4cm}
$\Psi_3$ & $\chi_3$ & defect $0$
\rule[-0.2cm]{0cm}{0.3cm}\\
\rule{-0.12cm}{0.4cm}
$\Psi_4$ & ${_{P_b}\chi_{8}}(0)^G - \chi_5 - \chi_6$ & $3 \nmid |P_b|$
\rule[-0.2cm]{0cm}{0.3cm}\\
\rule{-0.12cm}{0.4cm}
$\Psi_{5,1}$ & ${_{P_b}\chi_{39}}^G$ & $3 \nmid |P_b|$
\rule[-0.2cm]{0cm}{0.3cm}\\
\rule{-0.12cm}{0.4cm}
$\Psi_5$ & $\chi_5$ & defect $0$
\rule[-0.2cm]{0cm}{0.3cm}\\
\rule{-0.12cm}{0.4cm}
$\Psi_6$ & $\chi_6$ & defect $0$
\rule[-0.2cm]{0cm}{0.3cm}\\
\hline
\end{tabular}
\end{center}
\end{table}

\newpage


\begin{landscape}
\begin{table}[!ht]
\caption[]{Relations (with respect to the basic set of unipotent
characters) for the irreducible characters in the unipotent blocks in
case $\ell \, | \, q^2-1$. For details on the labeling and entries of
this table, see Subsection~\ref{relations}.}  
\label{tab:rels1} 

\begin{center}
\begin{tabular*}{20cm}{
l@{\extracolsep\fill}|ccccccccccccccccccccc
}
\hline
\rule{0cm}{0.5cm}
Characters & \hspace{-0.2cm} $1$ \hspace{-0.2cm} &
\hspace{-0.2cm} $2$ \hspace{-0.2cm} & \hspace{-0.2cm} $3$ \hspace{-0.2cm} &
\hspace{-0.2cm} $4$ \hspace{-0.2cm} & \hspace{-0.2cm} $5$ \hspace{-0.2cm} & 
\hspace{-0.2cm} $6$ \hspace{-0.2cm} & \hspace{-0.2cm} $7$\hspace{-0.2cm}  &
\hspace{-0.2cm} $8$ \hspace{-0.2cm} & \hspace{-0.2cm} $9$ \hspace{-0.2cm} &
\hspace{-0.2cm} $10$ \hspace{-0.2cm} & 
\hspace{-0.2cm} $11$ \hspace{-0.2cm} & \hspace{-0.2cm} $12$ \hspace{-0.2cm} &
\hspace{-0.2cm} $13$ \hspace{-0.2cm} & \hspace{-0.2cm} $14$ \hspace{-0.2cm} &
\hspace{-0.2cm} $15$ \hspace{-0.2cm} & 
\hspace{-0.2cm} $16$ \hspace{-0.2cm} & \hspace{-0.2cm} $17$ \hspace{-0.2cm} &
\hspace{-0.2cm} $18$ \hspace{-0.2cm} & \hspace{-0.2cm} $19$ \hspace{-0.2cm} &
\hspace{-0.2cm} $20$ \hspace{-0.2cm} & 
\hspace{-0.2cm} $21$ \hspace{-0.2cm} 
\rule[-0.1cm]{0cm}{0.5cm}\\
\hline
\rule{0cm}{0.5cm}
$\{\chi_{2,1}\}$ & 
$1$ & $.$ & $.$ & $1$ & $1$ & $1$ & $1$ & $.$ & $.$ & $.$ & $.$ & $.$
& $.$ & $.$ & $.$ & $.$ & $.$ & $.$ & $.$ & $.$ & $.$ 
\rule[-0.1cm]{0cm}{0.5cm}\\
\rule{0cm}{0.5cm}
$\{\chi_{2,\frac{q}{\sqrt{2}}(q^2-1)_a}\}$ & 
$.$ & $1$ & $.$ & $.$ & $.$ & $.$ & $.$ & $.$ & $.$ & $.$ & $.$ & $.$
& $.$ & $.$ & $.$ & $.$ & $.$ & $.$ & $1$ & $.$ & $.$ 
\rule[-0.1cm]{0cm}{0.5cm}\\
\rule{0cm}{0.5cm}
$\{\chi_{2,\frac{q}{\sqrt{2}}(q^2-1)_b}\}$ & 
$.$ & $.$ & $1$ & $.$ & $.$ & $.$ & $.$ & $.$ & $.$ & $.$ & $.$ & $.$
& $.$ & $.$ & $.$ & $.$ & $.$ & $.$ & $.$ & $1$ & $.$ 
\rule[-0.1cm]{0cm}{0.5cm}\\
\rule{0cm}{0.5cm}
$\{\chi_{2,St}\}$ & 
$.$ & $.$ & $.$ & $.$ & $1$ & $1$ & $1$ & $.$ & $.$ & $.$ & $.$ & $.$
& $.$ & $.$ & $.$ & $.$ & $.$ & $1$ & $.$ & $.$ & $1$ 
\rule[-0.1cm]{0cm}{0.5cm}\\
\rule{0cm}{0.5cm}
$\{\chi_{3,1}\}$ & 
$1$ & $.$ & $.$ & $.$ & $1$ & $1$ & $1$ & $.$ & $.$ & $.$ & $.$ & $.$
& $.$ & $.$ & $.$ & $.$ & $.$ & $1$ & $.$ & $.$ & $.$ 
\rule[-0.1cm]{0cm}{0.5cm}\\
\hline
\end{tabular*} 
\end{center}
\end{table}


\begin{table}[!ht]
\caption[]{Relations (with respect to the basic set of unipotent
characters) for the irreducible characters in the unipotent blocks in
case $3 \neq \ell \, | \, q^2+1$. For details on the labeling and entries of
this table, see Subsection~\ref{relations}.}  
\label{tab:rels2} 

\begin{center}
\begin{tabular*}{20cm}{
l@{\extracolsep\fill}|rrrrrrrrrrrrrrrrrrrrr
}
\hline
\rule{0cm}{0.5cm}
Characters &  $1$ & $2$  &  $3$  &  $4$  &  $5$  &  $6$  &  $7$  &
 $8$  &  $9$  & $10$  &  $11$  &  $12$  & $13$  &  $14$  & $15$  & 
 $16$  &  $17$  & $18$  &  $19$  & $20$  &  $21$  
\rule[-0.1cm]{0cm}{0.5cm}\\
\hline
\rule{0cm}{0.5cm}
$\{\chi_{6,1}\}$ & 
$-1$ & $.$ & $.$ & $2$ & $.$ & $.$ & $-1$ & $1$ & $1$ & $.$ & $.$ & $.$ & $.$ & 
$.$ & $.$ & $.$ & $-2$ & $1$ & $.$ & $.$ & $.$
\rule[-0.1cm]{0cm}{0.5cm}\\
\rule{0cm}{0.5cm}
$\{\chi_{6,St}\}$ & 
$.$ & $.$ & $.$ & $-1$ & $.$ & $.$ & $1$ & $-1$ & $-1$ & $.$ & $.$ & $.$ & $.$ & 
$.$ & $.$ & $.$ & $2$ & $-2$ & $.$ & $.$ & $1$
\rule[-0.1cm]{0cm}{0.5cm}\\
\rule{0cm}{0.5cm}
$\{\chi_{15,1}\}$ & 
$1$ & $.$ & $.$ & $-3$ & $.$ & $.$ & $2$ & $-2$ & $-2$ & $.$ & $.$ & $.$ & $.$ & 
$.$ & $.$ & $.$ & $4$ & $-3$ & $.$ & $.$ & $1$
\rule[-0.1cm]{0cm}{0.5cm}\\
\hline
\end{tabular*} 
\end{center}
\end{table}


\begin{table}[!ht]
\caption[]{Relations (with respect to the basic set of unipotent
characters) for the irreducible characters in the unipotent blocks in
case $\ell \, | \, q^2+\sqrt{2}q+1$. For details on the labeling and entries of
this table, see Subsection~\ref{relations}.}  
\label{tab:rels3} 

\begin{center}
\begin{tabular*}{20cm}{
l@{\extracolsep\fill}|rrrrrrrrrrrrrrrrrrrrr
}
\hline
\rule{0cm}{0.5cm}
Characters &  $1$  &
 $2$  &  $3$  &
 $4$  &  $5$  & 
 $6$  &  $7$  &
 $8$  &  $9$  &
 $10$  & 
 $11$  &  $12$  &
 $13$  &  $14$  &
 $15$  & 
 $16$  &  $17$  &
 $18$  &  $19$  &
 $20$  & 
 $21$                 
\rule[-0.1cm]{0cm}{0.5cm}\\
\hline
\rule{0cm}{0.5cm}
$\{\chi_{10,1}\}$ & 
$-1$ & $-1$ & $-1$ & $-1$ & $1$ & $.$ & $.$ & $.$ & $1$ & $1$ & $1$ &
$1$ & $.$ & $.$ & $.$ & $.$ & $.$ & $.$ & $.$ & $.$ & $.$
\rule[-0.1cm]{0cm}{0.5cm}\\
\rule{0cm}{0.5cm}
$\{\chi_{10,\frac{q}{\sqrt{2}}(q^2-1)_a}\}$ & 
$.$ & $-1$ & $.$ & $-1$ & $1$ & $.$ & $.$ & $.$ & $1$ & $.$ & $.$ &
$1$ & $-1$ & $.$ & $.$ & $.$ & $-1$ & $-1$ & $1$ & $.$ & $.$
\rule[-0.1cm]{0cm}{0.5cm}\\
\rule{0cm}{0.5cm}
$\{\chi_{10,\frac{q}{\sqrt{2}}(q^2-1)_b}\}$ & 
$.$ & $.$ & $-1$ & $-1$ & $1$ & $.$ & $.$ & $.$ & $1$ & $.$ & $1$ &
$.$ & $.$ & $-1$ & $.$ & $.$ & $-1$ & $-1$ & $.$ & $1$ & $.$
\rule[-0.1cm]{0cm}{0.5cm}\\
\rule{0cm}{0.5cm}
$\{\chi_{10,St}\}$ & 
$.$ & $.$ & $.$ & $.$ & $-1$ & $.$ & $.$ & $.$ & $-1$ & $-1$ & $-1$ &
$-1$ & $.$ & $.$ & $.$ & $.$ & $.$ & $1$ & $-1$ & $-1$ & $1$
\rule[-0.1cm]{0cm}{0.5cm}\\
\rule{0cm}{0.5cm}
$\{\chi_{14,1}\}$ & 
$1$ & $2$ & $2$ & $3$ & $-4$ & $.$ & $.$ & $.$ & $-4$ & $-2$ & $-3$ &
$-3$ & $1$ & $1$ & $.$ & $.$ & $2$ & $3$ & $-2$ & $-2$ & $1$
\rule[-0.1cm]{0cm}{0.5cm}\\
\hline
\end{tabular*} 
\end{center}
\end{table}


\begin{table}[!ht]
\caption[]{Relations (with respect to the basic set of unipotent
characters) for the irreducible characters in the unipotent blocks in
case $\ell \, | \, q^2-\sqrt{2}q+1$. For details on the labeling and entries of
this table, see Subsection~\ref{relations}.}  
\label{tab:rels4} 

\begin{center}
\begin{tabular*}{20cm}{
l@{\extracolsep\fill}|rrrrrrrrrrrrrrrrrrrrr
}
\hline
\rule{0cm}{0.5cm}
Characters &  $1$  &
 $2$  &  $3$  &
 $4$  &  $5$  & 
 $6$  &  $7$  &
 $8$  &  $9$  &
 $10$  & 
 $11$  &  $12$  &
 $13$  &  $14$  &
 $15$  & 
 $16$  &  $17$  &
 $18$  &  $19$  &
 $20$  & 
 $21$
\rule[-0.1cm]{0cm}{0.5cm}\\
\hline
\rule{0cm}{0.5cm}
$\{\chi_{8,1}\}$ & 
$-1$ & $1$ & $1$ & $-1$ & $.$ & $1$ & $.$ & $1$ & $.$ & $1$ & $.$ & $.$ & $1$ & 
$1$ & $.$ & $.$ & $.$ & $.$ & $.$ & $.$ & $.$
\rule[-0.1cm]{0cm}{0.5cm}\\
\rule{0cm}{0.5cm}
$\{\chi_{8,\frac{q}{\sqrt{2}}(q^2-1)_a}\}$ & 
$.$ & $-1$ & $.$ & $1$ & $.$ & $-1$ & $.$ & $-1$ & $.$ & $.$ & $1$ & $.$ & $.$ & $
-1$ & $.$ & $.$ & $1$ & $1$ & $1$ & $.$ & $.$
\rule[-0.1cm]{0cm}{0.5cm}\\
\rule{0cm}{0.5cm}
$\{\chi_{8,\frac{q}{\sqrt{2}}(q^2-1)_b}\}$ & 
$.$ & $.$ & $-1$ & $1$ & $.$ & $-1$ & $.$ & $-1$ & $.$ & $.$ & $.$ & $1$ & $-1$ & 
$.$ & $.$ & $.$ & $1$ & $1$ & $.$ & $1$ & $.$
\rule[-0.1cm]{0cm}{0.5cm}\\
\rule{0cm}{0.5cm}
$\{\chi_{8,St}\}$ & 
$.$ & $.$ & $.$ & $.$ & $.$ & $-1$ & $.$ & $-1$ & $.$ & $-1$ & $.$ & $.$ & $-1$ & 
$-1$ & $.$ & $.$ & $.$ & $1$ & $1$ & $1$ & $1$
\rule[-0.1cm]{0cm}{0.5cm}\\
\rule{0cm}{0.5cm}
$\{\chi_{13,1}\}$ & 
$1$ & $-2$ & $-2$ & $3$ & $.$ & $-4$ & $.$ & $-4$ & $.$ & $-2$ & $1$ & $1$ & $-3$
 & $-3$ & $.$ & $.$ & $2$ & $3$ & $2$ & $2$ & $1$
\rule[-0.1cm]{0cm}{0.5cm}\\
\hline
\end{tabular*} 
\end{center}
\end{table}


\begin{table}[!ht]
\caption[]{Relations (with respect to the ordinary basic set for the
unipotent blocks) in case $\ell = 3$. For details on the labeling and entries of
this table, see Section~\ref{sec:decompmatl3}.}  
\label{tab:relsl3} 

\begin{center}
\begin{tabular*}{20cm}{
l@{\extracolsep\fill}|rrrrrrrrrrrrrrrrrrr
}
\hline
\rule{0cm}{0.5cm}
Characters &  $1$  &
 $2$  &  $3$  &
 $4$  &  $5,1$  & 
 $5$  & 
 $6$  &  $7$  &
 $8$  & 
 $10$  & 
 $11$  &  $12$  &
 $13$  &  $14$  &
 $15$  & 
 $18$  &  $19$  &
 $20$  & 
 $21$                 
\rule[-0.1cm]{0cm}{0.5cm}\\
\hline
\rule{0cm}{0.5cm}
$\chi_{9}$ & 
$1$ & $.$ & $.$ & $-1$ & $1$ & $.$ & $.$ & $.$ & $-1$ & $1$ & $.$ &
$.$ & $.$ & $.$ & $.$ & $.$ & $.$ & $.$ & $.$ 
\rule[-0.1cm]{0cm}{0.5cm}\\
\rule{0cm}{0.5cm}
$\chi_{16}$ & 
$.$ & $.$ & $.$ & $.$ & $.$ & $.$ & $.$ & $.$ & $.$ & $.$ & $.$ & $.$
& $.$ & $.$ & $1$ & $.$ & $.$ & $.$ & $.$ 
\rule[-0.1cm]{0cm}{0.5cm}\\
\rule{0cm}{0.5cm}
$\chi_{17}$ & 
$-1$ & $.$ & $.$ & $1$ & $-1$ & $.$ & $.$ & $.$ & $.$ & $.$ & $.$ &
$.$ & $.$ & $.$ & $1$ & $.$ & $.$ & $.$ & $.$ 
\rule[-0.1cm]{0cm}{0.5cm}\\
\rule{0cm}{0.5cm}
$\chi_{5,q^2(q^2-1)}$ & 
$2$ & $.$ & $.$ & $-1$ & $2$ & $.$ & $.$ & $-1$ & $.$ & $1$ & $.$ &
$.$ & $.$ & $.$ & $-2$ & $1$ & $.$ & $.$ & $.$ 
\rule[-0.1cm]{0cm}{0.5cm}\\
\rule{0cm}{0.5cm}
$\chi_{5,St}$ & 
$-1$ & $.$ & $.$ & $1$ & $-1$ & $.$ & $.$ & $.$ & $.$ & $.$ & $.$ &
$.$ & $.$ & $.$ & $.$ & $-1$ & $.$ & $.$ & $1$ 
\rule[-0.1cm]{0cm}{0.5cm}\\
\rule{0cm}{0.5cm}
$\{\chi_{6,1}\}$ & 
$2$ & $.$ & $.$ & $-1$ & $3$ & $.$ & $.$ & $-1$ & $.$ & $1$ & $.$ &
$.$ & $.$ & $.$ & $-2$ & $1$ & $.$ & $.$ & $.$ 
\rule[-0.1cm]{0cm}{0.5cm}\\
\rule{0cm}{0.5cm}
$\{\chi_{6,St}\}$ & 
$-3$ & $.$ & $.$ & $2$ & $-3$ & $.$ & $.$ & $1$ & $.$ & $-1$ & $.$ &
$.$ & $.$ & $.$ & $2$ & $-2$ & $.$ & $.$ & $1$ 
\rule[-0.1cm]{0cm}{0.5cm}\\
\rule{0cm}{0.5cm}
$\{\chi_{15,1}\}$ & 
$-5$ & $.$ & $.$ & $3$ & $-6$ & $.$ & $.$ & $2$ & $.$ & $-2$ & $.$ &
$.$ & $.$ & $.$ & $4$ & $-3$ & $.$ & $.$ & $1$
\rule[-0.1cm]{0cm}{0.5cm}\\
\hline
\end{tabular*} 
\end{center}
\end{table}
\end{landscape}


\begin{landscape}
\section*{Appendix C: Decomposition numbers of unipotent characters}
\label{sec:appendixC}

\setcounter{table}{0}
\renewcommand{\thetable}{\textrm{C.\arabic{table}}}

\begin{table}[!ht]
\caption[]{The decomposition numbers of the unipotent blocks of $G$
for $\ell \, | \, q^2-1$. In the right most column, $\ell^f$ is the
largest power of $\ell$ dividing $q^2-1$.} \label{tab:decuni1} 

\begin{center}
\begin{tabular*}{20cm}{
l@{\extracolsep\fill}|ccccccccccccccccccccc|c
}
\hline
\rule{0cm}{0.19cm}
& \hspace{-0.2cm} $\phi_{1}$ \hspace{-0.2cm} & $\phi_{2}$ & $\phi_{3}$ & $\phi_{4}$ & $\phi_{5}$ &
\hspace{-0.15cm} $\phi_{6}$ \hspace{-0.15cm} & \hspace{-0.15cm}
$\phi_{7}$ \hspace{-0.15cm} &\hspace{-0.15cm}
$\phi_{8}$ \hspace{-0.15cm} & \hspace{-0.15cm}
$\phi_{9}$ \hspace{-0.15cm} & \hspace{-0.15cm}
$\phi_{10}$ \hspace{-0.15cm} & 
\hspace{-0.15cm} $\phi_{11}$ \hspace{-0.15cm} & \hspace{-0.15cm}
$\phi_{12}$ \hspace{-0.15cm} & \hspace{-0.15cm}
$\phi_{13}$ \hspace{-0.15cm} & \hspace{-0.15cm}
$\phi_{14}$ \hspace{-0.15cm} & \hspace{-0.15cm}
$\phi_{15}$ \hspace{-0.15cm} & 
\hspace{-0.15cm} $\phi_{16}$ \hspace{-0.15cm} & \hspace{-0.15cm}
$\phi_{17}$ \hspace{-0.15cm} & \hspace{-0.15cm}
$\phi_{18}$ \hspace{-0.15cm} & $\phi_{19}$ & $\phi_{20}$ & 
\hspace{-0.2cm} $\phi_{21}$ \hspace{-0.2cm} &
\rule[-0.1cm]{0cm}{0.3cm}\\
\rule{0cm}{0.3cm}
Series \,\, & \hspace{-0.2cm} $ps$ \hspace{-0.2cm} & \hspace{-0.2cm} ${^2B_2}[a]$ \hspace{-0.2cm}
& \hspace{-0.2cm} ${^2B_2}[b]$ \hspace{-0.2cm} & $ps$ & $ps$ & $ps$ & 
$ps$ & \hspace{-0.2cm} $c$ \hspace{-0.2cm} & \hspace{-0.2cm}
$c$ \hspace{-0.2cm} & \hspace{-0.2cm} $c$ \hspace{-0.2cm}
& \hspace{-0.2cm} $c$ \hspace{-0.2cm} & \hspace{-0.2cm}
$c$ \hspace{-0.2cm} & \hspace{-0.2cm} $c$ \hspace{-0.2cm}
& \hspace{-0.2cm} $c$ \hspace{-0.2cm} & \hspace{-0.2cm}
$c$ \hspace{-0.2cm} & \hspace{-0.2cm} $c$ \hspace{-0.2cm} & 
\hspace{-0.2cm} $c$ \hspace{-0.2cm} 
& \hspace{-0.2cm} $ps$ \hspace{-0.2cm} 
& \hspace{-0.2cm} ${^2B_2}[a]$ \hspace{-0.2cm} & \hspace{-0.2cm}
${^2B_2}[b]$ \hspace{-0.2cm} & \hspace{-0.2cm} $ps$ \hspace{-0.2cm} & Number  
\rule[-0.1cm]{0cm}{0.3cm}\\
\hline
\rule{0cm}{0.3cm}
$\chi_{1}$ & $1$ & \multicolumn{1}{|c}{$.$} & $.$ & $.$ & $.$ & $.$ & $.$ & $.$ & $.$ & $.$ &
$.$ & $.$ & $.$ & $.$ & $.$ & $.$ & $.$ & $.$ & $.$ & $.$ & $.$ & $1$
\rule[-0.1cm]{0cm}{0.3cm}\\
\hline
\end{tabular*} 
\end{center}
\end{table}

\newpage

\noindent {\sc Table~C.1.} (\textit{continued})

\begin{center}
\begin{tabular*}{20cm}{
l@{\extracolsep\fill}|ccccccccccccccccccccc|c
}
\hline
\rule{0cm}{0.5cm}
& \hspace{-0.2cm} $\phi_{1}$ \hspace{-0.2cm} &
\hspace{-0.2cm} $\phi_{2}$ \hspace{-0.2cm} & \hspace{-0.2cm} $\phi_{3}$ \hspace{-0.2cm} &
\hspace{-0.2cm} $\phi_{4}$ \hspace{-0.2cm} & \hspace{-0.2cm} $\phi_{5}$ \hspace{-0.2cm} & 
\hspace{-0.2cm} $\phi_{6}$ \hspace{-0.2cm} & \hspace{-0.2cm} $\phi_{7}$\hspace{-0.2cm}  &
\hspace{-0.2cm} $\phi_{8}$ \hspace{-0.2cm} & \hspace{-0.2cm} $\phi_{9}$ \hspace{-0.2cm} &
\hspace{-0.2cm} $\phi_{10}$ \hspace{-0.2cm} & 
\hspace{-0.2cm} $\phi_{11}$ \hspace{-0.2cm} & \hspace{-0.2cm} $\phi_{12}$ \hspace{-0.2cm} &
\hspace{-0.2cm} $\phi_{13}$ \hspace{-0.2cm} & \hspace{-0.2cm} $\phi_{14}$ \hspace{-0.2cm} &
\hspace{-0.2cm} $\phi_{15}$ \hspace{-0.2cm} & 
\hspace{-0.2cm} $\phi_{16}$ \hspace{-0.2cm} & \hspace{-0.2cm} $\phi_{17}$ \hspace{-0.2cm} &
\hspace{-0.2cm} $\phi_{18}$ \hspace{-0.2cm} & \hspace{-0.2cm} $\phi_{19}$ \hspace{-0.2cm} &
\hspace{-0.2cm} $\phi_{20}$ \hspace{-0.2cm} & 
\hspace{-0.2cm} $\phi_{21}$ \hspace{-0.2cm} & Number
\rule[-0.1cm]{0cm}{0.5cm}\\
\hline
\rule{0cm}{0.5cm}
$\{\chi_{2,1}\}$ & 
$1$ & $.$ & $.$ & $1$ & $1$ & $1$ & $1$ & $.$ & $.$ & $.$ & $.$ & $.$
& $.$ & $.$ & $.$ & $.$ & $.$ & $.$ & $.$ & $.$ & $.$ & $\frac{1}{2}(\ell^f-1)$
\rule[-0.1cm]{0cm}{0.5cm}\\
\rule{0cm}{0.5cm}
$\{\chi_{2,\frac{q}{\sqrt{2}}(q^2-1)_a}\}$ & 
$.$ & $1$ & $.$ & $.$ & $.$ & $.$ & $.$ & $.$ & $.$ & $.$ & $.$ & $.$
& $.$ & $.$ & $.$ & $.$ & $.$ & $.$ & $1$ & $.$ & $.$ & $\frac{1}{2}(\ell^f-1)$
\rule[-0.1cm]{0cm}{0.5cm}\\
\rule{0cm}{0.5cm}
$\{\chi_{2,\frac{q}{\sqrt{2}}(q^2-1)_b}\}$ & 
$.$ & $.$ & $1$ & $.$ & $.$ & $.$ & $.$ & $.$ & $.$ & $.$ & $.$ & $.$
& $.$ & $.$ & $.$ & $.$ & $.$ & $.$ & $.$ & $1$ & $.$ & $\frac{1}{2}(\ell^f-1)$
\rule[-0.1cm]{0cm}{0.5cm}\\
\hline
\end{tabular*} 
\end{center}

\newpage


\begin{table}[!ht]
\caption[]{The decomposition numbers of the unipotent blocks of $G$
for $3 \neq \ell \, | \, q^2+1$. In the right most column, $\ell^f$ is the
largest power of $\ell$ dividing $q^2+1$.} \label{tab:decuni2} 

\vspace{-0.4cm}

\begin{center}
\begin{tabular*}{20cm}{
l@{\extracolsep\fill}|ccccccccccccccccccccc|c
}
\hline
& \hspace{-0.2cm} $\phi_{1}$ \hspace{-0.2cm} & \hspace{-0.2cm}
$\phi_{2}$ \hspace{-0.2cm} & \hspace{-0.2cm}
$\phi_{3}$ \hspace{-0.2cm} & \hspace{-0.2cm}
$\phi_{4}$ \hspace{-0.2cm} & \hspace{-0.2cm}
$\phi_{5}$ \hspace{-0.2cm} & \hspace{-0.2cm}
$\phi_{6}$ \hspace{-0.2cm} & \hspace{-0.2cm}
$\phi_{7}$ \hspace{-0.2cm} & \hspace{-0.2cm}
$\phi_{8}$ \hspace{-0.2cm} & \hspace{-0.2cm}
$\phi_{9}$ \hspace{-0.2cm} & \hspace{-0.2cm}
$\phi_{10}$ \hspace{-0.2cm} & \hspace{-0.2cm}
$\phi_{11}$ \hspace{-0.2cm} & \hspace{-0.2cm}
$\phi_{12}$ \hspace{-0.2cm} & \hspace{-0.2cm}
$\phi_{13}$ \hspace{-0.2cm} & \hspace{-0.2cm}
$\phi_{14}$ \hspace{-0.2cm} & \hspace{-0.2cm}
$\phi_{15}$ \hspace{-0.2cm} & \hspace{-0.2cm}
$\phi_{16}$ \hspace{-0.2cm} & \hspace{-0.2cm}
$\phi_{17}$ \hspace{-0.2cm} & \hspace{-0.2cm}
$\phi_{18}$ \hspace{-0.2cm} & \hspace{-0.2cm}
$\phi_{19}$ \hspace{-0.2cm} & \hspace{-0.2cm} 
$\phi_{20}$ \hspace{-0.2cm} & 
\hspace{-0.2cm} $\phi_{21}$ \hspace{-0.2cm} &\\
Series \,\, & \hspace{-0.2cm} $ps$ \hspace{-0.2cm} & \hspace{-0.2cm} ${^2B_2}[a]$ \hspace{-0.2cm}
& \hspace{-0.3cm} ${^2B_2}[b]$ \hspace{-0.2cm} & \hspace{-0.2cm} $A_1$ \hspace{-0.2cm} & $ps$ & $ps$ & $ps$
 & $c$ & $c$ & $c$ & $c$ & $c$ & $c$ & $c$ & $c$ & $c$ & $c$
& $c$ & \hspace{-0.2cm} ${^2B_2}[a]$ \hspace{-0.2cm} & \hspace{-0.3cm} ${^2B_2}[b]$
 \hspace{-0.2cm} & \hspace{-0.2cm} $c$ \hspace{-0.2cm} & Number  \\
\hline
$\chi_{1}$ & $1$ & \multicolumn{1}{|c}{$.$} & $.$ & $.$ & $.$ & $.$ & $.$ & $.$ & $.$ & $.$ &
$.$ & $.$ & $.$ & $.$ & $.$ & $.$ & $.$ & $.$ & $.$ & $.$ & $.$ & $1$\\
\hline
$\chi_{2}$ & $.$ & \multicolumn{1}{|c}{$1$} & \multicolumn{1}{c|}{$.$} & $.$ & $.$ & $.$ & $.$ & $.$ & $.$ & $.$ &
$.$ & $.$ & $.$ & $.$ & $.$ & $.$ & $.$ & $.$ & $.$ & $.$ & $.$ & $1$\\
$\chi_{3}$ & $.$ & \multicolumn{1}{|c}{$.$} & \multicolumn{1}{c|}{$1$} & $.$ & $.$ & $.$ & $.$ & $.$ & $.$ & $.$ &
$.$ & $.$ & $.$ & $.$ & $.$ & $.$ & $.$ & $.$ & $.$ & $.$ & $.$ & $1$\\
\hline
\end{tabular*} 
\end{center}
\end{table}

\newpage


\begin{table}[!ht]
\caption[]{The decomposition numbers of the unipotent blocks of $G$
for $\ell \, | \, q^2+\sqrt{2}q+1$. In the right most column, $\ell^f$ is the
largest power of $\ell$ dividing $q^2+\sqrt{2}q+1$.} \label{tab:decuni3} 

\begin{center}
\begin{tabular*}{20cm}{
l@{\extracolsep\fill}|ccccccccccccccccccccc|c
}
\hline
\rule{0cm}{0.19cm}
& \hspace{-0.2cm} $\phi_{1}$ \hspace{-0.2cm} & $\phi_{2}$ & $\phi_{3}$ & $\phi_{4}$ & $\phi_{5}$ &
$\phi_{6}$ & $\phi_{7}$ & $\phi_{8}$ & $\phi_{9}$ & $\phi_{10}$ &
$\phi_{11}$ & $\phi_{12}$ & $\phi_{13}$ & $\phi_{14}$ & $\phi_{15}$ &
$\phi_{16}$ & $\phi_{17}$ & \hspace{-0.2cm} $\phi_{18}$ \hspace{-0.2cm} & $\phi_{19}$ & $\phi_{20}$ &
\hspace{-0.2cm} $\phi_{21}$ \hspace{-0.2cm} &
\rule[-0.1cm]{0cm}{0.32cm}\\
\rule{0cm}{0.32cm}
Series \,\, & \hspace{-0.2cm} $ps$ \hspace{-0.2cm} & \hspace{-0.3cm} ${^2B_2}[a]$ \hspace{-0.3cm}
& \hspace{-0.2cm} ${^2B_2}[b]$ \hspace{-0.3cm} & \hspace{-0.2cm}
$ps$ \hspace{-0.2cm} & \hspace{-0.3cm} ${^2B_2}[\text{St}]$ \hspace{-0.3cm} & \hspace{-0.2cm}
$ps$ \hspace{-0.2cm} & \hspace{-0.2cm} $ps$ \hspace{-0.2cm} & $c$ &
$c$ & $c$ & $c$ & $c$ & $c$ & $c$ & $c$ & $c$ & $c$ & \hspace{-0.2cm} $c$ \hspace{-0.2cm} 
& \hspace{-0.2cm} $c$ \hspace{-0.2cm} & \hspace{-0.2cm} $c$
 \hspace{-0.2cm} & \hspace{-0.2cm} $c$ \hspace{-0.2cm} & Number  
\rule[-0.1cm]{0cm}{0.32cm}\\
\hline
\rule{0cm}{0.32cm}
$\chi_{1}$ & $1$ & \multicolumn{1}{|c}{$.$} & $.$ & $.$ & $.$ & $.$ & $.$ & $.$ & $.$ & $.$ &
$.$ & $.$ & $.$ & $.$ & $.$ & $.$ & $.$ & $.$ & $.$ & $.$ & $.$ & $1$
\rule[-0.1cm]{0cm}{0.32cm}\\
\hline
\rule{0cm}{0.32cm}
$\chi_{2}$ & $.$ & \multicolumn{1}{|c}{$1$} & \multicolumn{1}{c|}{$.$} & $.$ & $.$ & $.$ & $.$ & $.$ & $.$ & $.$ &
$.$ & $.$ & $.$ & $.$ & $.$ & $.$ & $.$ & $.$ & $.$ & $.$ & $.$ & $1$
\rule[-0.1cm]{0cm}{0.32cm}\\
\rule{0cm}{0.32cm}
$\chi_{3}$ & $.$ & \multicolumn{1}{|c}{$.$} & \multicolumn{1}{c|}{$1$} & $.$ & $.$ & $.$ & $.$ & $.$ & $.$ & $.$ &
$.$ & $.$ & $.$ & $.$ & $.$ & $.$ & $.$ & $.$ & $.$ & $.$ & $.$ & $1$
\rule[-0.1cm]{0cm}{0.32cm}\\
\hline
\end{tabular*} 
\end{center}
\end{table}

\newpage

\noindent {\sc Table~C.3.} (\textit{continued})

\begin{center}
\begin{tabular*}{20cm}{
l@{\extracolsep\fill}|ccccccccccccccccccccc|c
}
\hline
\rule{0cm}{0.5cm}
& \hspace{-0.2cm} $\phi_{1}$ \hspace{-0.2cm} &
\hspace{-0.2cm} $\phi_{2}$ \hspace{-0.2cm} & \hspace{-0.2cm} $\phi_{3}$ \hspace{-0.2cm} &
\hspace{-0.2cm} $\phi_{4}$ \hspace{-0.2cm} & \hspace{-0.2cm} $\phi_{5}$ \hspace{-0.2cm} & 
\hspace{-0.2cm} $\phi_{6}$ \hspace{-0.2cm} & \hspace{-0.2cm} $\phi_{7}$\hspace{-0.2cm}  &
\hspace{-0.2cm} $\phi_{8}$ \hspace{-0.2cm} & \hspace{-0.2cm} $\phi_{9}$ \hspace{-0.2cm} &
\hspace{-0.2cm} $\phi_{10}$ \hspace{-0.2cm} & 
\hspace{-0.2cm} $\phi_{11}$ \hspace{-0.2cm} & \hspace{-0.2cm} $\phi_{12}$ \hspace{-0.2cm} &
\hspace{-0.2cm} $\phi_{13}$ \hspace{-0.2cm} & \hspace{-0.2cm} $\phi_{14}$ \hspace{-0.2cm} &
\hspace{-0.2cm} $\phi_{15}$ \hspace{-0.2cm} & 
\hspace{-0.2cm} $\phi_{16}$ \hspace{-0.2cm} & \hspace{-0.2cm} $\phi_{17}$ \hspace{-0.2cm} &
\hspace{-0.2cm} $\phi_{18}$ \hspace{-0.2cm} & \hspace{-0.2cm} $\phi_{19}$ \hspace{-0.2cm} &
\hspace{-0.2cm} $\phi_{20}$ \hspace{-0.2cm} & 
\hspace{-0.2cm} $\phi_{21}$ \hspace{-0.2cm} & Number
\rule[-0.1cm]{0cm}{0.5cm}\\
\hline
\rule{0cm}{0.5cm}
$\{\chi_{10,1}\}$ & 
$.$ & $.$ & $.$ & $.$ & $1$ & $.$ & $.$ & $.$ & $1$ & $1$ & $1$ & $1$ & $.$ & $.$ &
$.$ & $.$ & $.$ & $.$ & $.$ & $.$ & $.$ 
& $\frac{1}{4}(\ell^f-1)$
\rule[-0.1cm]{0cm}{0.5cm}\\
\rule{0cm}{0.5cm}
$\{\chi_{10,\frac{q}{\sqrt{2}}(q^2-1)_a}\}$ & 
$.$ & $.$ & $.$ & $.$ & $.$ & $.$ & $.$ & $.$ & $*$ & $.$ & $*$ & $*$ & $*$ & $.$ &
$.$ & $.$ & $*$ & $j-1$ & $1$ & $.$ & $.$
& $\frac{1}{4}(\ell^f-1)$
\rule[-0.1cm]{0cm}{0.5cm}\\
\rule{0cm}{0.5cm}
$\{\chi_{10,\frac{q}{\sqrt{2}}(q^2-1)_b}\}$ & 
$.$ & $.$ & $.$ & $.$ & $.$ & $.$ & $.$ & $.$ & $*$ & $.$ & $*$ & $*$ & $.$ &
\hspace{-0.2cm} $h-1$ \hspace{-0.2cm} & $.$ & $.$ & $*$ & $j-1$ & $.$ & $1$ & $.$
& $\frac{1}{4}(\ell^f-1)$
\rule[-0.1cm]{0cm}{0.5cm}\\
\hline
\end{tabular*} 
\end{center}

\newpage


\begin{table}[!ht]
\caption[]{The decomposition numbers of the unipotent blocks of $G$
for $\ell \, | \, q^2-\sqrt{2}q+1$. In the right most column, $\ell^f$ is the
largest power of $\ell$ dividing $q^2-\sqrt{2}q+1$.} \label{tab:decuni4} 

\begin{center}
\begin{tabular*}{20cm}{
l@{\extracolsep\fill}|ccccccccccccccccccccc|c
}
\hline
\rule{0cm}{0.32cm}
& \hspace{-0.2cm} $\phi_{1}$ \hspace{-0.2cm} & $\phi_{2}$ & $\phi_{3}$ & $\phi_{4}$ & $\phi_{5}$ &
$\phi_{6}$ & $\phi_{7}$ & $\phi_{8}$ & $\phi_{9}$ & $\phi_{10}$ &
$\phi_{11}$ & $\phi_{12}$ & $\phi_{13}$ & $\phi_{14}$ & $\phi_{15}$ &
$\phi_{16}$ & $\phi_{17}$ & \hspace{-0.2cm}
$\phi_{18}$ \hspace{-0.2cm} & $\phi_{19}$ \hspace{-0.15cm}
& \hspace{-0.15cm} $\phi_{20}$ & \hspace{-0.2cm} $\phi_{21}$ \hspace{-0.2cm} &
\rule[-0.1cm]{0cm}{0.32cm}\\
\rule{0cm}{0.32cm}
Series \,\, & \hspace{-0.2cm} $ps$ \hspace{-0.2cm} & \hspace{-0.2cm} ${^2B_2}[a]$ \hspace{-0.25cm}
& \hspace{-0.25cm} ${^2B_2}[b]$ \hspace{-0.2cm} & $ps$ & $ps$ \hspace{-0.2cm}
& \hspace{-0.25cm} ${^2B_2}[\text{St}]$ \hspace{-0.2cm} & $ps$ & $c$ & $c$ & $c$ &
$c$ & $c$ & $c$ & $c$ & $c$ & $c$ & $c$ 
& \hspace{-0.2cm} $c$ \hspace{-0.2cm} 
& \hspace{-0.2cm} $c$ \hspace{-0.2cm} & \hspace{-0.2cm}
$c$ \hspace{-0.2cm} & \hspace{-0.2cm} $c$ \hspace{-0.2cm} & Number  
\rule[-0.1cm]{0cm}{0.32cm}\\
\hline
\rule{0cm}{0.32cm}
$\chi_{1}$ & $1$ & \multicolumn{1}{|c}{$.$} & $.$ & $.$ & $.$ & $.$ & $.$ & $.$ & $.$ & $.$ &
$.$ & $.$ & $.$ & $.$ & $.$ & $.$ & $.$ & $.$ & $.$ & $.$ & $.$ & $1$
\rule[-0.1cm]{0cm}{0.32cm}\\
\hline
\rule{0cm}{0.32cm}
$\chi_{2}$ & $.$ & \multicolumn{1}{|c}{$1$} & \multicolumn{1}{c|}{$.$} & $.$ & $.$ & $.$ & $.$ & $.$ & $.$ & $.$ &
$.$ & $.$ & $.$ & $.$ & $.$ & $.$ & $.$ & $.$ & $.$ & $.$ & $.$ & $1$
\rule[-0.1cm]{0cm}{0.32cm}\\
\rule{0cm}{0.32cm}
$\chi_{3}$ & $.$ & \multicolumn{1}{|c}{$.$} & \multicolumn{1}{c|}{$1$} & $.$ & $.$ & $.$ & $.$ & $.$ & $.$ & $.$ &
$.$ & $.$ & $.$ & $.$ & $.$ & $.$ & $.$ & $.$ & $.$ & $.$ & $.$ & $1$
\rule[-0.1cm]{0cm}{0.32cm}\\
\hline
\rule{0cm}{0.32cm}
$\chi_{4}$ & $.$ & $.$ & \multicolumn{1}{c|}{$.$} & \multicolumn{1}{c|}{$1$} & $.$ & $.$ & $.$ & $.$ & $.$ & $.$ &
$.$ & $.$ & $.$ & $.$ & $.$ & $.$ & $.$ & $.$ & $.$ & $.$ & $.$ & $1$
\rule[-0.1cm]{0cm}{0.32cm}\\
\hline
\end{tabular*} 
\end{center}
\end{table}

\newpage


\noindent {\sc Table~C.4.} (\textit{continued})

\begin{center}
\begin{tabular*}{20cm}{
l@{\extracolsep\fill}|ccccccccccccccccccccc|c
}
\hline
\rule{0cm}{0.5cm}
& \hspace{-0.2cm} $\phi_{1}$ \hspace{-0.2cm} &
\hspace{-0.2cm} $\phi_{2}$ \hspace{-0.2cm} & \hspace{-0.2cm} $\phi_{3}$ \hspace{-0.2cm} &
\hspace{-0.2cm} $\phi_{4}$ \hspace{-0.2cm} & \hspace{-0.2cm} $\phi_{5}$ \hspace{-0.2cm} & 
\hspace{-0.2cm} $\phi_{6}$ \hspace{-0.2cm} & \hspace{-0.2cm} $\phi_{7}$\hspace{-0.2cm}  &
\hspace{-0.2cm} $\phi_{8}$ \hspace{-0.2cm} & \hspace{-0.2cm} $\phi_{9}$ \hspace{-0.2cm} &
\hspace{-0.2cm} $\phi_{10}$ \hspace{-0.2cm} & 
\hspace{-0.2cm} $\phi_{11}$ \hspace{-0.2cm} & \hspace{-0.2cm} $\phi_{12}$ \hspace{-0.2cm} &
\hspace{-0.2cm} $\phi_{13}$ \hspace{-0.2cm} & \hspace{-0.2cm} $\phi_{14}$ \hspace{-0.2cm} &
\hspace{-0.2cm} $\phi_{15}$ \hspace{-0.2cm} & 
\hspace{-0.2cm} $\phi_{16}$ \hspace{-0.2cm} & \hspace{-0.2cm} $\phi_{17}$ \hspace{-0.2cm} &
\hspace{-0.2cm} $\phi_{18}$ \hspace{-0.2cm} & \hspace{-0.2cm} $\phi_{19}$ \hspace{-0.2cm} &
\hspace{-0.2cm} $\phi_{20}$ \hspace{-0.2cm} & 
\hspace{-0.2cm} $\phi_{21}$ \hspace{-0.2cm} & Number
\rule[-0.1cm]{0cm}{0.5cm}\\
\hline
\rule{0cm}{0.5cm}
$\{\chi_{8,1}\}$ & 
$.$ & $1$ & $1$ & $.$ & $.$ & $1$ & $.$ & $1$ & $.$ & $1$ & $.$ & $.$ & $1$ & $1$ &
$.$ & $.$ & $.$ & $.$ & $.$ & $.$ & $.$ & $\frac{1}{4}(\ell^f-1)$
\rule[-0.1cm]{0cm}{0.5cm}\\
\rule{0cm}{0.5cm}
$\{\chi_{8,\frac{q}{\sqrt{2}}(q^2-1)_a}\}$ & 
$.$ & $.$ & $.$ & $.$ & $.$ & $.$ & $.$ & $*$ & $.$ & $.$ & $*$ & $.$ & $*$ &
$*$ & $.$ & $.$ & $*$ & $*$ & $1$ & $.$ & $.$ & $\frac{1}{4}(\ell^f-1)$
\rule[-0.1cm]{0cm}{0.5cm}\\
\rule{0cm}{0.5cm}
$\{\chi_{8,\frac{q}{\sqrt{2}}(q^2-1)_b}\}$ & 
$.$ & $.$ & $.$ & $.$ & $.$ & $.$ & $.$ & $*$ & $.$ & $.$ & $.$ & $*$ & $*$ &
$*$ & $.$ & $.$ & $*$ & $*$ & $.$ & $1$ & $.$ & $\frac{1}{4}(\ell^f-1)$
\rule[-0.1cm]{0cm}{0.5cm}\\
\hline
\end{tabular*} 
\end{center}

\newpage


\begin{table}[!ht]
\caption[]{The decomposition numbers of the unipotent blocks of $G$
for $\ell = 3$. In the right most column, $3^f$ is the largest
power of $3$ dividing $q^2+1$.} \label{tab:decunil3}  

\begin{center}
\begin{tabular*}{20cm}{
l@{\extracolsep\fill}|ccccccccccccccccccc|c
}
\hline
\rule{0cm}{0.32cm}
& \hspace{-0.2cm} $\phi_{1}$ \hspace{-0.2cm} & $\phi_{2}$ & $\phi_{3}$ & $\phi_{4}$ & $\phi_{5,1}$ &
$\phi_{5}$ & $\phi_{6}$ & $\phi_{7}$ & $\phi_{8}$ &
$\phi_{10}$ & $\phi_{11}$ & $\phi_{12}$ & $\phi_{13}$ & $\phi_{14}$ &
$\phi_{15}$ & \hspace{-0.2cm} $\phi_{18}$ \hspace{-0.2cm} & $\phi_{19}$ & $\phi_{20}$ &
\hspace{-0.2cm} $\phi_{21}$ \hspace{-0.2cm} &
\rule[-0.1cm]{0cm}{0.32cm}\\
\rule{0cm}{0.32cm}
Series \,\, & \hspace{-0.2cm} $ps$ \hspace{-0.2cm} & \hspace{-0.2cm}
${^2B_2}[a]$ \hspace{-0.2cm} & \hspace{-0.2cm}
${^2B_2}[b]$ \hspace{-0.2cm} & $A_1$ & $c$ & \hspace{-0.2cm}
$ps$ \hspace{-0.2cm} & $ps$ & $ps$ & $c$ & $c$ & $c$ 
& $c$ & $c$ & $c$ & $c$ & \hspace{-0.2cm} $c$ \hspace{-0.2cm} 
& \hspace{-0.2cm} ${^2B_2}[a]$ \hspace{-0.2cm} & \hspace{-0.2cm} ${^2B_2}[b]$ 
\hspace{-0.2cm} & \hspace{-0.2cm} $c$ \hspace{-0.2cm} & Number  
\rule[-0.1cm]{0cm}{0.32cm}\\
\hline
\rule{0cm}{0.32cm}
$\chi_{1}$ & $1$ & $.$ & $.$ & $.$ & $.$ & $.$ & $.$ & $.$ &
$.$ & $.$ & $.$ & $.$ & $.$ & $.$ & $.$ & $.$ & $.$ & $.$ & $.$ & $1$
\rule[-0.1cm]{0cm}{0.32cm}\\
\rule{0cm}{0.32cm}
$\chi_{2}$ & $.$ & $1$ & $.$ & $.$ & $.$ & $.$ & $.$ & $.$ &
$.$ & $.$ & $.$ & $.$ & $.$ & $.$ & $.$ & $.$ & $.$ & $.$ & $.$ & $1$
\rule[-0.1cm]{0cm}{0.32cm}\\
\rule{0cm}{0.32cm}
$\chi_{3}$ & $.$ & $.$ & $1$ & $.$ & $.$ & $.$ & $.$ & $.$ &
$.$ & $.$ & $.$ & $.$ & $.$ & $.$ & $.$ & $.$ & $.$ & $.$ & $.$ & $1$
\rule[-0.1cm]{0cm}{0.32cm}\\
\rule{0cm}{0.32cm}
$\chi_{4}$ & $1$ & $.$ & $.$ & $1$ & $.$ & $.$ & $.$ & $.$ &
$.$ & $.$ & $.$ & $.$ & $.$ & $.$ & $.$ & $.$ & $.$ & $.$ & $.$ & $1$
\rule[-0.1cm]{0cm}{0.32cm}\\
\rule{0cm}{0.32cm}
$\chi_{5,1}$ & $.$ & $.$ & $.$ & $1$ & $1$ & $.$ & $.$ & $.$ &
$.$ & $.$ & $.$ & $.$ & $.$ & $.$ & $.$ & $.$ & $.$ & $.$ & $.$ & $1$
\rule[-0.1cm]{0cm}{0.32cm}\\
\hline
\end{tabular*} 
\end{center}
\end{table}

\newpage


\noindent {\sc Table~C.5.} (\textit{continued})

\begin{center}
\begin{tabular*}{20cm}{
l@{\extracolsep\fill}|ccccccccccccccccccc|c
}
\hline
\rule{0cm}{0.5cm}
& \hspace{-0.2cm} $\phi_{1}$ \hspace{-0.2cm} &
\hspace{-0.2cm} $\phi_{2}$ \hspace{-0.2cm} & \hspace{-0.2cm} $\phi_{3}$ \hspace{-0.2cm} &
\hspace{-0.2cm} $\phi_{4}$ \hspace{-0.2cm} & \hspace{-0.2cm} $\phi_{5,1}$ \hspace{-0.2cm} & 
\hspace{-0.2cm} $\phi_{5}$ \hspace{-0.2cm} & \hspace{-0.2cm} $\phi_{6}$\hspace{-0.2cm}  &
\hspace{-0.2cm} $\phi_{7}$ \hspace{-0.2cm} & 
\hspace{-0.2cm} $\phi_{8}$ \hspace{-0.2cm} & 
\hspace{-0.2cm} $\phi_{10}$ \hspace{-0.2cm} & \hspace{-0.2cm} $\phi_{11}$ \hspace{-0.2cm} &
\hspace{-0.2cm} $\phi_{12}$ \hspace{-0.2cm} & \hspace{-0.2cm} $\phi_{13}$ \hspace{-0.2cm} &
\hspace{-0.2cm} $\phi_{14}$ \hspace{-0.2cm} & 
\hspace{-0.2cm} $\phi_{15}$ \hspace{-0.2cm} &
\hspace{-0.2cm} $\phi_{18}$ \hspace{-0.2cm} & \hspace{-0.2cm} $\phi_{19}$ \hspace{-0.2cm} &
\hspace{-0.2cm} $\phi_{20}$ \hspace{-0.2cm} & 
\hspace{-0.2cm} $\phi_{21}$ \hspace{-0.2cm} & Number
\rule[-0.1cm]{0cm}{0.5cm}\\
\hline
\rule{0cm}{0.5cm}
$\chi_{9}$ & 
$.$ & $.$ & $.$ & $.$ & $*$ & $.$ & $.$ & $.$ & $.$ & $1$ &
$.$ & $.$ & $.$ & $.$ & $.$ & $.$ & $.$ & $.$ & $.$ & $1$
\rule[-0.1cm]{0cm}{0.5cm}\\
\rule{0cm}{0.5cm}
$\chi_{16}$ & 
$.$ & $.$ & $.$ & $.$ & $x_{15}$ & $.$ & $.$ & $.$ & $.$ & $.$ & $.$ &
$.$ & $.$ & $.$ & $1$ & $.$ & $.$ & $.$ & $.$ & $1$
\rule[-0.1cm]{0cm}{0.5cm}\\
\rule{0cm}{0.5cm}
$\chi_{17}$ & 
$.$ & $.$ & $.$ & $.$ & $x_{15}-1$ & $.$ & $.$ & $.$ & $.$ & $.$ &
$.$ & $.$ & $.$ & $.$ & $1$ & $.$ & $.$ & $.$ & $.$ & $1$
\rule[-0.1cm]{0cm}{0.5cm}\\
\hline
\rule{0cm}{0.5cm}
$\chi_{5,q^2(q^2-1)}$ & 
$.$ & $.$ & $.$ & $.$ & $*$ & $.$ & $.$ &
$.$ & $1$ & $1$ & $.$ & $.$ & $.$ & $.$ & $a-2$ & $1$ & $.$ & $.$ &
$.$ & $1$ 
\rule[-0.1cm]{0cm}{0.5cm}\\
\rule{0cm}{0.5cm}
$\chi_{5,St}$ & 
$.$ & $.$ & $.$ & $1$ & $x_{21}-x_{18}-1$ & $.$ & $.$ & $.$ & $b$ &
$c$ & $.$ & $.$ & $.$ & $.$ & $d-a$ & $e-1$ & $.$ & $.$ & $1$ & $1$ 
\rule[-0.1cm]{0cm}{0.5cm}\\
\hline
\end{tabular*} 
\end{center}
\end{landscape}

\newpage


\section*{Appendix D: Decomposition numbers of non-unipotent characters}
\label{sec:appendixD}

\setcounter{table}{0}
\renewcommand{\thetable}{\textrm{D.\arabic{table}}}

\begin{table}[!ht]
\caption[]{The decomposition numbers of the irreducible
characters in $\mathcal{E}(G, s_2) = \{\chi_{2,1},
\chi_{2,\frac{q}{\sqrt{2}}(q^2-1)_a}, \chi_{2,\frac{q}{\sqrt{2}}(q^2-1)_b}, 
\chi_{2,St}\}$ where $s_2 \in G^*$ is a semisimple $\ell'$-element 
of type~$g_2$.}  
\label{tab:dec2}

\begin{center}
\begin{tabular*}{12.5cm}{
l@{\extracolsep\fill}|cccc|cccc|cccc
}
\hline
\rule{0cm}{0.4cm}
& \multicolumn{4}{l|}{$\ell \, | \, q^2+\sqrt{2}q+1$} &
  \multicolumn{4}{l|}{$\ell \, | \, q^2-\sqrt{2}q+1$} &
  \multicolumn{4}{l}{otherwise}\\
& \hspace{-0.15cm} $\phi_{2,1}$ \hspace{-0.15cm} & \hspace{-0.15cm} $\phi_{2,2}$ \hspace{-0.15cm} & \hspace{-0.15cm} $\phi_{2,3}$ \hspace{-0.15cm} & \hspace{-0.15cm} $\phi_{2,4}$ \hspace{-0.15cm}
& \hspace{-0.15cm} $\phi_{2,1}$ \hspace{-0.15cm} & \hspace{-0.15cm} $\phi_{2,2}$ \hspace{-0.15cm} & \hspace{-0.15cm} $\phi_{2,3}$ \hspace{-0.15cm} & \hspace{-0.15cm} $\phi_{2,4}$ \hspace{-0.15cm}
& \hspace{-0.15cm} $\phi_{2,1}$ \hspace{-0.15cm} & \hspace{-0.15cm} $\phi_{2,2}$ \hspace{-0.15cm} & \hspace{-0.15cm} $\phi_{2,3}$ \hspace{-0.15cm} & \hspace{-0.15cm} $\phi_{2,4}$ \hspace{-0.15cm}
\rule[-0.2cm]{0cm}{0.36cm}\\
\hline
\rule{0cm}{0.36cm}
$\chi_{2,1}$ & $1$ & $.$ & $.$ & $.$ & $1$ & $.$ & $.$ & $.$ & $1$ & $.$ & $.$ & $.$ 
\rule[-0.2cm]{0cm}{0.36cm}\\
\rule{0cm}{0.36cm}
$\chi_{2,\frac{q}{\sqrt{2}}(q^2-1)_a}$ & $.$ & $1$ & $.$ & $.$ & $.$ &
$1$ & $.$ & $.$ & $.$ & $1$ & $.$ & $.$ 
\rule[-0.2cm]{0cm}{0.36cm}\\
\rule{0cm}{0.36cm}
$\chi_{2,\frac{q}{\sqrt{2}}(q^2-1)_b}$ & $.$ & $.$ & $1$ & $.$ & $.$ &
$.$ & $1$ & $.$ & $.$ & $.$ & $1$ & $.$ 
\rule[-0.2cm]{0cm}{0.36cm}\\
\rule{0cm}{0.36cm}
$\chi_{2,St}$ & $1$ & $1$ & $1$ & $1$ & $1$ & $.$ & $.$ & $1$ & $.$ &
$.$ & $.$ & $1$ 
\rule[-0.2cm]{0cm}{0.36cm}\\
\hline
\end{tabular*} 
\end{center}
\end{table}


\begin{table}[!ht]
\caption[]{The decomposition numbers of the irreducible characters
in $\mathcal{E}(G, s_3) = \{\chi_{3,1}, \chi_{3,St}\}$ where
$s_3 \in G^*$ is a semisimple $\ell'$-element of type~$g_3$.} \label{tab:dec3}

\begin{center}
\begin{tabular*}{12.5cm}{
l@{\extracolsep\fill}|cc|cc|cc
}
\hline
\rule{0cm}{0.4cm}
& \multicolumn{2}{l|}{$\ell \, | \, q^2-1$} & 
  \multicolumn{2}{l|}{$\ell \, | \, q^2+1$} & 
  \multicolumn{2}{l}{otherwise}\\
& $\phi_{3,1}$ & $\phi_{3,2}$ & 
  $\phi_{3,1}$ & $\phi_{3,2}$ & 
  $\phi_{3,1}$ & $\phi_{3,2}$ 
\rule[-0.2cm]{0cm}{0.36cm}\\
\hline
\rule{0cm}{0.36cm}
$\chi_{3,1}$ \hspace{0.8cm} & $1$ & $.$ & $1$ & $.$ & $1$ & $.$ 
\rule[-0.2cm]{0cm}{0.36cm}\\
\rule{0cm}{0.36cm}
$\chi_{3,St}$ & $.$ & $1$ & $1$ & $1$ & $.$ & $1$
\rule[-0.2cm]{0cm}{0.36cm}\\
\hline
\end{tabular*} 
\end{center}
\end{table}


\begin{table}[!ht]
\caption[]{The decomposition numbers of the irreducible characters~in
$\mathcal{E}_\ell(G, s_5) = \{\chi_{5,1}, \chi_{5,q^2(q^2-1)},
\chi_{5,St}, \chi_{6,1}, \chi_{6,St}, \chi_{15,1}\}$ where $s_5 \in G^*$ is a
semisimple $\ell'$-element of type~$g_5$ and 
$3 \neq \ell \, | \, q^2+1$. In the right most column, $\ell^f$ is the 
largest power of $\ell$ dividing $q^2+1$.} 
\label{tab:decg5noncyc} 

\begin{center}
\begin{tabular*}{12.5cm}{
l@{\extracolsep\fill}|ccc|l
}
\hline
\rule{0cm}{0.4cm}
& \multicolumn{3}{l|}{$3 \neq \ell \, | \, q^2+1$} & \\
& $\phi_{5,1}$ & $\phi_{5,2}$ & $\phi_{5,3}$ & Number
\rule[-0.2cm]{0cm}{0.36cm}\\
\hline
\rule{0cm}{0.36cm}
$\chi_{5,1}$ & $1$ & $.$ & $.$ & $1$
\rule[-0.2cm]{0cm}{0.36cm}\\
\rule{0cm}{0.36cm}
$\chi_{5,q^2(q^2-1)}$ & $.$ & $1$ & $.$ & $1$
\rule[-0.2cm]{0cm}{0.36cm}\\
\rule{0cm}{0.36cm}
$\chi_{5,St}$ & $1$ & $a'$ & $1$ & $1$
\rule[-0.2cm]{0cm}{0.36cm}\\
\hline
\rule{0cm}{0.4cm}
$\{\chi_{6,1}\}$ & $1$ & $1$ & $.$ & $\ell^f-1$
\rule[-0.2cm]{0cm}{0.36cm}\\
\rule{0cm}{0.36cm}
$\{\chi_{6,St}\}$ & $1$ & $a'-1$ & $1$ & $\ell^f-1$
\rule[-0.2cm]{0cm}{0.36cm}\\
\hline
\rule{0cm}{0.4cm}
$\{\chi_{15,1}\}$ & $.$ & $a'-2$ & $1$ & $\frac{1}{6}(\ell^f-1)(\ell^f-2)$
\rule[-0.2cm]{0cm}{0.36cm}\\
\hline
\end{tabular*} 
\end{center}
\end{table}

\newpage


\begin{table}[!ht]
\caption[]{The decomposition numbers of the irreducible characters
in $\mathcal{E}(G, s_5) = \{\chi_{5,1}, \chi_{5,q^2(q^2-1)},
\chi_{5,St}\}$ where $s_5 \in G^*$ is a semisimple $\ell'$-element of
type~$g_5$ and $\ell \nmid q^2+1$.} \label{tab:decg5cyc}

\begin{center}
\begin{tabular*}{12.5cm}{
l@{\extracolsep\fill}|ccc|ccc|ccc
}
\hline
\rule{0cm}{0.4cm}
& \multicolumn{3}{l|}{$\ell \, | \, q^2-1$} & 
  \multicolumn{3}{l|}{$\ell \, | \, q^4-q^2+1$}&
  \multicolumn{3}{l}{otherwise}\\ 
& $\phi_{5,1}$ & $\phi_{5,2}$ & $\phi_{5,3}$ 
& $\phi_{5,1}$ & $\phi_{5,2}$ & $\phi_{5,3}$ 
& $\phi_{5,1}$ & $\phi_{5,2}$ & $\phi_{5,3}$ 
\rule[-0.2cm]{0cm}{0.36cm}\\
\hline
\rule{0cm}{0.36cm}
$\chi_{5,1}$ & $1$ & $.$ & $.$ & $1$ & $.$ & $.$ & $1$ & $.$ & $.$ 
\rule[-0.2cm]{0cm}{0.36cm}\\
\rule{0cm}{0.36cm}
$\chi_{5,q^2(q^2-1)}$ & $.$ & $1$ & $.$ & $.$ & $1$ & $.$ & $.$ & $1$ & $.$ 
\rule[-0.2cm]{0cm}{0.36cm}\\
\rule{0cm}{0.36cm}
$\chi_{5,St}$ & $.$ & $.$ & $1$ & $1$ & $.$ & $1$ & $.$ & $.$ & $1$ 
\rule[-0.2cm]{0cm}{0.36cm}\\
\hline
\end{tabular*} 
\end{center}
\end{table}


\begin{table}[!ht]
\caption[]{The decomposition numbers of the irreducible characters
in $\mathcal{E}(G, s_6) = \{\chi_{6,1}, \chi_{6,St}\}$ where 
$s_6 \in G^*$ is a semisimple $\ell'$-element of type~$g_6$.} \label{tab:dec6}

\begin{center}
\begin{tabular*}{12.5cm}{
l@{\extracolsep\fill}|cc|cc|cc
}
\hline
\rule{0cm}{0.4cm}
& \multicolumn{2}{l|}{$\ell \, | \, q^2-1$} & 
  \multicolumn{2}{l|}{$\ell \, | \, q^2+1$} & 
  \multicolumn{2}{l}{otherwise}\\
& $\phi_{6,1}$ & $\phi_{6,2}$ & 
  $\phi_{6,1}$ & $\phi_{6,2}$ & 
  $\phi_{6,1}$ & $\phi_{6,2}$ 
\rule[-0.2cm]{0cm}{0.36cm}\\
\hline
\rule{0cm}{0.36cm}
$\chi_{6,1}$ \hspace{0.8cm} & $1$ & $.$ & $1$ & $.$ & $1$ & $.$
\rule[-0.2cm]{0cm}{0.36cm}\\
\rule{0cm}{0.36cm}
$\chi_{6,St}$ & $.$ & $1$ & $1$ & $1$ & $.$ & $1$
\rule[-0.2cm]{0cm}{0.36cm}\\
\hline
\end{tabular*} 
\end{center}
\end{table}


\begin{table}[!ht]
\caption[]{The decomposition numbers of the irreducible characters in
$\mathcal{E}(G, s_8) = \{\chi_{8,1},
\chi_{8,\frac{q}{\sqrt{2}}(q^2-1)_a}, \chi_{8,\frac{q}{\sqrt{2}}(q^2-1)_b}, 
\chi_{8,St}\}$ where $s_8 \in G^*$ is a semisimple $\ell'$-element of
type~$g_8$.}  
\label{tab:dec8}

\begin{center}
\begin{tabular*}{12.5cm}{
l@{\extracolsep\fill}|cccc|cccc|cccc
}
\hline
\rule{0cm}{0.4cm}
& \multicolumn{4}{l|}{$\ell \, | \, q^2+\sqrt{2}q+1$} & 
  \multicolumn{4}{l|}{$\ell \, | \, q^2-\sqrt{2}q+1$} &
  \multicolumn{4}{l}{otherwise}\\
& \hspace{-0.15cm} $\phi_{8,1}$ \hspace{-0.15cm} & \hspace{-0.15cm} $\phi_{8,2}$ \hspace{-0.15cm} & \hspace{-0.15cm} $\phi_{8,3}$ \hspace{-0.15cm} & \hspace{-0.15cm} $\phi_{8,4}$ \hspace{-0.15cm}
& \hspace{-0.15cm} $\phi_{8,1}$ \hspace{-0.15cm} & \hspace{-0.15cm} $\phi_{8,2}$ \hspace{-0.15cm} & \hspace{-0.15cm} $\phi_{8,3}$ \hspace{-0.15cm} & \hspace{-0.15cm} $\phi_{8,4}$ \hspace{-0.15cm}
& \hspace{-0.15cm} $\phi_{8,1}$ \hspace{-0.15cm} & \hspace{-0.15cm} $\phi_{8,2}$ \hspace{-0.15cm} & \hspace{-0.15cm} $\phi_{8,3}$ \hspace{-0.15cm} & \hspace{-0.15cm} $\phi_{8,4}$ \hspace{-0.15cm}
\rule[-0.2cm]{0cm}{0.36cm}\\
\hline
\rule{0cm}{0.36cm}
$\chi_{8,1}$ & $1$ & $.$ & $.$ & $.$ & $1$ & $.$ & $.$ & $.$ & $1$ & $.$ & $.$ & $.$ 
\rule[-0.2cm]{0cm}{0.36cm}\\
\rule{0cm}{0.36cm}
$\chi_{8,\frac{q}{\sqrt{2}}(q^2-1)_a}$ & $.$ & $1$ & $.$ & $.$ & $.$ &
$1$ & $.$ & $.$ & $.$ & $1$ & $.$ & $.$ 
\rule[-0.2cm]{0cm}{0.36cm}\\
\rule{0cm}{0.36cm}
$\chi_{8,\frac{q}{\sqrt{2}}(q^2-1)_b}$ & $.$ & $.$ & $1$ & $.$ & $.$ &
$.$ & $1$ & $.$ & $.$ & $.$ & $1$ & $.$ 
\rule[-0.2cm]{0cm}{0.36cm}\\
\rule{0cm}{0.36cm}
$\chi_{8,St}$ & $1$ & $1$ & $1$ & $1$ & $1$ & $.$ & $.$ & $1$ & $.$ & $.$ & $.$ & $1$ 
\rule[-0.2cm]{0cm}{0.36cm}\\
\hline
\end{tabular*} 
\end{center}
\end{table}


\begin{table}[!ht]
\caption[]{The decomposition numbers of the irreducible
characters in $\mathcal{E}(G,s_{10}) = \{\chi_{10,1},  
\chi_{10,\frac{q}{\sqrt{2}}(q^2-1)_a}, \chi_{10,\frac{q}{\sqrt{2}}(q^2-1)_b}, \chi_{10,St}\}$ where
$s_{10}\in~G^*$ is a semisimple $\ell'$-element of type~$g_{10}$.}   
\label{tab:dec10}

\begin{center}
\begin{tabular*}{12.5cm}{
l@{\extracolsep\fill}|cccc|cccc|cccc
}
\hline
\rule{0cm}{0.4cm}
& \multicolumn{4}{l|}{$\ell \, | \, q^2+\sqrt{2}q+1$} & 
  \multicolumn{4}{l|}{$\ell \, | \, q^2-\sqrt{2}q+1$} &
  \multicolumn{4}{l}{otherwise}\\
& \hspace{-0.22cm} $\phi_{10,1}$ \hspace{-0.22cm} & \hspace{-0.22cm}
  $\phi_{10,2}$ \hspace{-0.22cm} & \hspace{-0.22cm}
  $\phi_{10,3}$ \hspace{-0.22cm} & \hspace{-0.22cm}
  $\phi_{10,4}$ \hspace{-0.22cm} 
& \hspace{-0.22cm} $\phi_{10,1}$ \hspace{-0.22cm} & \hspace{-0.22cm}
  $\phi_{10,2}$ \hspace{-0.22cm} & \hspace{-0.22cm}
  $\phi_{10,3}$ \hspace{-0.22cm} & \hspace{-0.22cm}
  $\phi_{10,4}$ \hspace{-0.22cm} 
& \hspace{-0.22cm} $\phi_{10,1}$ \hspace{-0.22cm} & \hspace{-0.22cm}
  $\phi_{10,2}$ \hspace{-0.22cm} & \hspace{-0.22cm}
  $\phi_{10,3}$ \hspace{-0.22cm} & \hspace{-0.22cm}
  $\phi_{10,4}$ \hspace{-0.22cm} 
\rule[-0.2cm]{0cm}{0.36cm}\\
\hline
\rule{0cm}{0.36cm}
$\chi_{10,1}$ & $1$ & $.$ & $.$ & $.$ & $1$ & $.$ & $.$ & $.$ & $1$ & $.$ & $.$ & $.$ 
\rule[-0.2cm]{0cm}{0.36cm}\\
\rule{0cm}{0.36cm}
$\chi_{10,\frac{q}{\sqrt{2}}(q^2-1)_a}$ & $.$ & $1$ & $.$ & $.$ & $.$
& $1$ & $.$ & $.$ & $.$ & $1$ & $.$ & $.$  
\rule[-0.2cm]{0cm}{0.36cm}\\
\rule{0cm}{0.36cm}
$\chi_{10,\frac{q}{\sqrt{2}}(q^2-1)_b}$ & $.$ & $.$ & $1$ & $.$ & $.$
& $.$ & $1$ & $.$ & $.$ & $.$ & $1$ & $.$ 
\rule[-0.2cm]{0cm}{0.36cm}\\
\rule{0cm}{0.36cm}
$\chi_{10,St}$ & $1$ & $1$ & $1$ & $1$ & $1$ & $.$ & $.$ & $1$ & $.$ & $.$ & $.$ & $1$ 
\rule[-0.2cm]{0cm}{0.36cm}\\
\hline
\end{tabular*} 
\end{center}
\end{table}


\begin{table}[!ht]
\caption[]{The decomposition numbers of the irreducible
characters in $\mathcal{E}(G,s_i) = \{\chi_{i,1}\}$ where
$s_i\in~G^*$ is a semisimple $\ell'$-element of type~$g_i$ and 
$i \in \{4,7,9,11,12,13,14,15,16,17,18\}$. These are the basic sets
corresponding to the regular semisimple $\ell'$-elements of $G^*$.}    
\label{tab:decnilp}

\begin{center}
\begin{tabular*}{3cm}{
l@{\extracolsep\fill}|l
}
\hline
\rule{0cm}{0.4cm}
& $\ell$ odd\\
& $\phi_{i,1}$ 
\rule[-0.2cm]{0cm}{0.36cm}\\
\hline
\rule{0cm}{0.36cm}
$\chi_{i,1}$ \hspace{0.4cm} & $1$ \hspace{0.8cm}
\rule[-0.2cm]{0cm}{0.36cm}\\
\hline
\end{tabular*} 
\end{center}
\end{table}

\newpage



\end{document}